\documentclass[12pt]{amsart}
\usepackage[all, import, 2cell]{xy} \UseAllTwocells 
\usepackage{graphicx} 
\usepackage{amsmath, amssymb}
\usepackage{mathrsfs}
\usepackage{stmaryrd}
\usepackage{amsfonts}
\usepackage{amsthm} 
\usepackage{chngcntr} 
\counterwithin{figure}{section}

\numberwithin{equation}{subsection}

\usepackage{hyperref}
\usepackage{xcolor} 

\theoremstyle{definition}
\newtheorem{theorem}{Theorem}[subsection]

\newtheorem{lemma}[theorem]{Lemma}
\newtheorem{proposition}[theorem]{Proposition}
\newtheorem{prop}[theorem]{Proposition}
\newtheorem{corollary}[theorem]{Corollary}

\theoremstyle{definition}
\newtheorem{definition}[theorem]{Definition}
\newtheorem{variant}[theorem]{Variant}
\newtheorem{defn}[theorem]{Definition}
\newtheorem{warning}[theorem]{Warning}
\newtheorem{notation}[theorem]{Notation}
\newtheorem{example}[theorem]{Example}

\newtheorem{remark}[theorem]{Remark}
\newtheorem{construction}[theorem]{Construction}

\newcommand\nc{\newcommand}
\nc{\on}{\operatorname}
\nc{\DMO}{\DeclareMathOperator}

\DMO{\id}{id}
\DMO{\op}{op}
\DMO{\ob}{ob}
\DMO{\flow}{Flow}
\DMO{\crit}{Crit}
\DMO{\indx}{index}
\DMO{\colim}{colim}

\newcommand{\Adjoint}[4]{\xymatrix@1{#2 \ar@<.4ex>[r]^-{#1} & #3 \ar@<.4ex>[l]^-{#4}}}
\DeclareMathOperator{\TopStack}{TopStk}
\DeclareMathOperator{\Conv}{Conv}
\DeclareMathOperator{\calG}{\mathcal{G}}
\DeclareMathOperator{\Ob}{Ob}
\DeclareMathOperator{\Tw}{Tw}
\DeclareMathOperator{\Mor}{Mor}
\DeclareMathOperator{\calB}{\mathcal{B}}

\DeclareMathOperator{\strict}{str}
\DeclareMathOperator{\cont}{cont}
\DeclareMathOperator{\smooth}{sm}
\DeclareMathOperator{\BG}{BG}

\DeclareMathOperator{\im}{im}
\DeclareMathOperator{\LinPreOrd}{\mathsf{PLin}}
\DeclareMathOperator{\LinOrd}{ \mathsf{Lin} }
\DeclareMathOperator{\SSet}{ \mathcal{S} }

\DeclareMathOperator{\Shv}{Shv}
\DeclareMathOperator{\sheafF}{ \mathscr{F}}
\DeclareMathOperator{\sheafG}{ \mathscr{G}}
\DeclareMathOperator{\sheafH}{ \mathscr{H}}
\DeclareMathOperator{\bHom}{Map}
\DeclareMathOperator{\lax}{lax}
\DeclareMathOperator{\Pt}{Pt}
\DeclareMathOperator{\Rep}{Rep}
\DeclareMathOperator{\Hom}{Hom}
\DeclareMathOperator{\BR}{B\mathbf{R}}
\DeclareMathOperator{\Fun}{Fun}
\DeclareMathOperator{\init}{init}
\DeclareMathOperator{\term}{term}
\DeclareMathOperator{\Z}{\mathbf{Z}}
\DeclareMathOperator{\Alg}{Alg}
\DeclareMathOperator{\lex}{lex}
\DeclareMathOperator{\nounit}{nu}
\DeclareMathOperator{\Ind}{Ind}
\DeclareMathOperator{\Amal}{Amalg}
\DeclareMathOperator{\fct}{fact}

\DeclareMathOperator{\calJ}{\mathcal{J}}
\DeclareMathOperator{\wfact}{wfact}
\DeclareMathOperator{\calU}{\mathcal{U}}
\DeclareMathOperator{\PShv}{PShv}
\nc{\Top}{\mathsf{Top}}
\nc{\cbl}{\mathsf{cbl}}
\nc{\cat}{\mathsf{Cat}}
\nc{\Cat}{\mathsf{Cat}}
\nc{\Gpd}{\mathsf{Gpd}}
\nc{\shv}{\mathsf{Shv}}
\nc{\fun}{\mathsf{Fun}}
\nc{\bij}{\mathsf{bij}}
\nc{\flag}{\mathsf{Flag}}
\nc{\mfld}{\mathsf{Mfld}}
\nc{\sing}{\mathsf{Sing}}
\nc{\exit}{\mathsf{Exit}}
\nc{\surj}{\mathsf{surj}}
\nc{\sets}{\mathsf{Sets}}
\DeclareMathOperator{\Set}{Set}
\nc{\vtrs}{\mathsf{Vtrs}}
\nc{\Span}{\mathsf{Span}}
\nc{\wvtrs}{\widetilde{\mathsf{Vtrs}}}
\nc{\groupoid}{\mathsf{Groupoid}}
\nc{\vietoris}{\mathsf{Vietoris}}
\nc{\groupoids}{\mathsf{Groupoids}}
\nc{\cospan}{\mathsf{Cospan}}
\nc{\broken}{\mathsf{Broken}}
\nc{\brTop}{\mathsf{BrTop}}
\nc{\BrokeTop}{\mathsf{BrTop}}
\nc{\strips}{\mathsf{Strips}}
\nc{\spaces}{\mathsf{Spaces}}
\nc{\stacks}{\mathsf{Stacks}}
\nc{\metric}{\mathsf{Metric}}
\nc{\arrows}{\mathsf{Arrows}}
\nc{\preord}{\mathsf{Preord}}
\nc{\ttraj}{\mathbb{T}\mathsf{raj}}
\nc{\cTopc}{{}_C \Top _C}
\nc{\trees}{\mathsf{Trees}}
\nc{\disks}{\mathsf{Disks}}
\nc{\traj}{\mathsf{Traj}}
\nc{\wtraj}{\widetilde{\mathsf{Traj}}}
\nc{\holdisks}{\mathsf{HolDisks}}
\nc{\data}{\mathsf{Data}}
\nc{\Fin}{{\mathscr{F\mkern-3.5mu}}{\operatorname{in}}}
\nc{\Ass}{{\mathscr{A\mkern-3mu}}{\operatorname{ss}}}

\nc{\w}{\widetilde}
\nc{\Rbar}{\overline{\RR}}
\nc{\RRbar}{\overline{\RR}}
\nc{\delbar}{\overline{\partial}}
\nc{\tensor}{\otimes}
\nc{\up}{\uparrow}
\nc{\into}{\hookrightarrow}
\nc{\dd}{\downarrow\downarrow}

\nc{\hiro}{\textcolor{blue}}
\nc{\jacob}{\textcolor{green!75!blue}}

\nc{\eqn}{\begin{equation}}
\nc{\eqnd}{\end{equation}}
\nc{\eqnn}{\begin{equation}\nonumber}
\nc{\enum}{\begin{enumerate}}
\nc{\enumd}{\end{enumerate}}
\nc{\xra}{\xrightarrow}

\DeclareMathOperator{\Tot}{Tot}

\DeclareMathOperator{\calD}{\mathcal{D}}

\def\cC{\mathcal C}\def\cD{\mathcal D}
\def\cE{\mathcal E}\def\cF{\mathcal F}
\def\cJ{\mathcal J}\def\cK{\mathcal K}

\def\RR{\mathbf R}

\begin{document}

\title{Associative algebras and broken lines}
\author{Jacob Lurie and Hiro Lee Tanaka}

\begin{abstract}
Inspired by Morse theory, we introduce a topological stack $\broken$, which we refer to as {\it the moduli stack of broken lines}. We show that $\broken$ can be presented as a Lie groupoid with corners
and provide a combinatorial description of sheaves on $\broken$ with values in any compactly generated $\infty$-category $\mathcal{C}$. Moreover, we show that
{\em factorizable} $\mathcal{C}$-valued sheaves (with respect to a natural semigroup structure on the stack $\broken$) can be identified with nonunital $A_{\infty}$-algebras in $\mathcal{C}$.
This is a first step in a program whose goal is to present an ``equation-free'' construction of the Morse complex associated to a compact Riemannian manifold.
\end{abstract}

\maketitle

\tableofcontents

\section{Introduction}\label{section.overview}

Our starting point in this paper is the following:

\begin{defn}\label{defn.broken-line}
Let $L$ be a topological space equipped with a continuous action $$\mu: \RR \times L \rightarrow L$$ of the group $\RR$ of real numbers.
We will say that the pair $(L, \mu)$ is a {\it broken line} if the following pair of conditions is satisfied:
\enum
\item\label{item.directed} There exists a homeomorphism $\gamma: L \simeq [0,1]$ with the property that, for every point $x \in L$ and every nonnegative real number
$t$, we have $\gamma( \mu(t,x) ) \geq \gamma(x)$.

\item\label{item.finite-fixed} The fixed point set $L^{\RR} = \{ x \in L \,:\, (\forall t \in \RR)[\mu(t,x)=x]\} \subseteq L$ is finite.
\enumd
\end{defn}

\begin{example}\label{example.standard-broken-line}
Let $[ - \infty, \infty ] = \RR \cup \{ \pm \infty \}$ denote the extended real line, equipped with the $\RR$-action given by translation.
Then $[-\infty, \infty]$ is a broken line. 
\end{example}

\begin{figure}[h]
		\[
			\xy
			\xyimport(8,8)(0,0){\includegraphics[width=5in]{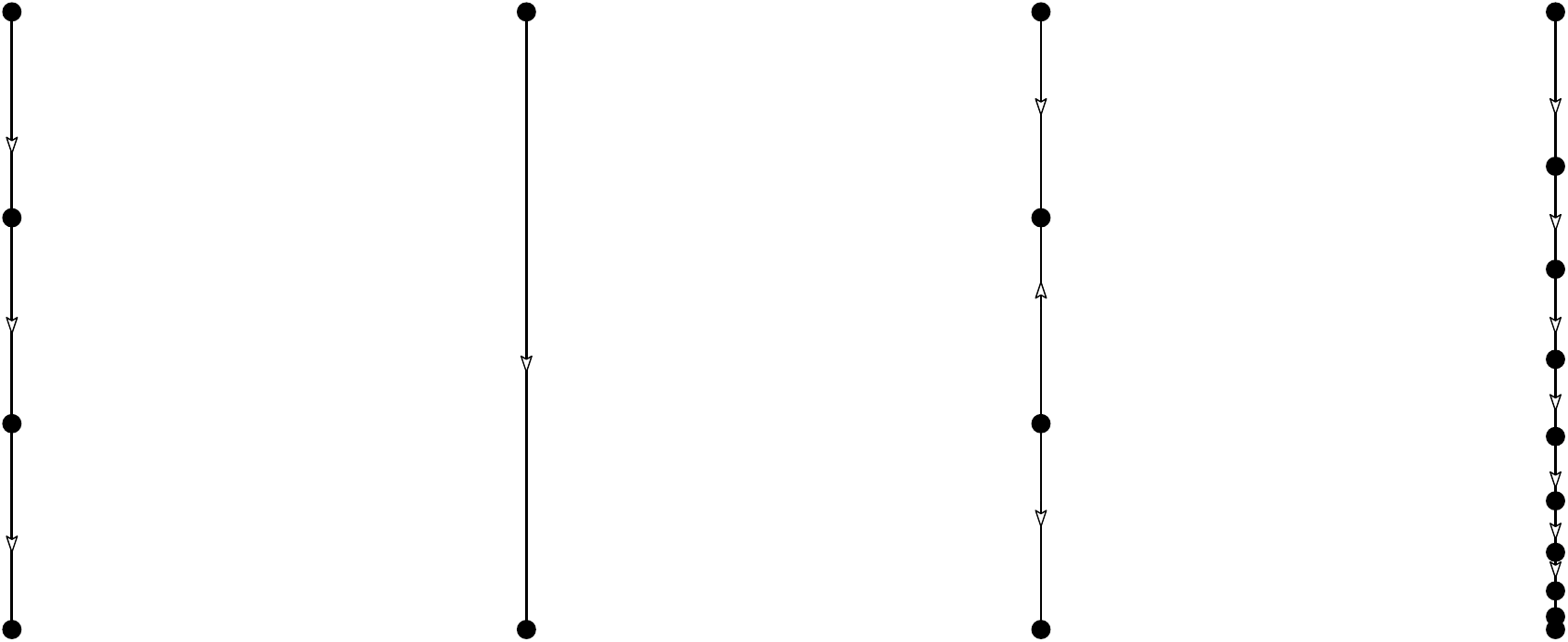}}
			,(0.05,-1)*+{(a)}
			,(2.7,-1)*+{(b)}
			,(5.35,-1)*+{(c)}
			,(7.95,-1)*+{(d)}
			\endxy
		\]
\caption{
Examples and non-examples of broken lines. The white arrowheads indicate the direction of the $\RR$-action, while the black dots indicate the $\RR$-fixed points. Diagrams (a) and (b) depict broken lines, but (c) and (d) do not: (c) violates axiom (\ref{item.directed}) of \ref{defn.broken-line} (the $\RR$-action is not directed), while (d) violates axiom (\ref{item.finite-fixed}) (the collection of fixed points is infinite).
}\label{figure.broken-lines}
\end{figure}

Broken lines are easy to classify. For any broken line $L$, the fixed point set $L^{\RR}$ has at least
two elements: the points $\gamma^{-1}(0)$ and $\gamma^{-1}(1)$, where $\gamma: L \simeq [0,1]$ is as above. We will refer
to $\gamma^{-1}(0)$ as the {\it initial point} of $L$ and $\gamma^{-1}(1)$ as the {\it terminal point} of $L$.
Note that if $L$ and $L'$ are broken lines with terminal point $p \in L$ and initial point $q \in L'$, then we can form
a new broken line from the disjoint union $L \amalg L'$ by identifying the points $p$ and $q$. We will refer
to this broken line as the {\it concatenation of $L$ and $L'$} and denote it by $L \star L'$. It is not difficult
to see that every broken line $L$ can be obtained by concatenating $n$ copies of $[ - \infty, \infty ]$,
for some $n \geq 1$ (see Corollary \ref{corollary.classification}). Our terminology is meant to suggest that the $\RR$-action
``breaks'' $L$ into $n$ pieces: more precisely, the non-fixed locus $L \setminus L^{\RR}$ has $n$ connected components.

In this paper, we study continuous families of broken lines. If $S$ is a topological space, we define
an {\it $S$-family of broken lines} to be a topological space $L_S$ equipped with a projection map
$\pi: L_S \rightarrow S$ and a continuous action of $\RR$ on $L_S$, satisfying certain axioms 
(see Definition \ref{defn.S-family}). These axioms guarantee that, for each point $s \in S$, the fiber
$L_s = \pi^{-1}\{s\}$ is a broken line, so that $L_s \setminus L_s^{\RR}$ has $n_s$ components for
some positive integer $n_s$. However, we do not require the function $s \mapsto n_s$ to be locally constant: 
our definition allows families which ``degenerate'' when specialized to closed subsets of $S$.

\begin{example}\label{easybreak}
Let $S = \RR_{\geq 0}$ be the set of nonnegative real numbers. Set $L_{S}^{\circ} = \RR_{\geq 0} \times \RR_{\geq 0}$,
and define $\pi^{\circ}: L_{S}^{\circ} \rightarrow S$ by the formula $\pi^{\circ}(x,y) = xy$. There is a continuous action
of $\RR$ on $L_{S}^{\circ}$ which preserves each fiber of $\pi^{\circ}$, given by the map
$$ \mu: \RR \times L_{S}^{\circ} \rightarrow L_{S}^{\circ} \quad \quad \mu(t,x,y) = ( e^{t} x, e^{-t} y).$$
The space $L_{S}^{\circ}$ admits a (fiberwise) compactification $L_{S} = L_{S}^{\circ} \cup (S \times \{ \pm \infty \} )$
which is a family of broken lines over $S$ (see Example \ref{easybreak2} for a description). The fiber $L_{s}$ over any point $s > 0$ can
be identified with the standard ``unbroken'' line $[ - \infty, \infty ]$ of Example \ref{example.standard-broken-line},
while the fiber $L_{0}$ can be identified with the concatenation $[ - \infty, \infty ] \star [ -\infty, \infty ]$.
\end{example}

\begin{figure}[h]

\begin{tabular}{lll}
		\raisebox{-0.1\height}
		{\includegraphics[height=1.5in]{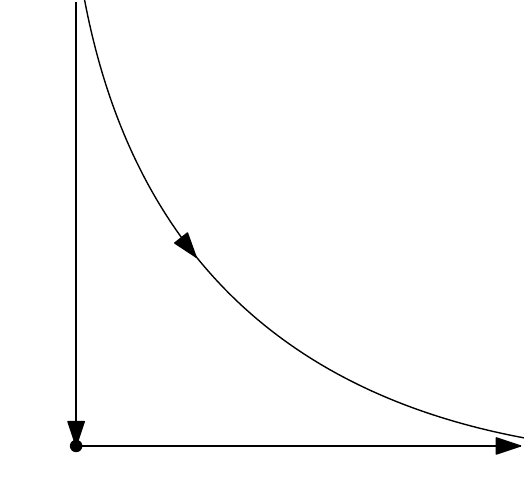}}
		& $\qquad$&
		\raisebox{0\height}{\includegraphics[height=1in]{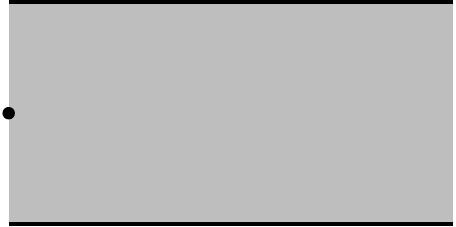}}
\end{tabular}
\caption{
Two depictions of the family of broken lines from Example~\ref{easybreak}. On the left are drawn the domain of $\pi^{\circ}$, along with the fibers $(\pi^{\circ})^{-1}\{0\}$ and $(\pi^{\circ})^{-1}\{t\}$ for $t>0$, whose arrowheads indicate the $\RR$-action. On the right is a drawing of $L_S$, where $S$ (not pictured) is a horizontal interval with a closed endpoint toward the left. Black dots indicate fixed points of the $\RR$-action, and on the right, the solid horizontal lines also indicated $\RR$-fixed points. See also Figure~\ref{figure.non-families}.
}\label{figure.family-xy}
\end{figure}

For every topological space $S$, the collection of $S$-families of broken lines can be organized
into a groupoid which we will denote by $\broken(S)$. The construction $S \mapsto \broken(S)$
is an example of a {\em topological stack}: that is, it is a rule which associates a groupoid
to every topological space, having good descent properties (see \S \ref{section.stacks} for a review of the theory of topological stacks).
We will denote this stack by $\broken$ and refer to it as {\it the moduli stack of broken lines}. Our first main result is that
$\broken$ is a reasonable geometric object:

\begin{theorem}\label{theorem.mainA}
The moduli stack $\broken$ is (representable by) a Lie groupoid with corners (see Definition \ref{definition.liegroupoid}).
\end{theorem}

The primary goal in this paper is to understand the theory of sheaves on the moduli stack $\broken$. Let $\cC$ be any category.
We define a {\it $\cC$-valued sheaf on $\broken$} to be a functor which associates an object $\sheafF( L_S ) \in \cC$ to every $S$-family of broken lines $L_S$,
satisfying a suitable descent condition for open coverings (see \S \ref{section.sheaves-on-stacks} for more details). Heuristically, one can think of
$\sheafF$ as a rule which associates an object $\sheafF(L) \in \cC$ to each broken line $L$, depending ``continuously'' on $L$. Since the
classification of broken lines is relatively simple, this heuristic suggests that sheaves on $\broken$ should admit a purely combinatorial description.
Our second main result asserts that this is indeed the case, at least if the category $\cC$ is well-behaved:

\begin{theorem}\label{theorem.mainB}
Let $\cC$ be a compactly generated category. Then there is a canonical equivalence of categories
$\Shv_{\cC}(\broken) \simeq \Fun( \LinOrd, \cC )$,
where $\LinOrd$ denotes the category whose objects are nonempty finite linearly ordered sets, and whose morphisms are monotone surjections
(see Theorems \ref{theorem.broken-sheaves-vague} and \ref{theorem.broken-sheaves} for more precise statements).
\end{theorem}

The concatenation construction $(L, L') \mapsto L \star L'$ can be extended to families of broken lines, and is classified by a map
$$ m: \broken \times \broken \rightarrow \broken.$$
The multiplication $m$ is associative (up to canonical isomorphism) and exhibits $\broken$ as a semigroup in the setting of topological stacks.
Suppose now that $\cC$ is a (sufficiently nice) monoidal category, with tensor product $\otimes: \cC \times \cC \rightarrow \cC$. In this case, we define
a {\it factorizable $\cC$-valued sheaf on $\broken$} to be a $\cC$-valued sheaf $\sheafF$ on $\broken$, equipped with an isomorphism
$\alpha: m^{\ast} \sheafF \simeq \sheafF \boxtimes \sheafF$ which satisfies the following coherence condition: the diagram
$$ \xymatrix{ (m \times \id)^{\ast} (m^{\ast} \sheafF) \ar[rr]^{\sim} \ar[d]^{ \alpha }& & (\id \times m)^{\ast} (m^{\ast} \sheafF) \ar[d]^{\alpha} \\
(m \times \id)^{\ast}( \sheafF \boxtimes \sheafF) \ar[dr]^{\alpha \boxtimes \id} & & (\id \times m)^{\ast} (\sheafF \boxtimes \sheafF) \ar[dl]_{\id \boxtimes \alpha} \\
& \sheafF \boxtimes \sheafF \boxtimes \sheafF & }$$
commutes, in the category of $\cC$-valued sheaves on $\broken \times \broken \times \broken$ (see Definition \ref{definition.factorizable-sheaf} for a more precise definition). 

Heuristically, if we view a $\cC$-valued sheaf $\sheafF$ as a rule which associates an object $\sheafF(L) \in \cC$ to each broken line $L$, then
a factorizable sheaf consists of such a rule together with a collection of isomorphisms $\sheafF(L \star L') \simeq \sheafF(L) \otimes \sheafF(L')$,
satisfying a suitable associativity condition. Since every broken line $L$ can be obtained by concatenating finitely many copies of the interval
$[ - \infty, \infty ]$, this heuristic suggests that a factorizable sheaf $\sheafF$ should be determined by the single object $\sheafF( [ - \infty, \infty ] ) \in \cC$.
Our third main result asserts that this heuristic is essentially correct, modulo the caveat that we must remember some additional structure
on the object $\sheafF( [ - \infty, \infty] )$:

\begin{theorem}\label{theorem.mainC}
Let $\cC$ be a compactly generated category with a closed monoidal structure and let
$\Shv^{\fct}_{\cC}(\broken)$ denote the category of factorizable $\cC$-valued sheaves on $\broken$.
Then the construction $\sheafF \mapsto \sheafF( [ -\infty, \infty] )$ can be promoted to an equivalence of categories
	\eqnn
	\Shv^{\fct}_{\cC}( \broken) \simeq \Alg^{\nounit}(\cC),
	\eqnd
where $\Alg^{\nounit}(\cC)$ denotes the category of nonunital
associative algebra objects of $\cC$.
\end{theorem}

\begin{remark}
Let $\sheafF$ be a factorizable $\cC$-valued sheaf on $\broken$ and set $A = \sheafF( [ -\infty, \infty] ) \in \cC$.
Set $S = \RR_{\geq 0}$, and let $L_S$ be the $S$-family of broken lines described in Example \ref{easybreak}.
Then $L_S$ is classified by a map $S \rightarrow \broken$, we can regard the restriction $\sheafF|_{S}$ 
as a $\cC$-valued sheaf on $S$. On the open set $\RR_{> 0} \subseteq S$, this sheaf is constant
with the value $A$. At the point $0 \in S$, the stalk $(\sheafF|_{S})_{0}$ is given by
$\sheafF( [ -\infty, \infty] \star [ - \infty, \infty] ) = A \otimes A$. The structure of the sheaf $\sheafF|_{S}$
is then encoded by a cospecialization map $A \otimes A \rightarrow A$. The content of Theorem \ref{theorem.mainC}
is that this map endows $A$ with the structure of a (nonunital) associative algebra, from which
we can recover the entire factorizable sheaf $\sheafF$. 
\end{remark}

We will actually prove stronger versions of Theorems \ref{theorem.mainB} and \ref{theorem.mainC} in this paper, where we allow
$\cC$ to be an $\infty$-category rather than an ordinary category. In this case, the associativity in the definition
of factorizable sheaf and on the resulting algebra object of $\cC$ should be understood in the homotopy coherent sense:
for example, if $\cC$ is the $\infty$-category of chain complexes over some field $k$, then Theorem \ref{theorem.mainC} provides
a geometric model for the theory of (nonunital) $A_{\infty}$-algebras over $k$. 

\begin{remark}[Motivation]
One impetus for the current work is to present a ``coordinate-free'' construction of Morse theory---a formal framework in which the input of a Morse function, a sufficiently generic Riemannian metric, and sufficient tangential structures output a filtered stable homotopy type equivalent to the original manifold (filtered by the Morse function). In later works, we plan to show that the example of Section~\ref{section.morse-example} (combined with the results of this paper) outputs an $A_\infty$-algebra in spectra encoding a deformation problem, and that a choice of fundamental cycle for the moduli of Morse trajectories yields a Maurer-Cartan element for the algebra. The associated solution recovers the filtered stable homotopy type of the original manifold, whose associated graded pieces are the Morse attaching spheres.
\end{remark}

\begin{notation}
In the literature, the term $\infty$-categories is sometimes used model-independently. In this work, by an $\infty$-category we mean a quasi-category as introduced by Boardman-Vogt~\cite{boardman-vogt} and later developed by Joyal~\cite{joyal}: a simplicial set satisfying the weak Kan condition. In our notation, we will not distinguish between a category and its nerve; for example, if $\cC$ is a category and $\cD$ is an $\infty$-category, we will write $\fun(\cC,\cD)$ rather than $\fun(N(\cC),\cD)$ to mean the $\infty$-category of functors from the nerve of $\cC$ to $\cD$.
\end{notation}

\subsection*{Acknowledgments}

During the period of time in which this work was carried out, the first author was supported by the National Science Foundation under grant number 1510417,
and the second author was also supported by the National Science Foundation under award number DMS-1400761.

\clearpage
\section{Families of Broken Lines}

\subsection{Definitions}

Our first goal in is to introduce the notion of a {\em family of broken lines} parametrized by a topological space $S$ (Definition \ref{defn.S-family}),
which specializes to Definition \ref{defn.broken-line} in the case where $S$ is a point. We begin with some general remarks.

\begin{defn}\label{definition.directed}
Let $L$ be a topological space equipped with a continuous action $\mu : \RR \times L \to L$, and let $[a,b]$ be a closed interval for real numbers $a< b$.
We will say that a homeomorphism $\gamma : L \cong [a, b]$ is {\em directed} if, for every point $x \in L$, the induced map
\begin{equation*}
\RR \cong \RR\times\{x\} \to\RR\times L \xra{\mu} L \xra{\gamma} [a,b]
\end{equation*}
is nondecreasing.
\end{defn}

\begin{notation}[The Ordering of a Broken Line]\label{notation.ordering}
Let $L$ be a broken line (in the sense of Definition \ref{defn.broken-line}), and let $\gamma: L \simeq [0,1]$
be a directed homeomorphism. For every pair of points $x,y \in L$, we write $x \leq_{L} y$ if $\gamma(x) \leq \gamma(y)$.
Note that this condition is independent of the choice of $\gamma$. The relation $\leq_{L}$ is a linear ordering of $L$.
Moreover, the ordered set $( L, \leq_{L} )$ has a least element $\gamma^{-1}(0)$ (the initial point of $L$)
and a greatest element $\gamma^{-1}(1)$ (the terminal point of $L$). Note that the initial and terminal points are automatically fixed by the action of $\RR$ on $L$.
\end{notation}

\begin{defn}\label{defn.S-family}
Let $S$ be a topological space. An {\it $S$-family of broken lines} is a triple $( L_S, \pi, \mu)$, where
$L_S$ is a topological space, $\pi: L_S \rightarrow S$ is a continuous map, and $\mu: \RR \times L_S \rightarrow L_S$
is a continuous action of $\RR$ on $L_S$ which preserves each fiber of $\pi$ and satisfies the following additional conditions:

\begin{itemize}
\item[$(a)$] For every point $s \in S$, there exists an open set $U \subseteq S$ containing $s$
and a continuous map $f: L_S \times_{S} U \rightarrow [0,1]$ with the following properties:
\begin{itemize}
\item The induced map $L_S \times_{S} U \cong [0,1] \times U$ is a homeomorphism.
\item For each $s' \in U$, the restriction $f|_{L_{s'} }: L_{s'} \rightarrow [0,1]$ is a directed
homeomorphism (in the sense of Definition \ref{definition.directed}). Here $L_{s'} = L_{S} \times_{S} \{s'\}$ denotes
the fiber $\pi^{-1} \{s' \}$.
\end{itemize}

\item[$(b)$] Let $L_S^{\RR}$ denote the set of fixed points for the action of
$\RR$ on $L_S$. Then the restriction $\pi|_{ L_S^{\RR} }: L_{S}^{\RR} \rightarrow S$ is unramified.
More precisely, for every point $s \in S$ there exists an open set $U \subseteq S$ containing $s$ and a decomposition
$$\pi^{-1}(U) \cap L_{S}^{\RR} \cong K_1 \amalg K_2 \amalg \cdots \amalg K_n,$$
where each of the induced maps $\pi|_{K_i}: K_i \rightarrow U$ is a closed embedding of topological spaces. 
\end{itemize}
\end{defn}

\begin{remark}
In the situation of Definition \ref{defn.S-family}, we will often abuse terminology by referring to the
topological space $L_S$ or the projection map $\pi: L_{S} \rightarrow S$ as an $S$-family of broken lines (in this case,
we implicitly assume that an action of $\RR$ on $L_{S}$ has also been specified). 
\end{remark}

\begin{figure}
		\[
			\xy
			\xyimport(8,8)(0,0){\includegraphics[width=5in]{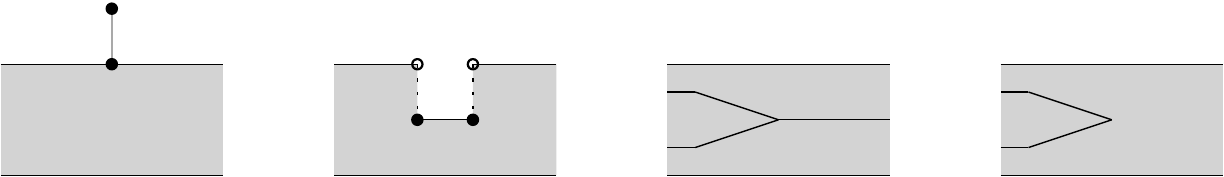}}
			,(0.7,-1.7)*+{(a)}
			,(2.9,-1.7)*+{(b)}
			,(5.1,-1.7)*+{(c)}
			,(7.3,-1.7)*+{(d)}
			\endxy
		\]
\caption{
Non-examples of families of broken lines over $S = [0,1]$. The projection map to $S$, in each figure, is the projection to a horizontal line. Also in each figure, a black line segment or a black dot indicates points fixed by the $\RR$ action; a circle or a dotted line indicates loci which are not included in $\tilde S$. The $\RR$-actions in each picture are by flowing downward. Example $(a)$ fails to satisfy axiom $(A5)$,
Example $(b)$ fails to satisfy $(A3)$ and $(A4)$, and examples $(c)$ and $(d)$ fail to satisfy $(A2)$.
}\label{figure.non-families}
\end{figure}

\begin{remark}[Pullbacks]\label{remark.build-pullback}
Let $S$ be a topological space and let $\pi: L_{S} \rightarrow S$ be an $S$-family of broken lines. For any
continuous map of topological spaces $S' \rightarrow S$, we can endow the pullback $L_{S'} = L_{S} \times_{S} S'$
with the structure of an $S'$-family of broken lines (via projection onto the second factor, with $\RR$-action inherited from the $\RR$-action on $L_{S}$).
In the special case where $S' = \{s\}$ is a single point, we will denote the pullback $L_{\{s\}}$ simply by $L_{s}$. 
\end{remark}

\begin{example}
Let $S$ be a point. Then the data of an $S$-family of broken lines $( L_{S}, \pi, \mu)$ (in the sense of
Definition \ref{defn.S-family}) is equivalent to the data of a broken line $(L_S, \mu)$ (in the sense of Definition \ref{defn.broken-line}). Condition $(b)$ of Definition \ref{defn.S-family} is equivalent to the requirement that the fixed point locus $L_{S}^{\RR}$ is finite. 

More generally, if $S$ is an arbitrary topological space and $L_{S}$ is an $S$-family of broken lines (in the sense of Definition \ref{defn.S-family}), then
for each point $s \in S$, the fiber $L_{s}$ is a broken line (in the sense of Definition~\ref{defn.broken-line}).
\end{example}

\begin{remark}[Semicontinuity]\label{remark.semicontinuity}
Condition $(b)$ of Definition \ref{defn.S-family} implies that the function
$(s \in S) \mapsto (| L_{s}^{\RR} | \in \Z)$ is upper semicontinuous. In particular, the collection of points $s \in S$ for which there exists
an isomorphism of broken lines $L_{s} \simeq [ - \infty, \infty ]$ is an open subset of $S$.
\end{remark}

For further examples of broken lines (and a complete classification of their local behavior), we refer the reader to \S \ref{section.stack}.

\begin{remark}[Locality]\label{remark.locality}
Let $\pi: L_S \rightarrow S$ be a map of topological spaces and let $\mu: \RR \times L_{S} \rightarrow \tilde S$ be a continuous action of $\RR$ on $L_{S}$ which preserves each fiber of $\pi$. Then the condition that $(L_{S}, \pi, \mu)$ is an $S$-family of broken lines can
be tested locally on $S$. More precisely, if there exists an open covering $\{ U_{\alpha} \}$ of $S$ for which
each $( L_{U_{\alpha}}, \pi|_{ L_{U_{\alpha}} }, \mu|_{ \RR \times L_{U_{\alpha}}} )$ is a
$U_{\alpha}$-family of broken lines, then $(L_{S}, \pi, \mu)$ is an $S$-family of broken lines.
\end{remark}

\begin{remark}
Let $S$ be a topological space and let $\pi: L_{S} \rightarrow S$ be an $S$-family of broken lines.
If $S$ is a paracompact Hausdorff space, then $L_{S}$ satisfies the following {\it a priori} stronger version of
condition $(a)$ of Definition \ref{defn.S-family}:
\begin{itemize}
	\item[$(a')$] There exists a continuous map $f: L_{S} \rightarrow [0,1]$ for which the product map
$(f \times \pi): L_{S} \rightarrow [0,1] \times S$ is a homeomorphism and, for each
$s \in S$, the induced homeomorphism $f|_{ L_{s} }: L_{s} \rightarrow [0,1]$ is directed.
\end{itemize}
To prove this, choose a cover of $S$ by open subsets $\{ U_{\alpha} \}$ and continuous maps
$f_{\alpha}: L_{U_{\alpha}} \rightarrow [0,1]$ satisfying condition $(a')$ for each $L_{ U_{\alpha} }$. 
The assumption that $S$ is a paracompact Hausdorff space guarantees that there exists a partition of unity
$\{ \psi_{\alpha} \}$ subordinate to the open cover $\{ U_{\alpha} \}$. It is then easy to check that the formula
$$ f(x) = \sum_{\alpha} \psi_{\alpha}( \pi(x) ) f_{\alpha}(x)$$
determines a function $f: L_S \rightarrow [0,1]$ satisfying $(a')$ (where we adopt the convention that
the product $ \psi_{\alpha}( \pi(x) ) f_{\alpha}(x)$ is equal to zero when $\pi(x) \notin U_{\alpha}$). 
\end{remark}

\subsection{Example: Broken Gradient Trajectories}\label{section.morse-example}

Let $M$ be a compact, smooth Riemannian manifold, and let $h: M \to \RR$ be a smooth function.
Then the manifold $M$ can be equipped with a (smooth) action
$$\mu: \RR \times M \rightarrow M,$$
which is the unique solution to the differential equation $$\frac{ \partial \mu(t,x) }{ \partial t } = (\nabla h)( \mu(t,x) ).$$
We will refer to the action $\mu$ as the {\it gradient flow} associated to $h$. It enjoys the following properties:
\begin{itemize}
\item The fixed point set $M^{\RR}$ can be identified with the set $\crit(M) \subseteq M$ of
{\it critical points} for the function $h$: that is, the set of points $x$ for which $(\nabla h)(x) = 0$.

\item For every point $x \in X$, the function $t \mapsto h( \mu(t,x) )$ is either constant (if $x$ is a critical point of $h$)
or strictly increasing (if $x$ is not a critical point of $h$).
\end{itemize}

\begin{definition}\label{definition.traj-version1}
Let $x$ and $y$ be critical points of $M$ satisfying $h(x) < h(y)$. A {\it broken gradient trajectory} from $x$ to $y$
is a continuous path $p: [ h(x), h(y) ] \to M$ satisfying the following conditions:
\begin{itemize}
\item[$(a)$] The path $p$ satisfies $p( h(x) ) = x$ and $p(h(y)) = y$.
\item[$(b)$] For $t \in [ h(x), h(y) ]$, we have $h( p(t) ) = t$.
\item[$(c)$] The image $\im(p) = \{ p(t): h(x) \leq t \leq h(y) \} \subseteq M$ is invariant under the gradient flow associated to $h$.
\end{itemize}
We let $\traj_{x,y}$ denote the set of all broken gradient trajectories from $x$ to $y$. We will regard
$\traj_{x,y}$ as a topological space by endowing it with the compact-open topology.
\end{definition}

Under mild hypotheses, the topological space $\traj_{x,y}$ supports a family of broken lines, in the sense of Definition \ref{defn.S-family}.

\begin{notation}
Let $x$ and $y$ be critical points of $M$ satisfying $h(x) < h(y)$. We let
$\wtraj_{x,y}$ denote the subset of $\traj_{x,y} \times M$ consisting of those pairs $(p, z)$ where $z$ belongs to the image of $p$.
It follows from condition $(c)$ of Definition \ref{definition.traj-version1} that the gradient flow $\mu: \RR \times M \rightarrow M$
restricts to an action of $\RR$ on $\wtraj_{x,y}$, which preserves each fiber of the projection map
$\wtraj_{x,y} \rightarrow \traj_{x,y}$.
\end{notation}

\begin{proposition}
Let $M$ be a compact smooth Riemannian manifold equipped with a smooth function
$h: M \to \RR$ for which the critical set $\crit(M)$ is finite (this condition is satisfied, for example, if $h$ is a Morse function).
Then the projection map $\pi: \wtraj_{x,y} \rightarrow \traj_{x,y}$ exhibits $\wtraj_{x,y}$ as a $\traj_{x,y}$-family of broken lines, in the sense of Definition \ref{defn.S-family}.
\end{proposition}

\begin{proof}
Let  $e: \traj_{x,y} \times [ h(x), h(y) ] \rightarrow M$ be the evaluation map given by $e(p,t) = p(t)$, and let
$\pi: \traj_{x,y} \times [h(x), h(y) ] \rightarrow \traj_{x,y}$ be the projection onto the first factor. Unwinding the definitions, we see that
$\wtraj_{x,y}$ is the image of the product map $$(\pi \times e): \traj_{x,y} \times [h(x), h(y) ] \rightarrow \traj_{x,y} \times M.$$
This map is closed (by the compactness of $[ h(x), h(y) ]$) and injective, hence determines a homeomorphism
$\traj_{x,y} \times [ h(x), h(y) ] \simeq \wtraj_{x,y}$ which is directed on each fiber. Consequently, to show that $\wtraj_{x,y}$ is a $\traj_{x,y}$-family of broken lines,
it will suffice to show that the projection map $\wtraj_{x,y}^{\RR} \rightarrow \traj_{x,y}$ is unramified. To prove this, we note that
$\wtraj_{x,y}^{\RR}$ can be identified with a closed subset of the product $\traj_{x,y} \times M^{\RR}$, where
$M^{\RR} = \crit(M)$ is the set of critical points of $h$ (and is therefore a finite set equipped with the discrete topology).
\end{proof}

\subsection{The Moduli Stack of Broken Lines}\label{section.stacks}

In this section, we introduce the main object of interest in this paper: the moduli stack $\broken$ of broken lines. We begin by reviewing some terminology.

\begin{definition}\label{definition.fibration}
Let $\rho: \cE \rightarrow \cC$ be a functor between categories. We say that $\rho$ is a {\it fibration in groupoids} if the following conditions are satisfied:
\begin{itemize}
\item[$(a)$] For every object $E \in \cE$ and every morphism $f: E' \rightarrow E''$ in $\cE$, the diagram of sets
$$ \xymatrix{ \Hom_{\cE}( E, E' ) \ar[r]^-{f \circ} \ar[d] & \Hom_{\cE}( E, E'') \ar[d] \\
\Hom_{\cC}( \rho(E), \rho(E') ) \ar[r]^-{ \rho(f) \circ} & \Hom_{\cC}( \rho(E), \rho(E'') ) }$$
is a pullback square.

\item[$(b)$] For every object $E \in \cE$ and every morphism $f: C \rightarrow \rho(E)$ in the category $\cC$, there
exists a morphism $f: \overline{C} \rightarrow E$ in $\cE$ with $\rho(f) = f_0$.
\end{itemize}
\end{definition}

\begin{remark}\label{remark.equivast}
Let $\rho: \cE \rightarrow \cC$ be a fibration in groupoids. Then, for each object $C \in \cC$, the fiber $\cE_{C} = \rho^{-1} \{C \}$ is a groupoid---that is, every
morphism in $\cE_{C}$ is invertible. Moreover, if $f: C' \rightarrow C$ is a morphism in the category $\cC$, then for every object $E \in \cE_{C}$
we can apply condition $(b)$ of Definition \ref{definition.fibration} to choose a morphism $\overline{f}: f^{\ast} E \rightarrow E$ in the category
$\cE$ with $\rho( \overline{f} ) = f$. It follows from condition $(a)$ that the object $f^{\ast} E$ is well-defined up to (unique) isomorphism and that the construction
$E \mapsto f^{\ast} E$ determines a functor $f^{\ast}: \cE_{C} \rightarrow \cE_{C'}$. If we regard the category $\cC$ as fixed, this construction establishes a dictionary
$$ \xymatrix{ \{ \text{Fibrations in Groupoids $\rho: \cE \to \cC$} \} \ar[d] \\ \{ \text{Functors of $2$-categories $\cC^{\op} \to \{ \text{Groupoids} \}$} \}. }$$
\end{remark}

\begin{definition}\label{definition.topological-stack}
Let $\Top$ denote the category of topological spaces. A {\it topological prestack $X$} is a category $\Pt(X)$ equipped with a fibration in groupoids
$\rho: \Pt(X) \rightarrow \Top$. Note that if $X$ is a topological stack and $S$ is a topological space, then the construction
$$ (U \subseteq S) \mapsto \rho^{-1}(U)$$
determines a presheaf on $S$ with values in the $2$-category of groupoids. We will say that $\Pt(X)$ is a {\it topological stack}
if this presheaf is a sheaf for every topological space $S$ (in other words, if the construction $S \mapsto \rho^{-1}(S)$ satisfies
descent with respect to the Grothendieck topology on $\Top$ generated by the open covers).
\end{definition}

\begin{remark}
The terminology of Definition \ref{definition.topological-stack} is potentially confusing, since a topological stack $X$ ``is'' the category $\Pt(X)$ (and the functor $\rho: \Pt(X) \rightarrow \Top$).
However, our notation is intended to emphasize the idea that a topological stack $X$ should be viewed as a geometric object of some kind, from which one can extract a category
$\Pt(X)$ (whose objects can be viewed as topological spaces $S$ equipped with a map $S \rightarrow X$). We refer to $\Pt(X)$ as the {\it category of points} of $X$.
\end{remark}

\begin{remark}[The $2$-Category of Topological Stacks]\label{remark.topstack}
Let $X$ and $Y$ be topological stacks. A {\it morphism of topological stacks} from $X$ to $Y$ is a functor
$f: \Pt(X) \rightarrow \Pt(Y)$ for which the diagram
$$ \xymatrix{ \Pt(X) \ar[rr]^{f} \ar[dr]^{\rho_X} & & \Pt(Y) \ar[dl]_{\rho_Y} \\
& \Top & }$$
commutes. The collection of all morphisms from $X$ to $Y$ can be organized into a category, where
an isomorphism from $f: \Pt(X) \rightarrow \Pt(Y)$ to $g: \Pt(X) \rightarrow \Pt(Y)$ is a natural transformation
$\alpha: f \rightarrow g$ for which the induced natural transformation 
$$ \rho_{X} = \rho_{Y} \circ f \xrightarrow{\alpha} \rho_{Y} \circ g = \rho_X$$
is the identity (that is, for every object $\widetilde{S} \in \Pt(X)$, the map $\alpha( \widetilde{S}): f( \widetilde{S}) \rightarrow g( \widetilde{S})$ 
induces the identity on the underlying topological space $S$). Note that such a natural transformation $\alpha$ is
automatically invertible (since $\rho_Y$ is a fibration in groupoids).

We let $\TopStack$ denote the (strict) $2$-category whose objects are topological stacks, and whose morphisms (and $2$-morphisms)
are defined as above. Note that every $2$-morphism in $\TopStack$ is invertible.
\end{remark}

\begin{example}[Topological Spaces as Topological Stacks]\label{example.space-as-stack}
Let $X$ be a topological space, and let $\Pt(X)$ denote the category $\Top_{ / X }$, whose objects are topological spaces $S$ equipped with a map $S \rightarrow X$.
Then the forgetful functor $\rho: \Pt(X) \rightarrow \Top$ is a topological stack, which (by a slight abuse of terminology) we will identify with the original topological space $X$.
The construction $X \mapsto \Pt(X)$ determines a fully faithful embedding from the category of topological spaces to the $2$-category of topological stacks. 

Further, let $\rho: \Pt(Y) \to \Top$ be a topological stack. Then by the Yoneda Lemma, there is a canonical equivalence between the category of morphisms $X \to Y$ and the category $\rho^{-1}(X)$. 
\end{example}

\begin{example}[The Classifying Stack of a Group]\label{example.class-stack}
Let $G$ be a topological group. We let $\Pt(\BG)$ denote the category whose objects are maps $P \to S$, where $S$ is a topological space and $P$ is a principal $G$-bundle over
$S$; a morphism from $(P \to S)$ to $(P' \to S')$ in the category $\Pt(\BG)$ is given by a commutative diagram of topological spaces
$$ \xymatrix{ P \ar[r] \ar[d] & P' \ar[d] \\
S \ar[r] & S' }$$
where the upper horizontal map is $G$-equivariant. Then the construction $(P \to S) \mapsto S$ determines a topological stack $\Pt(\BG) \rightarrow \Top$. We will denote this
topological stack by $\BG$ and refer to it as the {\it classifying stack of the group $G$}.
\end{example}

\begin{notation}
Let $X$ be a topological stack. We will typically denote objects of the category $\Pt(X)$ by $\widetilde{S}$, and write $S$ for the image of $\widetilde{S}$ under the forgetful functor
$\rho: \Pt(X) \rightarrow \Top$. If $f: U \rightarrow S$ is a continuous map of topological spaces, we let $\widetilde{S}|_{U}$ denote the domain of a morphism
$\widetilde{f}: \widetilde{S}|_{U} \rightarrow \widetilde{S}$ in the category $\Pt(X)$ (our assumption that $\rho$ is fibered in groupoids guarantees that
the object $\widetilde{S}|_{U}$ is determined uniquely up to (unique) isomorphism). Heuristically, we view the object $\widetilde{S} \in \Pt(X)$ as encoding the
datum of a map from $S$ into the topological stack $X$, in which case $\widetilde{S}|_{U}$ encodes the composite map $U \xrightarrow{f} S \rightarrow X$.
\end{notation}

We now introduce the main object of interest in this paper.

\begin{construction}\label{construction.category-of-lines}
We define a category $\Pt(\broken)$ as follows:
\begin{itemize}
\item An object of $\Pt(\broken)$ is a topological space $S$ together with an $S$-family of broken lines $(\pi: L_S \to S, \mu)$ (see Definition \ref{defn.S-family}).

\item A morphism from $(\pi': L_{S'} \to S', \mu')$ to $(\pi: L_S \to S, \mu)$ in
the category $\Pt(\broken)$ is a pair of continuous maps $f: S' \to S$, $\widetilde{f}: L_{S'} \to L_{S}$,
where $\widetilde{f}$ is $\RR$-equivariant and the diagram
$$ \xymatrix{ L_{S'} \ar[r]^{ \widetilde{f} } \ar[d]^{\pi'} & L_{S} \ar[d]^{ \pi } \\
S' \ar[r]^{f} & S }$$
is a pullback square in the category of topological spaces.
\end{itemize}
\end{construction}

\begin{remark}
In what follows, we will generally abuse notation by identifying an object $(\pi: L_{S} \rightarrow S, \mu)$
of the category $\Pt(\broken)$ with the topological space $L_{S}$, and a morphism
$$( \widetilde{f}, f): (\pi': L_{S'} \rightarrow S', \mu') \rightarrow (\pi: L_{S} \rightarrow S, \mu)$$ with
the underlying map of topological spaces $\widetilde{f}: L_{S'} \rightarrow L_{S}$ (note that the map $\widetilde{f}$ determines
$f$). 
\end{remark}

Note that the construction $(\pi: L_{S} \rightarrow S, \mu) \mapsto S$ determines a functor $\rho$ from the category
$\Pt(\broken)$ to the category $\Top$ of topological spaces. We will refer to $\rho$ as {\it the forgetful functor}.

\begin{proposition}\label{proposition.cartesian-fibration}
The forgetful functor $\rho: \Pt(\broken) \rightarrow \Top$ is a topological stack, in the sense of Definition \ref{definition.topological-stack}.
\end{proposition}

\begin{proof}
Condition $(a)$ of Definition \ref{definition.fibration} is immediate from the definitions, and condition $(b)$ follows from Remark \ref{remark.build-pullback}.
It follows that $\rho$ is a fibration in groupoids, so that we can regard the construction $S \mapsto \Pt(\broken) \times_{ \Top} \{ S \}$ as a
groupoid-valued functor on the category $\Top^{\op}$. This functor satisfies descent by virtue of Remark \ref{remark.locality}.
\end{proof}

\begin{notation}\label{not-broke}
We will denote the topological stack of Proposition \ref{proposition.cartesian-fibration} by $\broken$, and refer
to it as the {\it moduli stack of broken lines}. For every topological space $S$, we let $\broken(S)$ denote the fiber product $\Pt(\broken) \times_{ \Top } \{S \}$.
This is the groupoid whose objects are $S$-families of broken lines, and whose morphisms are $\RR$-equivariant homeomorphisms which are compatible
with the projection to $S$. It follows from Remark \ref{remark.equivast} that we can regard the construction $S \mapsto \broken(S)$ as a contravariant functor
from the category $\Top^{\op}$ to the $2$-category of groupoids.
\end{notation}

\begin{example}[Points of $\broken$]\label{example.points-of-broken}
The groupoid $\broken(\ast)$ can be described as follows:
\begin{itemize}
\item An object of $\broken(\ast)$ is a broken line $L$ (Definition \ref{defn.broken-line}).
\item A morphism from $L$ to $L'$ in $\broken(\ast)$ is an $\RR$-equivariant homeomorphism $L \simeq L'$.
\end{itemize}
The structure of this category is easy to describe. Up to isomorphism, it has one object $L_{n}$ for every positive integer $n$
(given by a concatenation of $n$ copies of the broken line $[ - \infty, \infty ]$ of Example \ref{example.standard-broken-line}),
whose automorphism group is isomorphic to $\RR^{n}$ (which acts by translation separately on each component of $L_{n}^{\circ} = L_{n} \setminus L_{n}^{\RR}$). 
\end{example}

\begin{remark}[The Moduli Stack of Unbroken Lines]\label{remark.unbroken}
Let $\Pt(\broken^{\circ})$ denote the full subcategory of $\Pt(\broken)$ spanned by those families of broken lines $L_{S} \to S$ where each fiber $L_{s}$ is isomorphic to
the broken line $[ - \infty, \infty ]$ of Example \ref{example.standard-broken-line}. Then the forgetful functor $\Pt(\broken^{\circ}) \rightarrow \Top$ determines a topological
stack which we denote by $\broken^{\circ}$. Note that if $L_{S}$ is an $S$-family of broken lines satisfying this condition,
then the open subset $L_{S} \setminus L_{S}^{\RR}$ is a principal $\RR$-bundle over $S$. It is not difficult to see that the construction $L_S \mapsto L_S \setminus L_{S}^{\RR}$
induces an equivalence of topological stacks $\broken^{\circ} \simeq \BR$, where $\BR$ denotes the classifying stack of the additive group of real numbers (Example \ref{example.class-stack}).
\end{remark}

\begin{remark}
The moduli stack $\broken^{\circ}$ of Remark \ref{remark.unbroken} can be regarded as an open substack of the moduli stack of broken lines. More generally,
for any integer $n > 0$, there is a locally closed substack $\broken^{=n} \subseteq \broken$, equivalent to the classifying stack $\BR^{n}$, which parametrized $S$-families of broken lines $L_{S} \to S$
having the property that each of the broken lines $L_{s}$ is isomorphic to a concatenation of $n$ copies of $[- \infty, \infty ]$ (that is, the non-fixed locus $L_{s} \setminus L_{s}^{\RR}$
has exactly $n$ connected components). The substacks $\broken^{=n} \subseteq \broken$ determine a stratification of the topological stack $\broken$, given by the construction
which assigns to each $L_{S} \in \broken(S)$ the upper semicontinuous function $(s \in S) \mapsto | \pi_0( L_{s} \setminus L_{s}^{\RR}) |$ (see Remark \ref{remark.semicontinuity}).
\end{remark}

\subsection{A Recognition Principle for Broken Lines}

Condition $(a)$ of Definition \ref{defn.S-family} can be somewhat inconvenient to work with, because
a broken line $L$ does not come equipped with any {\em canonical} choice of directed homeomorphism $L \cong [0,1]$. 
It will therefore be useful to have a more intrinsic characterization of $S$-families of broken lines.

\begin{theorem}\label{theorem.line-recognition}
Let $\pi: L_{S} \rightarrow S$ be a continuous map of topological spaces and let $\mu: \RR \times L_S \rightarrow L_S$ 
be a continuous $\RR$-action which stabilizes each fiber of $\pi$. Then $( \pi: L_S \to S, \mu)$ is an $S$-family of broken
lines (in the sense of Definition \ref{defn.S-family}) if and only if the following axioms are satisfied:
\begin{itemize}
	\item[$(A1)$] For each point $s \in S$, there exists a directed homeomorphism
	$L_{s} \cong [0,1]$.
	
	\item[$(A2)$] The restriction $\pi|_{ L_{S}^{\RR} }: L_{S}^{\RR} \rightarrow S$ is unramified
	(as in Definition \ref{defn.S-family}). 
	
	\item[$(A3)$] The map $\pi$ is closed. 
	
	\item[$(A4)$] For every pair of points $x,y \in L_{S}$ having the same image $s = \pi(x) = \pi(y)$ in $S$, let
	us write $x \leq_{L_{S} } y$ if $x \leq_{ L_{s} } y$ (see Notation \ref{notation.ordering}).
	Then $\{ (x,y) \in L_{S} \times_{S} L_{S}: x \leq_{ L_S } y \}$ is a closed subset of the fiber product
	$L_{S} \times_{S} L_{S}$.
	
	\item[$(A5)$] The restriction of $\pi$ to the non-fixed locus $L_{S}^{\circ} = L_{S} / L_{S}^{\RR}$ has the local lifting property. More precisely,
	$x$ is any point in $L_{S}^{\circ}$, then there exists a neighborhood $U \subseteq S$ of $s = \pi(x)$ and a continuous map $\sigma: U \to L_{S}^{\circ}$ satisfying
	$\sigma(s) = x$ and $\pi \circ \sigma = \id_{U}$.
\end{itemize}
\end{theorem}

\begin{proof}[Proof of Necessity]
Let $L_S$ be an $S$-family of broken lines; we will show that conditions $(A1)$ through $(A5)$ are satisfied (the proof of the converse will require some preliminaries, and
will be given in \S \ref{section.line-recognition}). Working locally on $S$, we can assume that there exists a continuous map $\gamma: L_{S} \rightarrow [0,1]$ 
which induces a homeomorphism $L_{S} \simeq [0,1] \times S$ which restricts to a directed homeomorphism $L_{s} \simeq [0,1]$ for each $s \in S$. 
Assertions $(A1)$ and $(A2)$ are immediate, assertion $(A3)$ follows from the compactness of the interval $[0,1]$, and assertion $(A4)$ follows from the identification
$$ \{ (x,y) \in L_{S} \times_{S} L_{S}: x \leq_{L_S} y \} = \{ (x,y) \in L_{S} \times_{S} L_{S}: \gamma(y) - \gamma(x) \leq 0 \}.$$
To prove $(A5)$, we note that for any point $x \in L_{S}$ having image $s = \pi(x)$ in $S$, there is a unique section $\sigma: S \rightarrow L_{S}$ of the projection map $\pi$
such that $\sigma(s) = x$ and $\gamma \circ \sigma: S \rightarrow [0,1]$ is constant. If $x$ does not belong to the fixed locus $L_{S}^{\RR}$, then
there is an open neighborhood $U \subseteq S$ containing $s$ such that $\sigma|_{U}$ factors through $L_{S}^{\circ}$.
\end{proof}

\begin{remark}\label{remark.properness}
Let $S$ be a topological space and let $(L_{S}, \pi, \mu)$ satisfy axioms $(A1)$ through $(A5)$ of
Theorem \ref{theorem.line-recognition}. It follows from $(A4)$ that the relative diagonal
$L_S \rightarrow L_{S} \times_{S} L_{S}$ is a closed embedding, and from $(A1)$ that each fiber of $\pi$ is compact. Consequently, axiom
$(A3)$ is equivalent to the requirement that the map $\pi: L_{S} \rightarrow S$ is proper.
\end{remark}

We now give a convenient reformulation of condition $(A4)$ of Theorem \ref{theorem.line-recognition}.

\begin{notation}[Translation Distance]\label{notation.translation-distance}
Let $L$ be a broken line equipped with $\RR$-action $\mu: \RR \times L \rightarrow L$,
and let $\leq_{L}$ be the ordering of Notation \ref{notation.ordering}. 
For points $x \in L \setminus L^{\RR}$ and $y \in L$, we define
$$ d_L(x,y) = \begin{cases} \infty & \text{ if } \mu(t,x) \leq_{L} y \text{ for all $t \in \RR$ } \\
t & \text{ if } \mu(t,x) = y \text{ for some $t \in \RR$ } \\
-\infty & \text{ if } y \leq_{L} \mu(t,x) \text{ for all $t \in \RR$}. \end{cases}$$
We will refer to $d_L(x,y) \in [ - \infty, \infty ]$ as the {\it translation distance from $x$ to $y$}.

More generally, suppose that we are given a map of topological spaces
$\pi: L_{S} \rightarrow S$ and an action of $\RR$ on $L_{S}$ which exhibits each
fiber of $\pi$ as a broken line. For points $x \in L_{S}^{\circ} = L_{S} \setminus L_{S}^{\RR}$ and $y \in L_{S}$ having the same image $s \in S$, we define $d_{L_{S} }(x,y) = d_{L_{s} }( x,y) \in [ - \infty, \infty ]$.
\end{notation}

\begin{proposition}\label{proposition.translation-distance}
Let $\pi: L_{S} \rightarrow S$ be a map of topological spaces and suppose that $L_{S}$ is equipped with
a continuous action of $\RR$ which exhibits each fiber of $\pi$ as a broken line. 
Then $\pi: L_{S} \rightarrow S$ satisfies axiom $(A4)$ of Theorem \ref{theorem.line-recognition} if and only if the translation distance function
$$ d_{ L_{S}}: L_{S}^{\circ} \times_{S} L_{S} \rightarrow [ -\infty, \infty ]$$
is continuous.
\end{proposition}

\begin{proof}
Let $U$ be an open subset of $[ - \infty, \infty ]$; we wish to show that
$$ \{ (x,y) \in L_{S}^{\circ} \times_{S} L_{S}: d_{L_S}(x,y) \in U \}$$
is an open subset of the fiber product $L_{S}^{\circ} \times_{S} L_{S}$.
Without loss of generality, we may assume that $U$ is a half-open interval of the form
$(t, \infty]$ or $[ - \infty, t)$ for some real number $t$ (such intervals form a subbasis for the topology of
$[ - \infty, \infty ]$). By symmetry, we may assume that $U = (t, \infty ]$. Using the continuity of the $\RR$-action
on the second factor, we may reduce to the case $t = 0$. In this case, we have
$$d_{L_{S}}(x,y) \in U \Leftrightarrow y \nleq_{L_S} x,$$ so the desired conclusion is equivalent to axiom $(A4)$.
\end{proof}

\subsection{The Proof of Theorem \ref{theorem.line-recognition}}\label{section.line-recognition}

Our goal in this section is to supply a proof of Theorem \ref{theorem.line-recognition}. We begin by introducing a technical device
which is useful for producing ``local coordinates'' on the moduli stack of broken lines (see \S~\ref{section.stack}).

\begin{notation}\label{linpreorder}
Let $I$ be a set. A {\it linear preordering of $I$} is a binary relation $\leq_{I}$ on $I$ satisfying the following conditions:
\begin{itemize}
\item For every pair of elements $i, j \in I$, either $i \leq_{I} j$ or $j \leq_{I} i$.
\item If $i, j, k \in I$ satisfy $i \leq_{I} j$ and $j \leq_{I} k$, then $i \leq_{I} k$.
\end{itemize}
Note that the first condition guarantees that $\leq_{I}$ is reflexive: that is, we have $i \leq_{I} i$ for each $i \in I$.

If $\leq_{I}$ is a linear preordering of a set $I$ containing elements $i,j$, we write $i =_{I} j$ if
$i \leq_{I} j$ and $j \leq_{I} i$. Then $=_{I}$ is an equivalence relation on $I$. We will say that
$\leq_{I}$ is a {\it linear ordering} if $i =_{I} j$ implies that $i = j$ (that is, if the relation $\leq_{I}$ is antisymmetric).
\end{notation}

\begin{definition}\label{definition.I-section}
Let $\pi: L_{S} \rightarrow S$ be a map of topological spaces and suppose that
$L_{S}$ is equipped with an $\RR$-action which exhibits each fiber of $\pi$ as a broken line.

Let $I$ be a finite linearly preordered set. An {\it $I$-section of $L_{S}$}
is a collection of continuous maps $\{ \sigma_i: S \rightarrow L_{S} \setminus L_{S}^{\RR} \}_{i \in I}$ with the following properties:
\begin{itemize}
\item[$(a)$] Each $\sigma_{i}$ is a section of $\pi$: that is, we have $\pi \circ \sigma_i = \id_{S}$.

\item[$(b)$] For each $i \leq j$ in $I$ and each point $s \in S$, we have $d_{L_{S}}( \sigma_i(s), \sigma_j(s) ) > - \infty$;
here $d_{ L_{S} }$ is the translation distance function of Notation \ref{notation.translation-distance}.

\item[$(c)$] For each point $s \in S$ and each connected component $A$ of $L_s^{\circ}$,
there exists $i \in I$ such that $\sigma_i(s) \in A$. 
\end{itemize}
\end{definition}

\begin{proposition}[Existence of $I$-Sections]\label{proposition.section-existence}
Let $\pi: L_{S} \rightarrow S$ be a map of topological spaces and let $\mu: \RR \times L_{S} \rightarrow L_{S}$
be a continuous action of $\RR$ which preserves each fiber of $\pi$ and satisfies axioms
$(A1)$, $(A2)$, $(A4)$, and $(A5)$ of Theorem \ref{theorem.line-recognition}.
Then, for every point $s \in S$, there exists an open set $U \subseteq S$ containing $s$, a finite linearly ordered set $I$,
and an $I$-section of the projection map $L_{S} \times_{S} U \rightarrow U$.
\end{proposition}

\begin{proof}
Regard the fiber $L_s$ as equipped with the linear ordering $\leq_{L_s}$ of Notation \ref{notation.ordering}. 
Then the fixed point locus $L_{s}^{\RR}$ is a finite set $\{ x_0 <_{L} x_1 < \cdots <_{L} x_n \}$. 
Using axiom $(A5)$, we can choose an open set $U \subseteq S$ containing $s$ and a collection of continuous 
maps $\{ \sigma_i: U \rightarrow L_{S}^{\circ} \}_{1 \leq i \leq n}$ satisfying $x_{i-1} <_{L} \sigma_i(s) <_{L} x_i$.

Using axiom $(A2)$, we can further assume (after replacing $U$ by a smaller neighborhood of $s$ if necessary) that the fixed point locus $L_{S}^{\RR}$ decomposes as a disjoint union $K_0 \amalg K_1 \amalg \cdots \amalg K_n$, where $x_{i} \in K_i$ and each restriction $\pi|_{ K_i }: K_i \rightarrow U$ is a closed embedding. Using Axiom $(A4)$, we can further arrange (by shrinking $U$ as necessary) that for
all $y \in K_i$, we have $\sigma_{j}( \pi(y) ) <_{L_S} y$ if $j \leq i$, and $y <_{ L_{S} } \sigma_{j}( \pi(y) )$ for $j > i$. 
Finally, using the continuity of the translation distance function $d_{L_{S}}$ (Proposition \ref{proposition.translation-distance}), we can assume (after shrinking $U$) that $d_{L_S}( \sigma_i(s'), \sigma_j(s') ) > - \infty$ for all $s' \in U$ and $1 \leq i \leq j \leq n$.

We now claim that $\{ \sigma_i \}_{1 \leq i \leq n}$ is an $I$-section of the projection map $L_{U} \to U$, where $I$ is the set $\{ 1, \ldots, n \}$ (equipped with its usual ordering). 
The only nontrivial point is to verify that for each $s' \in U$, the set $\{ \sigma_i(s') \}_{ 1 \leq i \leq n}$ intersects each connected component $A$ of 
$L_{s'} \setminus L_{s'}^{\RR}$. Note that the closure of $A$ in the broken line $L_{s'}$
contains exactly two $\RR$-fixed points $p$ and $q$, which belong to the sets $K_{j}$ and $K_{k}$ for some $j < k$. 
In this case, our preceding assumptions guarantee that $p <_{ L_{S}} \sigma_k(s') <_{L_{S}} q$, so that $\sigma_{k}(s')$ belongs to $A$.
\end{proof}

\begin{proof}[Proof of Theorem \ref{theorem.line-recognition}]
Let $\pi: L_{S} \rightarrow S$ be a map of topological spaces and let $\mu: \RR \times L_{S} \rightarrow L_{S}$
be a continuous action of $\RR$ which satisfies axioms $(A1)$ through $(A5)$ of Theorem \ref{theorem.line-recognition}.
We now complete the proof of Theorem \ref{theorem.line-recognition} by showing that $( L_{S}, \pi, \mu)$ is an $S$-family of broken lines. Since this assertion is local on $S$
(Remark \ref{remark.locality}), we can reduce to the case where there exists a finite linearly ordered set $I$ and an $I$-section $\{ \sigma_{i} \}_{i \in I}$ of $\pi$ (Proposition \ref{proposition.section-existence}).
Let $n$ be the cardinality of $I$. We may assume without loss of generality that $n > 0$ (otherwise, the space $S$ is empty and there is nothing to prove). 

Let $d_{L_{S}}: L_{S}^{\circ} \times_{S} L_{S} \rightarrow [ - \infty, \infty ]$ denote the translation distance function
of Notation \ref{notation.translation-distance}, and choose an orientation-preserving homeomorphism
$$\rho: [ -\infty, \infty ] \simeq [ 0, \frac{1}{n} ].$$ We define a map 
$\gamma: L_{S} \rightarrow [0,1]$ by the formula
$$\gamma(x) = \sum_{i \in I} \rho( d_{ L_S}( \sigma_i( \pi(x) ), x ) ).$$
It follows from Proposition \ref{proposition.translation-distance} that the map $\gamma$ is continuous. To complete the proof
of Theorem \ref{theorem.line-recognition}, it will suffice to establish the following:

\begin{itemize}
\item[$(a)$] For each point $s \in S$, the restriction $\gamma|_{ L_s }: L_s \rightarrow [0,1]$ is a directed homeomorphism.

\item[$(b)$] The product map $(\gamma, \pi): L_{S} \rightarrow [0,1] \times S$ is a homeomorphism.
\end{itemize}

We first prove $(a)$. Fix a point $s \in S$, and regard the broken line $L_s$ as equipped with the ordering $\leq_{ L_s }$ of Notation \ref{notation.ordering}.
For each $i \in I$, the function $(x \in L_s) \mapsto d_{L_S}( \sigma_i(s), x) = d_{L_s}( \sigma_i(s), x)$ determines a nondecreasing, endpoint-preserving map
from $L_s$ to the interval $[ - \infty, \infty ]$. It follows that the function $\gamma|_{ L_s}: L_s \rightarrow [0,1]$ is nondecreasing and endpoint-preserving.
We will complete the proof by showing that $\gamma|_{L_s}$ is strictly increasing. By continuity, it will suffice to show that $\gamma|_{L_{s}}$ is strictly
increasing on each connected component $A$ of $L^{\circ}_{s} = L_{s} \setminus L_{s}^{\RR}$. Our assumption that
$\{ \sigma_i \}$ is an $I$-section of $\pi$ then guarantees that $A$ contains $\sigma_{i}(s)$ for some point $i \in I$. We now observe that the function $(x \in L_s) \mapsto d_{L_S}( \sigma_i(s), x)$ is strictly increasing when restricted to $A$. This completes the proof of $(a)$.

To prove $(b)$, we note that the product map $(\gamma, \pi): L_{S} \rightarrow [0,1] \times S$ is continuous and bijective (by virtue of $(a)$).
It will therefore suffice to show that it is a closed map, which follows from the fact that $L_{S}$ is proper over $S$ (Remark \ref{remark.properness}).
\end{proof}

\clearpage
\section{Presenting the Moduli Stack of Broken Lines}\label{section.stack}

Our goal in this section is to construct an atlas for the topological stack $\broken$ of Notation \ref{not-broke}. For this, it will be convenient to consider broken lines equipped with some additional
structure.

\begin{notation}\label{slabic}
Let $(I, \leq_I )$ be a finite linearly preordered set (Notation \ref{linpreorder}). For every topological space $S$, we let $\broken^{I}(S)$ denote
the groupoid whose objects are triples $(\pi: L_S \to S, \mu, \{ \sigma_i \}_{i \in I} )$, where the pair $(\pi: L_S \to S, \mu)$ is an $S$-family of broken
lines (Definition \ref{defn.S-family}) and $\{ \sigma_i \}_{i \in I}$ is an $I$-section of $L_S$ (Definition \ref{definition.I-section}). 
The construction $S \mapsto \broken^{I}(S)$ determines a stack (in groupoids) on the category of topological spaces, which we will
denote by $\broken^{I}$.
\end{notation}

\begin{remark}
The associated fibration in groupoids (see Definition~\ref{definition.fibration}) is as follows: We let $\Pt(\broken^I)$ denote the category whose objects are triples $(\pi: L_S \to S, \mu, \{ \sigma_i \}_{i \in I} )$. A morphism of $\Pt(\broken^I)$ is a morphism of $\Pt(\broken)$ compatible with the choices of $I$-sections. The forgetful functor $\Pt(\broken^I) \to \Top$ is the fibration in groupoids associated to $\broken^I$ (see Definition~\ref{definition.topological-stack}).
\end{remark}

For each $I$, $\broken^{I}$ is equipped with a map of topological stacks $\broken^{I} \to \broken$. Moreover, Proposition \ref{proposition.section-existence} guarantees
that these maps determine a covering of $\broken$: every family of broken lines $L_S \to S$ can be equipped with an $I$-section in a neighborhood of each
point $s \in S$, for some finite linearly preordered set $I$ (which might depend on $s$). Consequently, we can recover the moduli stack $\broken$ from
the moduli stacks $\broken^{I}$ together with information about fiber products of the form $\broken^{I} \times_{ \broken } \broken^{J}$. The main results of this section can be summarized as follows:

\begin{itemize}
\item For every finite linearly preordered set $I$, the topological stack $\broken^{I}$ is actually (representable by) a topological space (Theorem \ref{theorem-representab}). In other words,
there is a universal example of a family of broken lines equipped with an $I$-section; we will construct this universal family explicitly in \S \ref{section-universal}.

\item For every pair of linearly preordered sets $I$ and $J$, the fiber product $$\broken^{I} \times_{ \broken} \broken^{J}$$ is a manifold (with corners) which
admits an explicit open covering by open subsets of the form $\broken^{K}$, where $K$ ranges over all linearly preordered sets which can be obtained by
{\em amalgamating} $I$ and $J$ (Theorem \ref{theorem.fiber-product-description-vague}). Using this, we show that $\broken$ is (representable by) a Lie groupoid with corners (Proposition \ref{proposition.brokenpresentation}).

\item The topological stack $\broken$ can be realized as the (homotopy) colimit $\varinjlim_{I} \broken^{I}$, where $I$ ranges over all finite linearly preordered sets
(Theorem \ref{leftor}).
\end{itemize}

\subsection{The Moduli Stack \texorpdfstring{$\broken^{I}$}{Broken-I}}\label{cullus}

Our first goal is to show that, for every linearly preordered set $I$, the moduli stack $\broken^{I}$ is representable by a topological space.

\begin{notation}\label{preorderrep}
Let $(I, \leq_{I} )$ be a finite nonempty linearly preordered set. We can then form a category whose objects are the elements of $I$ whose morphisms are given by
$$ \Hom(i, j) = \begin{cases} \{ \ast \} & \text{ if } i \leq_{I} j \\
\emptyset & \text{ otherwise. } \end{cases}$$
In what follows, we will generally not distinguish between a linearly preordered set $I$ and its associated category.

Let $\RR^{+}$ denote the union $\RR \cup \{ \infty \}$, equipped with the topology given by its realization as a half-open interval
$( - \infty, \infty ]$. We regard $\RR^{+}$ as a commutative monoid under addition (with the convention that $t + \infty = \infty + t = \infty$
for all $t \in \RR^{+}$).

Let $\BR^{+}$ denote the category having a unique object $\ast$, with $\Hom_{ \BR^{+} }( \ast, \ast) = \RR^{+}$
(and composition of morphisms given by addition). For every finite linearly preordered set $I$, we let $\Rep(I, \BR^{+} )$ denote the set of all functors from
$I$ to $\BR^{+}$. 

More concretely, $\Rep( I, \BR^{+} )$ can be described as the set of all functions
$$ \alpha: \{ (i,j) \in I \times I: i \leq_{I} j \} \rightarrow \RR^{+}$$
which satisfy the equations 
$$ \alpha(i,i) = 0$$
$$ \alpha(i,j) + \alpha(j,k) = \alpha(i,k) \text{ if } i \leq_{I} j \leq_{I} k.$$
Consequently, we can view $\Rep( I, \BR^{+} )$ as a closed subset of a product of finitely many copies of $\RR^{+}$. 
We will regard $\Rep(I, \BR^{+} )$ as a topological space by equipping it with the subspace topology: that is, the coarsest
topology for which all of the functions $\alpha \mapsto \alpha(i,j) \in \RR^{+}$ are continuous.
\end{notation}

\begin{example}\label{ex1}
Let $I$ be a nonempty finite set. Then we can regard $I$ as a linearly preordered set by equipping it with the {\it indiscrete preordering}:
that is, by declaring that $i \leq_{I} j$ for every pair of elements $i, j \in I$. In this case, the set $\Rep(I, \BR^{+} )$ is homeomorphic to
$\RR^{ |I|-1}$. More precisely, for any choice of element $i \in I$, we have a canonical homeomorphism
$\Rep(I, \BR ) \simeq \RR^{ I - \{i\}}$, given by the construction
$$ (\alpha \in \Rep(I, \BR^{+} )) \mapsto \{ \alpha(i,j) \}_{j \neq i}.$$
\end{example}

\begin{example}\label{ex2}
Let $I = [n] = \{ 0 < 1 < 2 < \ldots < n \}$, which we regard as equipped with its usual ordering. In this case, we have
a canonical homeomorphism $\Rep(I, \BR ) \simeq (\RR^{+})^{n}$, given by
the construction
$$ (\alpha \in \Rep(I, \BR^{+} )) \mapsto ( \alpha(0,1), \alpha(1,2), \ldots, \alpha(n-1, n) ).$$
\end{example}

\begin{remark}\label{remark.opener}
Let $I$ be any nonempty finite linearly preordered set. Then we can (non-uniquely) write $I = \{ i_0, i_1, \ldots, i_n \}$ 
where $i_0 \leq_{I} i_1 \leq_{I} i_2 \leq_{I} \cdots \leq_{I} i_n$. In this case, the construction
$$ (\alpha \in \Rep(I, \BR^{+} )) \mapsto ( \alpha(i_0,i_1), \alpha(i_1, i_2), \ldots, \alpha(i_{n-1}, i_n) )$$
determines an open embedding $\Rep( I, \BR^{+} ) \hookrightarrow (\RR^{+})^{n}$, whose image
consists of those points $(t_1, \ldots, t_n) \in (\RR^{+})^{n}$ satisfying the condition
$t_m < \infty$ whenever $i_{m} \leq_{I} i_{m-1}$---that is, whenever $i_m \cong i_{m-1}$ in the category $I$. In particular, $\Rep(I, \BR^{+} )$ is a manifold (with corners) of dimension
$n$.
\end{remark}

\begin{construction}\label{suls}
Let $(I, \leq_I)$ be a nonempty finite linearly preordered set, let $S$ be a topological space, and let $(\pi: L_S \to S, \mu, \{ \sigma_i \}_{i \in I} )$
be an object of $\broken^{I}(S)$. To this data, we associate a map 
$$\alpha: S \rightarrow \Rep(I, \BR^{+} ) \quad \quad s \mapsto \alpha_s$$
by the formula
$$ \alpha_s(i,j) = d_{ L_S}( \sigma_i(s), \sigma_j(s) ).$$
Note that condition $(b)$ of Definition \ref{definition.I-section} guarantees that $\alpha_s(i,j) \neq - \infty$ for $i \leq_{I} j$,
and Proposition \ref{proposition.translation-distance} guarantees that the map $\alpha: S \to \Rep(I, \BR^{+} )$ is continuous.
\end{construction}

We can now formulate the first result of this section:

\begin{theorem}\label{theorem-representab}
Let $(I, \leq_I)$ be a nonempty finite linearly preordered set. Then, for every topological space $S$, Construction \ref{suls} induces an equivalence of categories
$$ \broken^{I}(S) \rightarrow \Hom_{ \Top}( S, \Rep(I, \BR^{+} ) )$$
where we regard the set $\Hom_{ \Top}( S, \Rep(I, \BR^{+} ) )$ as a category having only identity morphisms. In other words,
the topological stack $\broken^{I}$ is represented by the topological space $\Rep(I, \BR^{+} )$.
\end{theorem}

\subsection{The Universal Family}\label{section-universal}

Theorem \ref{theorem-representab} asserts that there is a universal family of broken lines which can be equipped with
an $I$-section, which is parametrized by the topological space $\Rep( I, \BR^{+} )$ of Notation \ref{preorderrep}.

\begin{example}\label{easybreak2}
Consider the topological space $[-\infty,\infty] \times [-\infty,\infty]$, equipped with the $\RR$-action given by translation on each factor.
For every real number $\alpha$, the construction $t \mapsto ( t + \alpha, t)$ determines a map $[ -\infty, \infty ] \rightarrow [ - \infty, \infty ] \times [ - \infty, \infty ]$,
whose image is a closed subset $L_{\alpha} \subseteq [ -\infty, \infty ] \times [ -\infty, \infty]$. As $\alpha$ approaches $\infty$, the sets $L_{\alpha}$ converge to a broken line
$L_{\infty} \subseteq [ -\infty, \infty ] \times [ -\infty, \infty ]$, given by the set-theoretic union
$$ ([ - \infty, \infty ] \times \{ - \infty \}) \cup ( \{ \infty \}, [ - \infty, \infty ] ).$$
The collection $\{ L_{\alpha} \}_{\alpha \in \RR \cup \{ \infty \} }$ can be organized into a family of broken lines over the topological space
$\RR^{+} \simeq \Rep( \{ 0 < 1 \}, \BR^{+} )$. Note that each $L_{\alpha}$ can be identified with a compactification of the set
$$ \{ (x,y) \in \RR_{\geq 0}: xy = e^{-\alpha} \}$$
via the construction $(x,y) \mapsto ( -\log(x), \log(y) )$ (compare with Example \ref{easybreak}).
\end{example}

We now consider a generalization of Example \ref{easybreak2}.

\begin{construction}\label{makelines}
Let $(I, \leq_I)$ be a nonempty finite linearly preordered set. We let
$$ \widetilde{\Rep}( I, \BR^{+} ) \subseteq \Rep(I, \BR^{+} ) \times [ - \infty, \infty ]^{I}$$
denote the subset consisting of those pairs $( \alpha: I \to \BR^{+}, \{ \beta(i) \}_{i \in I} )$
which satisfy the following condition:
\begin{itemize}
\item[$(\ast)$] Suppose that $i \leq_{I} j$ in $I$. If either $\beta(j) > - \infty$ or $\alpha(i,j) < \infty$, then $\beta(i) = \alpha(i,j) + \beta(j)$.
\end{itemize}

Note that there is an action of $\RR$ on $\widetilde{\Rep}(I, \BR^{+} )$, given by the map
$$ \mu: \RR \times \widetilde{\Rep}(I, \BR^{+} ) \rightarrow \widetilde{\Rep}(I, \BR^{+} )$$
$$ \mu( t, (\alpha, \{ \beta(i) \}_{i \in I} ) ) = ( \alpha, \{ \beta(i) + t \}_{i \in I}).$$
\end{construction}

\begin{proposition}\label{bez}
Let $(I, \leq_I)$ be a nonempty finite linearly preordered set. Then the projection map
$$ \widetilde{\Rep}(I, \BR^{+}) \subseteq \Rep(I, \BR^{+}) \times [ - \infty, \infty]^{I} \rightarrow \Rep(I, \BR^{+} )$$
exhibits $\widetilde{\Rep}(I, \BR^{+} )$ as a $\Rep(I, \BR^{+})$-family of broken lines (with respect to the $\RR$-action
described in Construction \ref{makelines}.
\end{proposition}

We begin by analyzing the fibers of the map $\widetilde{\Rep}(I, \BR^{+}) \to \Rep(I, \BR^{+})$.

\begin{lemma}\label{bezlem}
Let $(I, \leq_I)$ be a nonempty finite linearly preordered set, let $\alpha: I \to \BR^{+}$ be a functor,
and set $L = \widetilde{\Rep}( I, \BR^{+}) \times_{ \Rep( I, \BR^{+} ) } \{ \alpha \}$. Then $L$ is a broken line 
(with respect to the $\RR$-action described in Construction \ref{makelines}).
\end{lemma}

\begin{proof}
The map $\alpha$ determines an equivalence relation on $I$, where two elements $i, j \in I$
are equivalent if either $i \leq_{I} j$ and $\alpha(i,j) < \infty$, or $j \leq_I i$ and $\alpha(j,i) < \infty$.
This equivalence relation determines a partition $I = I_1 \amalg I_2 \amalg \cdots \amalg I_m$
into equivalence classes. For each equivalence class $I_a \subseteq I$, choose a maximal
element $i_a \in I_a$ (with respect to the linear preordering $\leq_I$). 
Reindexing if necessary, we may assume that for $1 \leq a < b \leq m$,
we have $i_a \leq_I i_b$ (so that $\alpha(i_a, i_b) = \infty$).

Let us identify $L$ with the subset of $[ - \infty, \infty ]^{I}$ consisting of those tuples $\{ \beta(i) \}_{i \in I}$
satisfying condition $(\ast)$ of Construction \ref{makelines}. For $1 \leq a \leq m$, we define an $\RR$-equivariant map
$\gamma_a: [ -\infty, \infty ] \to L$ by the formula
$$ \gamma_a(t)(i) = \begin{cases} \infty & \text{ if $i \in I_{b}$ for $b < a$} \\
\alpha(i, i_a) + t & \text{ if $i \in I_a$ } \\
- \infty & \text{ if $i \in I_{b}$ for $b > a$}. \end{cases}$$
Unwinding the definitions, we see that the maps $\{ \gamma_a \}_{1 \leq a \leq m}$ induce an
$\RR$-equivariant isomorphism of $L$ with the concatenation of $m$ copies of $[ - \infty, \infty ]$, so that
$L$ is a broken line.
\end{proof}

\begin{proof}[Proof of Proposition \ref{bez}]
We verify that $\widetilde{\Rep}( I, \BR^{+} )$ satisfies the hypotheses $(A1)$ through $(A5)$ of Theorem \ref{theorem.line-recognition}:

\begin{itemize}
\item[$(A1)$] For each point $\alpha \in \Rep( I, \BR^{+} )$, the fiber $\widetilde{\Rep}( I, \BR^{+}) \times_{ \Rep( I, \BR^{+} ) } \{ \alpha \}$
is a broken line; this follows from Lemma \ref{bezlem}.

\item[$(A2)$] The fixed point locus $\widetilde{\Rep}( I, \BR^{+} )^{\RR}$ is a closed subset of the product
$$\Rep(I, \BR^{+}) \times \{ \pm \infty \}^{I},$$ so the projection map $\widetilde{\Rep}( I, \BR^{+} )^{\RR} \rightarrow \Rep(I, \BR^{+} )$ is unramified.

\item[$(A3)$] The map $\widetilde{\Rep}( I, \BR^{+} ) \rightarrow \Rep(I, \BR^{+} )$ is closed, since it is a composition
of the closed embedding $\widetilde{\Rep}(I, \BR^{+} ) \hookrightarrow \Rep(I, \BR^{+}) \times [ - \infty, \infty ]^{I}$
with the projection map $\Rep(I, \BR^{+}) \times [ - \infty, \infty ]^{I} \rightarrow \Rep(I, \BR^{+} )$ (which is closed,
since the space $[ - \infty, \infty ]^{I}$ is compact).

\item[$(A4)$] Let $$C \subseteq \widetilde{\Rep}( I, \BR^{+} ) \times_{ \Rep(I, \BR^{+} ) } \widetilde{\Rep}(I, \BR^{+})$$
be the set of pairs $(x,y)$, where $x,y \in \widetilde{\Rep}(I, \BR^{+} )$ are points having the same
image $\alpha: I \to \BR^{+}$ in the space $\Rep(I, \BR^{+} )$ and $x \leq_{ L } y$, where
$L$ is the broken line of Lemma \ref{bezlem}. We wish to show that $C$ is closed.
This follows from the observation that $C$ is the inverse image of the closed subset
$C_0 \subseteq [ - \infty, \infty ]^{I} \times [ - \infty, \infty ]^{I}$ consisting of those
pairs of sequences $( \{ \beta(i) \}_{i \in I}, \{ \beta'(i) \}_{i \in I} )$ satisfying
$\beta(i) \leq \beta'(i)$ for each $i \in I$.

\item[$(A5)$] Let $p = (\alpha, \{ \beta(i) \}_{i \in I} )$ be a point
of $\widetilde{\Rep}(I, \BR^{+})$ which is {\em not} fixed by the action of $\RR$, so we can choose some element $j \in I$ such that
$\beta(j) \in \RR$. We define a map 
$$f: \Rep(I, \BR^{+} ) \rightarrow \widetilde{ \Rep}(I, \BR^{+} ) \quad \quad f( \alpha' ) = (\alpha', \{ \beta'(i) \}_{i \in I} )$$
where each $\beta'(i)$ is given by the formula
$$ \beta'(i) = \begin{cases} \alpha'(i,j) + \beta(j) & \text{ if $i \leq_I j$ } \\
- \alpha'(j,i) + \beta(j) & \text{ if  $j \leq_I i$}; \end{cases}$$
here we adopt the convention that $- \alpha'(j,i) + \beta(j) = - \infty$ when $\alpha'(j,i) = \infty$.
It is easy to see that $f$ is a section of the projection map $\widetilde{\Rep}( I, \BR^{+} ) \rightarrow \Rep(I, \BR^{+} )$
satisfying $f( \alpha ) = ( \alpha, \{ \beta(i) \}_{i \in I} )$. 
\end{itemize}
\end{proof}

\subsection{The Proof of Theorem \ref{theorem-representab}}

Throughout this section, we fix a topological space $S$ and a nonempty finite linearly preordered set $I$.
Our goal is to prove Theorem \ref{theorem-representab} by showing that Construction \ref{suls} induces an equivalence of
categories
$$ \broken^{I}(S) \rightarrow \Hom_{ \Top}( S, \Rep( I, \BR^{+} ).$$
We begin by noting that the left hand side is (equivalent to) a set:

\begin{lemma}\label{harmless}
Suppose we are given a pair of objects
$$ (\pi: L_S \to S, \mu, \{ \sigma_i \}_{i \in I}), 
( \pi': L'_{S} \to S, \mu', \{ \sigma'_{i} \}_{i \in I} ) \in \broken^{I}(S).$$
Then:
\begin{itemize}
\item[$(1)$] If $h: L_{S} \to L'_{S}$ is an $\RR$-equivariant map
satisfying $\pi' \circ h = \pi$ and $h \circ \sigma_i = \sigma'_{i}$ for
each $i \in I$, then $h$ is a homeomorphism (and can therefore be regarded
as an isomorphism of $L_{S}$ with $L'_{S}$ in the category $\broken^{I}(S)$).

\item[$(2)$] If $h, h': L_S \to L'_{S}$ are two maps satisfying the hypotheses of $(1)$, then
$h = h'$.
\end{itemize}
\end{lemma}

\begin{proof}
We first prove $(1)$. Note that since $L_{S}$ and $L'_{S}$ are both proper over $S$ (Remark \ref{remark.properness}),
any continuous map $h: L_S \to L'_{S}$ satisfying $\pi' \circ h = \pi$ is automatically closed.
Consequently, to show that $h$ is a homeomorphism, it will suffice to show that $h$ is bijective.
This condition can be checked fiberwise, so we can assume without loss of generality that
$S = \{s \}$ is a point. Note that the image of $h$ is an $\RR$-invariant subset
of $L'_{S}$ which contains each point $\sigma'_{i}(s)$, and therefore contains
the non-fixed locus $L'^{\circ}_{S} = L'_{S} \setminus L'^{\RR}_{S}$. Since
$h$ is closed and $L'^{\circ}_{S}$ is dense in $L'_{S}$, we conclude that $h$ is surjective.
To prove injectivity, suppose that we are given points $x \neq y$ in $L_{S}$ such
that $h(x) = h(y)$. Without loss of generality, we may assume that $x <_{L_S} y$
(where $\leq_{ L_S }$ is the ordering of Notation \ref{notation.ordering}). 
It then follows that the function $h$ is constant on the nonempty open set
$U = \{ z \in L_S: x <_{L_S} z <_{L_S} y \}$. The set $U$ has nonempty intersection
with some connected component $V$ of $L_{S}^{\circ} = L_{S} \setminus L_{S}^{\RR}$.
Choose $i \in I$ such that $\sigma_i(s)$ belongs to $V$. Then we can choose choose
real numbers $t \neq t' \in \RR$ such that $\mu(t, \sigma_i(s) ), \mu(t', \sigma_i(s)) \in U$.
It follows that
$$ \mu'(t, \sigma'_i(s) ) = h( \mu(t, \sigma_i(s)) ) = h( \mu(t', \sigma_i(s)) ) = \mu'(t', \sigma'_i(s)),$$
contradicting the fact that $\sigma'_i(s)$ belongs to the non-fixed locus $L'^{\circ}_{S}$. This completes the proof of $(1)$.

We now prove $(2)$. Suppose we are given a pair of maps $h,h': L_S \to L'_{S}$ satisfying the requirements of $(1)$;
we wish to show that $h(x) = h'(x)$ for each $x \in L_S$. 
Set $s = \pi(x)$. By continuity, we may assume without loss of generality that $x$ belongs
to the non-fixed locus $L_s^{\circ} = L_s \setminus L_s^{\RR}$. In this case, there
exists $t \in \RR$ and $i \in I$ such that $x = \mu(t, \sigma_i(s) )$, so that 
$h(x) = \mu'(t, \sigma'_i(s) ) = h'(x)$ by virtue of the $\RR$-equivariance of $h$ and $h'$.
\end{proof}

We now explicitly construct an inverse to the map of Construction \ref{suls}.

\begin{construction}\label{construction.inverse}
Let $\alpha: S \rightarrow \Rep(I, \BR^{+} )$ be a continuous map of topological spaces;
we denote the value of $\alpha$ on a point $s \in S$ by $\alpha_s \in \Rep(I, \BR^{+} )$. 
We let $[ \alpha ]_{S}$ denote the fiber product $\widetilde{\Rep}( I, \BR^{+} ) \times_{ \Rep(I, \BR^{+} )} S$.
Let $\pi: [\alpha]_S \to S$ be the projection onto the second factor.
By virtue of Proposition \ref{bez}, we can regard $\pi: [\alpha]_S \to S$ as a family of broken lines over $S$.
For each $j \in I$, we define a map $\widetilde{\alpha}_{j}: S \rightarrow \widetilde{ \Rep}(I, \BR^{+} )$
by the formula
$$ \widetilde{\alpha}_j( s ) = ( \alpha_s, \{ \beta_s(i) \}_{i \in I} )
\quad \quad \beta_s(i) = \begin{cases} \alpha_s(i,j) & \text{ if } i \leq_{I} j \\
- \alpha_s(j,i) & \text{ if } j \leq_{I} i. \end{cases}$$
Then each of the maps $\widetilde{\alpha}_{j}$ induces a section
$\sigma_j: S \to [ \alpha ]_{S}$ of the projection map $[ \alpha ]_{S}$.
It is not difficult to see that $\{ \sigma_j \}_{j \in I}$ is an $I$-section of $[ \alpha ]_{S}$,
in the sense of Definition \ref{definition.I-section}.
\end{construction}

\begin{proof}[Proof of Theorem \ref{theorem-representab}]
Let us abuse notation by identifying $\broken^{I}(S)$ with the set of isomorphism classes of objects
of $\broken^{I}(S)$ (this abuse is harmless, by virtue of Lemma \ref{harmless}). Construction \ref{construction.inverse}
then defines a map
$$ \rho: \Hom_{ \Top}( S, \Rep(I, \BR^{+} ) ) \rightarrow \broken^{I}(S)$$
$$ \rho( \alpha ) = ( [ \alpha ]_{S}, \{ \sigma_j \}_{j \in I} ).$$
If $\alpha: S \to \Rep(I, \BR^{+} )$ is a continuous map, then a simple calculation shows that the translation distance function
$d_{ [\alpha]_{S} }$ satisfies the identity
$d_{ [\alpha]_S }( \sigma_i(s), \sigma_j(s) ) = \alpha_{s}(i,j)$ for $i \leq_{I} j$. It follows that
$\rho$ is right inverse to the map given by Construction \ref{suls}. We claim that it is also a left inverse. To prove this,
suppose we are given an object $( \pi: L_S \to S, \mu, \{ \tau_j \}_{j \in I} ) \in \broken_{S}(i)$.
Define $\alpha: S \to \Rep(I, \BR^{+} )$ by the formula $\alpha_s(i,j) = d_{ L_S}( \tau_i(s), \tau_j(s) )$, as
in Construction \ref{suls}. Let $[\alpha]_{S}$ and $\{ \sigma_j: S \to [\alpha]_{S} \}_{j \in I}$ be as in Construction \ref{construction.inverse}.
We wish to show that $( L_S, \{ \tau_j \}_{j \in I} )$ and $( [ \alpha]_{S}, \{ \sigma_j \}_{i \in I} )$ are isomorphic
as objects of $\broken^{I}(S)$. To prove this, define a map
$h: L_S \to [\alpha]_S$ by the formula
$$ h(x) = ( \alpha_{\pi(s)}, \{ d_{ L_S}( \tau_i( \pi(s) ), x) \}_{i \in I} ).$$
A simple calculation shows that $h$ is an $\RR$-equivariant map satisfying $h \circ \tau_{i} = \sigma_i$ for $i \in I$
and compatible with the projection to $S$, and is therefore an isomorphism in $\broken^{I}(S)$ by virtue of
Lemma \ref{harmless}.
\end{proof}

\begin{corollary}\label{corollary.local-coordinates}
Let $S$ be a topological space and let $\pi: L_S \rightarrow S$ be an $S$-family of broken lines.
For each point $s \in S$, there exists an open set $U \subseteq S$ containing $s$,
a nonempty finite linearly ordered set $I$, a continuous map $\alpha: U \rightarrow \Rep( I, \BR^{+} )$, and
an isomorphism $$U \times_{S} L_S \simeq U \times_{ \Rep(I, \BR^{+} ) } \widetilde{\Rep}( I, \BR^{+} )$$
of $U$-families of broken lines.
\end{corollary}

\begin{proof}
Combine Theorem \ref{theorem-representab} (and its proof) with \ref{proposition.section-existence}.
\end{proof}

Corollary \ref{corollary.local-coordinates} can be regarded as ``parametrized version'' of the following elementary observation, which we asserted in Example~\ref{example.points-of-broken}:

\begin{corollary}[Classification of Broken Lines]\label{corollary.classification}
Let $L$ be a broken line. Then $L$ is isomorphic to a concatenation of finitely many copies of $[-\infty,\infty]$.
\end{corollary}

\begin{proof}
By virtue of Corollary \ref{corollary.local-coordinates}, we can assume that there exists a nonempty finite linearly ordered set $I$
and a map $\alpha: I \to \BR^{+}$ such that $L \simeq \{ \alpha \} \times_{ \Rep(I, \BR^{+} ) } \widetilde{\Rep}( I, \BR^{+} )$.
In this case, the desired result follows as in the proof of Lemma \ref{bezlem}.
\end{proof}

\subsection{Fiber Products}\label{section.fiber-products}	

Corollary \ref{corollary.local-coordinates} implies that the topological spaces $$\Rep(I, \BR^{+} ) \simeq \broken^{I}$$ form a covering of the moduli stack $\broken$ of broken lines. Our next goal is to describe the nerve of this covering.

\begin{notation}\label{fiberstack}
Let $(I, \leq_I)$ and $(J, \leq_J)$ be nonempty finite linearly ordered sets. We let
$\broken^{I,J}$ denote the fiber product $\broken^{I} \times_{\broken} \broken^{J}$, formed in the
$2$-category of topological stacks. More concretely, $\broken^{I,J}$ is the topological stack which
assigns to a topological space $S$ the groupoid of triples $(L_S, \{ \sigma_i \}_{i \in I}, \{ \tau_j \}_{j \in J} )$, where
$L_S$ is an $S$-family of broken lines, $\{ \sigma_i \}_{i \in I}$ is an $I$-section of $L_S$,
and $\{ \tau_j \}_{j \in J}$ is a $J$-section of $L_S$.
\end{notation}

\begin{proposition}\label{toprep}
For every pair of nonempty finite linearly ordered sets $(I, \leq_I)$ and $(J, \leq_J)$, the topological stack $\broken^{I,J}$ is (representable by) a topological space.
\end{proposition}

\begin{proof}
Let $X$ denote the $J$-fold fiber product of $\widetilde{\Rep}( I, \BR^{+} )$ with itself
over the space $\Rep(I, \BR^{+} )$. By definition, a point of $X$ is a tuple $\{ \widetilde{\alpha}_j \}_{j \in J}$
of points of $\widetilde{\Rep}(I, \BR^{+} )$, where each $\widetilde{\alpha}_j$ has the same image
$\alpha \in \Rep(I, \BR^{+} )$. Let $Y \subseteq X$ be the subset consisting of those
tuples $\{ \widetilde{\alpha}_j \}_{j \in J}$ which comprise a $J$-section of the broken line
$\{ \alpha \} \times_{ \Rep(I, \BR^{+} ) } \widetilde{\Rep}(I, \BR^{+} )$ (the subset $Y \subseteq X$ is open, though we do not need this). Unwinding the definitions, we have a pullback diagram
$$ \xymatrix{  Y \ar[r] \ar[d] & \broken^{I,J} \ar[d] \\
\Rep(I, \BR^{+} ) \ar[r] & \broken^{I}, }$$
where the bottom horizontal map is inverse to the equivalence of Theorem \ref{theorem-representab}
(that is, the map classifying the family of broken lines $\widetilde{\Rep}(I, \BR^{+} ) \to \Rep(I, \BR^{+} )$).
It follows that the upper horizontal map is also an equivalence: that is, $\broken^{I,J}$ is representable by the topological space $Y$.
\end{proof}

In what follows, we will abuse notation by identifying $\broken^{I,J}$ with the corresponding topological space. Our next goal is
to give an explicit open covering of $\broken^{I,J}$ by topological spaces of the form $\broken^{K}$ (Theorem \ref{theorem.fiber-product-description-vague}). 

\begin{remark}[Functoriality]\label{remark.functor}
Let $(I, \leq_I)$ and $(J, \leq_J)$ be a finite linearly preordered sets, so that $\leq_I$ and $\leq_J$ induce
equivalence relations $=_I$ and $=_J$ on the sets $I$ and $J$, respectively (see Notation \ref{linpreorder}).
We will say that a nondecreasing function $f: I \rightarrow J$ is {\it essentially surjective} if the induced map
$I / =_{I} \rightarrow J / =_J$ is surjective. If this condition is satisfied, then for any $S$-family of broken lines
$L_S$ and any $J$-section $\{ \sigma_j: S \to L_S \}_{j \in J}$ of $L_S$, the collection of maps
$\{ \sigma_{ f(i) }: S \to L_S \}_{i \in I}$ is an $I$-section of $L_S$. Consequently, the construction
$$ (\pi: L_S \to S, \mu, \{ \sigma_{j} \}_{j \in J} ) \mapsto (\pi: L_S \to S, \mu, \{ \sigma_{ f(i) } \}_{i \in I})$$
induces a functor $\Pt( \broken^{J} ) \to \Pt( \broken^{I} )$, which we can identify
with a morphism of topological stacks $\broken^{J} \to \broken^{I}$. (See Construction~\ref{construction.dotwell}.)
Under the identifications $\broken^{I} \simeq \Rep(I, \BR^{+} )$ and $\broken^{J} \simeq \Rep(J, \BR^{+} )$
supplied by Theorem \ref{theorem-representab}, this corresponds to the map
$\Rep(J, \BR^{+}) \to \Rep(I, \BR^{+} )$ given by precomposition with $f$. 
\end{remark}

\begin{notation}\label{notation.amalgam-poset}
Let $(I, \leq_{I})$ and $(J, \leq_{J})$ be nonempty finite linearly preordered sets. An {\it amalgam}
of $\leq_{I}$ and $\leq_{J}$ is a linearly preordered set $(K, \leq_K)$ with the following properties:
\begin{itemize}
\item The underlying set $K$ is the disjoint union $I \amalg J$.
\item The inclusion maps $I \hookrightarrow K \hookleftarrow J$ are nondecreasing and essentially surjective.
\end{itemize}
Let $\Amal(I,J)$ denote the collection of all amalgams of $I$ and $J$. We will regard $\Amal(I,J)$ as a category,
where a morphism from $(K, \leq_K)$ to $(K', \leq_{K'} )$ is a nondecreasing map which is the identity on $I$ and $J$. Note that this category is actually a partially ordered set---in particular, there is at most one morphism between every pair of objects of $\Amal(I,J)$.

It follows from Remark \ref{remark.functor} that each amalgam $(K, \leq_K) \in \Amal(I,J)$ determines a canonical map
$$ \broken^{K} \rightarrow \broken^{I,J} = \broken^{I} \times_{\broken} \broken^{J}.$$
Moreover, we can regard $\broken^{K}$ as a contravariant functor of $K \in \Amal(I,J)$.
\end{notation}

\begin{remark}
Let $(I, \leq_I)$ and $(J, \leq_J)$ be nonempty finite linearly preordered sets. Then the partially ordered set $\Amal(I,J)$ is a join semilattice. That is, every pair of objects $K, K' \in \Amal(I,J)$ has a least upper bound in
$\Amal(I,J)$, which we will denote by $K \vee K'$.
\end{remark}

We can now state the main result of this section:

\begin{theorem}\label{theorem.fiber-product-description-vague}
Let $(I, \leq_I)$ and $(J, \leq_J)$ be finite nonempty linearly preordered sets. Then:
\begin{itemize}
\item[$(a)$] For each $K \in \Amal(I,J)$, the canonical map
$\broken^{K} \rightarrow \broken^{I,J}$ is an open embedding, whose image is an open subset
$U_{K} \subseteq \broken^{I,J}$.
\item[$(b)$] The open sets $U_K$ cover the topological space $\broken^{I,J}$.
\item[$(c)$] For every pair of amalgams $K,K' \in \Amal(I,J)$, we have
$U_{K \vee K'} = U_K \cap U_{K'}$.
\end{itemize}
\end{theorem}

\begin{warning}
In the situation of Theorem \ref{theorem.fiber-product-description-vague}, 
each amalgam $K \in \Amal(I,J)$ can be viewed as a linearly preordered set which is {\em never} linearly ordered.
This is our primary motivation for considering the objects $\broken^{I}$ when $I$ is a general preordering, rather than restricting our attention to linear orderings.
\end{warning}

\begin{remark}
It follows formally from Theorem \ref{theorem.fiber-product-description-vague} that
$\broken^{I,J}$ can be realized as the colimit $\varinjlim_{K \in \Amal(I,J)^{\op} } \broken^{K}$, formed
either in the ordinary category of topological spaces or in the $2$-category of topological stacks. See Lemma \ref{lemma.answered} for a stronger statement along similar lines.
\end{remark}

\begin{proof}[Proof of Theorem \ref{theorem.fiber-product-description-vague}]
Suppose we are given a topological space $S$ and an object 
$$( L_S, \{ \sigma_i \}_{i \in I}, \{ \tau_j \}_{j \in J} ) \in \broken^{I,J}(S).$$ For every point $s \in S$, let us regard the elements $\{ \sigma_i(s), \tau_j(s) \}_{i \in I, j \in J}$
as defining a map of sets $f_s: I \amalg J \to L_s$. Let $\leq_s$ denote the linear preordering of
$I \amalg J$ given by 
$$ ( a \leq_s b ) \Leftrightarrow (d_{L_s}( a, b) \neq - \infty)$$
Let $K_s$ denote the linearly preordered set $(I \amalg J, \leq_s )$, which we will identify with an object
of $\Amal(I,J)$. For each object $K \in \Amal(I,J)$, we can identify the fiber product
$S \times_{ \broken^{I,J} } \broken^{K}$ with the open subset $S_K \subseteq S$
consisting of those points $s$ for which $K \leq K_s$ (in the partially ordered set $\Amal(I,J)$; see Notation~\ref{notation.amalgam-poset}).
From this observation, assertions $(a)$, $(b)$, and $(c)$ are immediate.
\end{proof}

\subsection{The Lie Groupoid of Broken Lines}

We now use the analysis of \S \ref{section.fiber-products} to define a smooth structure on the moduli stack $\broken$ of broken lines.

\begin{definition}\label{definition.liegroupoid}
Let $\calG$ be a category. We let $\Ob( \calG )$ denote the set of objects of $\calG$ and $\Mor(\calG)$ the set of morphisms of $\calG$,
so that we have maps $s,t: \Mor(\calG) \rightarrow \Ob(\calG)$ which associate to each morphism $f: C \rightarrow D$ its
source $C = s(f)$ and target $D = t(f)$, respectively. 

We will say that a category $\calG$ is a {\it Lie groupoid with corners} if the sets $\Ob(\calG)$ and $\Mor(\calG)$ are equipped with
the structure of smooth manifolds with corners satisfying the following conditions:
\begin{itemize}
\item[$(a)$] The map $s$ is submersive. That is, for every point $f \in \Mor(\calG)$, there exist open neighborhoods
$U \subseteq \Mor(\calG)$ containing $f$ and $V \subseteq \Ob(\calG)$ containing $s(f)$ for which
the restriction $s|_{U}$ factors as a composition
$$ U \xrightarrow{h} V \times \RR^{k} \rightarrow V \hookrightarrow \Ob(\calG),$$
where $h$ is a diffeomorphism.

\item[$(b)$] The map $t$ is submersive.

\item[$(c)$] The construction $(C \in \Ob(\calG) ) \mapsto (\id_{C} \in \Mor(\calG) )$ is a smooth map of manifolds with corners.

\item[$(d)$] The composition law $\Mor(\calG) \times_{\Ob(\calG)} \Mor(\calG) \rightarrow \Mor(\calG)$ is smooth map of manifolds with corners (note that
the fiber product $\Mor(\calG) \times_{ \Ob(\calG)} \Mor(\calG)$ inherits a smooth structure from assumption $(a)$ or $(b)$).

\item[$(e)$] The category $\calG$ is a groupoid: that is, every morphism $f \in \Mor(\calG)$ has an inverse $f^{-1} \in \Mor(\calG)$. Moreover,
the construction $f \mapsto f^{-1}$ determines a smooth map from $\Mor(\calG)$ to itself.
\end{itemize}
\end{definition}

\begin{remark}
As mentioned in the introduction, Lie groupoids represent reasonable geometric objects among all stacks on the site of smooth manifolds. This is analogous to the way in which Artin stacks are reasonable objects among all stacks on the site of algebraic varieties.
\end{remark}

\begin{remark}[From Lie Groupoids to Topological Stacks]\label{makestack}
Let $\calG$ be a Lie groupoid with corners and let $S$ be a topological space. We
define a category $\Hom^{\strict}_{\cont}( S, \calG )$ as follows:
\begin{itemize}
\item The objects of $\Hom^{\strict}_{\cont}(S, \calG)$ are continuous maps $S \rightarrow \Ob(\calG)$.
\item Given a pair of objects $C,D: S \rightarrow \Ob(\calG)$ of $\Hom^{\strict}_{\cont}(S, \calG)$, a morphism
from $C$ to $D$ is a continuous map $f: S \rightarrow \Mor(\calG)$ satisfying $s \circ f = C$ and $t \circ f = d$
(where $s,t: \Mor(\calG) \rightarrow \Ob(\calG)$ are the source and target maps, respectively).
\item Composition of morphisms in $\Hom^{\strict}_{\cont}(S, \calG)$ is given by composition of morphisms in $\calG$.
\end{itemize}
The construction $S \mapsto \Hom^{\strict}_{\cont}(S, \calG)$ determines a contravariant functor from the category
$\Top$ of topological spaces to the $2$-category of groupoids, which we can view as a topological prestack
(see Definition \ref{definition.topological-stack}). In general, this prestack is not a stack: that is, the construction
$S \mapsto \Hom^{\strict}_{\cont}(S, \calG)$ need not satisfy descent for open coverings. However, this can be rectified by passing
to the sheafification of the functor $S \mapsto \Hom^{\strict}_{\cont}(S, \calG)$, which we will denote by 
$S \mapsto \Hom_{\cont}(S, \calG)$.
\end{remark}

\begin{example}
Let $G$ be a Lie group, and let $\BG$ denote the groupoid having a single object with automorphism group
$G$ (so that $\Ob(\BG) = \ast$ and $\Mor(\BG) = G$). Then $\BG$ can be regarded as a Lie groupoid (without corners). 
For any topological space $S$, the groupoid $\Hom^{\strict}_{\cont}( S, \BG)$ has a single object with automorphism group
$\Hom_{\cont}( S, G)$. This category is equivalent to the category of {\em trivial} principal $G$-bundles on $S$.
The associated topological stack $S \mapsto \Hom_{\cont}(S, \BG)$ associates to each topological space $S$
the category of {\em all} $G$-bundles on $S$: that is, it agrees with the classifying stack of the topological group $G$,
in the sense of Example \ref{example.class-stack}.
\end{example}

We now apply these ideas to the study of broken lines. For every nonempty finite linearly ordered set $I$,
let us abuse notation by identifying $\broken^{I}$ with the topological space that represents it, given
explicitly by $\Rep(I, \BR^{+} )$ (Theorem \ref{theorem-representab}).

\begin{remark}[The Smooth Structure on $\broken^{[n]}$]\label{plusup0}
Let $I = [n] = \{ 0 < 1 < \cdots < n \}$ be the standard linearly ordered set with $n+1$ elements.
Then the construction
$$ ( \alpha \in \Rep(I, \BR^{+} ) ) \mapsto ( \alpha(0,1), \ldots, \alpha(n-1, n) )$$
determines a homeomorphism of topological spaces
$\broken^{I} \simeq \Rep(I, \BR^{+} ) \simeq (\RR^{+})^{n}$ (see Example \ref{ex2}).
Composing with the homeomorphism
$$ - \log: \RR^{+} \rightarrow \RR_{\geq 0},$$
we obtain a homeomorphism of topological spaces $\broken^{I} \simeq \RR^{n}_{\geq 0}$.
We will use this homeomorphism to endow $\broken^{I}$ with the structure of a smooth manifold with corners.
\end{remark}

\begin{remark}[The Smooth Structure on $\broken^{I}$]\label{plusup}
Let $(I, \leq_I)$ be an arbitrary nonempty finite linearly preordered set.
Then we can write $I = \{ i_0, i_1, \ldots, i_n \}$, where $i_0 \leq_{I} i_1 \leq_{I} \cdots \leq_{I} i_n$.
In other words, we can choose a map of linearly preordered sets $[n] \rightarrow I$ which is bijective on objects.
This map induces an open embedding
$$ \broken^{I} \hookrightarrow \broken^{[n]} \simeq \RR^{n}_{\geq 0},$$
whose image is the open subset of $\RR^{n}_{\geq 0}$ consisting of those sequences $(t_1, \ldots, t_n)$
satisfying $t_{m} > 0$ when $i_{m-1} =_{I} i_{m}$. Using this embedding, we endow
$\broken^{I}$ with the structure of a smooth manifold with corners; it is easy to see that the resulting
structure is independent of the chosen enumeration of $I$.
\end{remark}

\begin{remark}[The Smooth Structure on $\broken^{I,J}$]\label{plusup1}
Let $(I, \leq_I)$ and $(J, \leq_J)$ be nonempty finite linearly ordered sets, and let us abuse notation
by identifying the topological stack $\broken^{I,J}$ of Notation \ref{fiberstack} with its underlying
topological space. Then Theorem \ref{theorem.fiber-product-description-vague} supplies an
explicit covering of $\broken^{I,J}$ by open subspaces $\{ \broken^{K} \}_{K \in \Amal(I,J) }$. 
Using Remark \ref{plusup}, we can regard each of these subspaces as a smooth manifold with corners. It it easy to see that these smooth structures are compatible
along the intersections $\broken^{K} \cap \broken^{K'} \simeq \broken^{K \vee K'}$, so that
they endow $\broken^{I,J}$ with the structure of a smooth manifold with corners.
\end{remark}

\begin{construction}[The Lie Groupoid of Broken Lines]\label{construction.liegroup}
We define a category $\calB$ as follows:
\begin{itemize}
\item An object of $\calB$ consists of an integer $n \geq 0$ and a point $\alpha \in \Rep( [n], \BR^{+} )$ (that is, a sequence $( \alpha(0,1), \ldots, \alpha(n-1,n) ) \in \RR^{+}$). 
Identifying $\alpha$ with a continuous map from the one-point space to $\Rep( [n], \BR^{+} )$, we can apply Construction \ref{construction.inverse} to construct
a broken line, which we will denote by $[\alpha]$.

\item Let $([n], \alpha)$ and $([n'], \alpha')$ be objects of $\calB$. A morphism from $( [n], \alpha)$ to $( [n'], \alpha' )$ is an isomorphism of broken lines
$[ \alpha ] \simeq [ \alpha' ]$.
\end{itemize}
Note that the set $\Ob(\calB)$ of objects of $\calB$ can be identified with the disjoint union
$\coprod_{n \geq 0} \broken^{[n]}$, and that the set $\Mor(\calB)$ of morphisms of $\calB$ can be identified with the disjoint union
$\coprod_{m,n \geq 0} \broken^{[m], [n] }$. Using Remarks \ref{plusup0} and \ref{plusup1}, we can regard $\Ob(\calB)$ and $\Mor(\calB)$ as smooth manifolds with corners.
\end{construction}

\begin{proposition}
Let $\calB$ be the category of Construction \ref{construction.liegroup}. Then the smooth structures on $\Ob(\calB)$ and $\Mor(\calB)$ endow
$\calB$ with the structure of a Lie groupoid with corners, in the sense of Definition \ref{definition.liegroupoid}.
\end{proposition}

\begin{proof}
We will show that the source and target maps $s,t: \Mor(\calB) \rightarrow \Ob(\calB)$ are submersive, and leave the remaining axioms of Definition \ref{definition.liegroupoid} to the reader.
Fix nonnegative integers $m$ and $n$; we wish to show that the projection maps
$$ \broken^{[m]} \leftarrow \broken^{[m],[n]} \rightarrow \broken^{[n]}$$
are submersive. Covering $\broken^{[m],[n]}$ by open subsets of the form $\broken^{K}$ for $K \in \Amal([m], [n] )$, we are reduced to showing that
the projection maps $\broken^{[m]} \leftarrow \broken^{K} \rightarrow \broken^{[n]}$ are submersive (Theorem \ref{theorem.fiber-product-description-vague}).
Let $I$ and $J$ denote the sets $[m]$ and $[n]$, respectively, endowed with the linear preorderings given by the restriction of $\leq_{K}$.
Then we can identify $\broken^{I}$ and $\broken^{J}$ with open subsets of $\broken^{[m]}$ and $\broken^{[n]}$, respectively (Remark \ref{plusup}).
We are therefore reduced to showing that the projection maps $\broken^{I} \leftarrow \broken^{K} \rightarrow \broken^{J}$ are submersive.
By symmetry, it will suffice to treat the case of the projection map $s: \broken^{K} \rightarrow \broken^{I}$. For each
$j \in J$, choose an element $\phi(j) \in I$ such that $j =_{K} \phi(j)$. We conclude by observing that
the construction
$\alpha \mapsto ( s(\alpha), \{ \alpha( j, \phi(j) ) \}_{j \in J} )$ induces a diffeomorphism $\broken^{K} \simeq \broken^{I} \times \RR^{J}$.
\end{proof}

For any topological space $S$, we can identify $\Hom^{\strict}_{\cont}( S, \calB )$ with the groupoid of $S$-families of broken lines $L_S$
with the following property: there exists a decomposition of $S$ into closed and open subsets $\{ S_n \}_{n \geq 0}$ such that
each $L_{S_n}$ admits an $[n]$-section. According to Proposition \ref{proposition.section-existence}, every $S$-family of broken lines satisfies this condition
{\em locally} on $S$. We therefore obtain the following more precise formulation of Theorem \ref{theorem.mainA}:

\begin{proposition}\label{proposition.brokenpresentation}
For every topological space $S$, there is a canonical equivalence of categories $\broken(S) \simeq \Hom_{ \cont}( S, \calB)$, where
$\calB$ is the Lie groupoid with corners of Construction \ref{construction.liegroup}. In other words, the topological stack
$\broken$ is presented by the Lie groupoid $\calB$ via the construction of Remark \ref{makestack}.
\end{proposition}

We can regard Proposition \ref{proposition.brokenpresentation} as supplying a ``smooth structure'' on the topological stack $\broken$, which allows us to make sense of
{\em smooth} families of broken lines.

\begin{variant}[From Lie Groupoids to Smooth Stacks]
Let $\calG$ be a Lie groupoid with corners and let $S$ be a smooth manifold. We
define a category $\Hom^{\strict}_{\smooth}( S, \calG )$ as follows:
\begin{itemize}
\item The objects of $\Hom^{\strict}_{\smooth}(S, \calG)$ are smooth maps $S \rightarrow \Ob(\calG)$.
\item Given a pair of objects $C,D: S \rightarrow \Ob(\calG)$ of $\Hom^{\strict}_{\smooth}(S, \calG)$, a morphism
from $C$ to $D$ is a smooth map $f: S \rightarrow \Mor(\calG)$ satisfying $s \circ f = C$ and $t \circ f = d$
(where $s,t: \Mor(\calG) \rightarrow \Ob(\calG)$ are the source and target maps, respectively).
\item Composition of morphisms in $\Hom^{\strict}_{\smooth}(S, \calG)$ is given by composition of morphisms in $\calG$.
\end{itemize}
The construction $S \mapsto \Hom^{\strict}_{\smooth}(S, \calG)$ determines a contravariant functor from the category
$\Top$ of topological spaces to the $2$-category of groupoids. We denote the sheafification of this functor
by $S \mapsto \Hom_{\smooth}( S, \calG )$.
\end{variant}

\begin{example}[Smooth Families of Broken Lines]
Let $S$ be a smooth manifold and let $\calB$ be the Lie groupoid with corners of Construction \ref{construction.liegroup}.
Unwinding the definitions, we can identify the objects of $\Hom_{\smooth}( S, \calB )$ with $S$-families of broken lines $L_{S}$
equipped with a smooth structure on the non-fixed locus $L_{S}^{\circ}$ satisfying the following conditions:
\begin{itemize}
\item The projection map $L_{S}^{\circ} \rightarrow S$ is a smooth submersion.
\item Set $U = \{ (x,y) \in L_{S}^{\circ} \times_{S} L_{S}^{\circ}: d_{ L_S}( x, y) < \infty \}$. Then
the function
$$ ( (x,y) \in U ) \mapsto ( \log( d_{L_S}(x,y) ) \in \RR_{\geq 0} )$$
is smooth.
\end{itemize}
\end{example}

\begin{warning}
More generally, one can attempt to use Proposition \ref{proposition.brokenpresentation} to define the notion of a smooth $S$-families of broken lines where
$S$ is a manifold with corners. There are multiple candidates for this definition, depending on how one defines the notion of smooth map
$f: S \rightarrow \RR_{\geq 0}^{n}$ when $S$ is a manifold with corners (that is, what sort of differentiability does one require at the boundary of $S$?).
We caution the reader that the most conservative option (where we require $f$ to extend to a smooth function on a larger manifold without boundary)
does not seem appropriate for describing families of broken lines which arise naturally in geometric situations (such as Morse theory; see \S \ref{section.morse-example}).
\end{warning}

\subsection{Another Presentation of \texorpdfstring{$\broken$}{Broken}}

We now apply the analysis of \S \ref{section.fiber-products} to obtain another presentation of the moduli stack $\broken$, which will be particularly
convenient for our analysis of sheaves in \S \ref{section.sheaves-on-broken}. We begin by elaborating on Remark \ref{remark.functor}.

\begin{construction}[The Category of Linear Preorderings]\label{construction.dotwell}
We define a category $\LinPreOrd$ as follows:
\begin{itemize}
\item The objects of $\LinPreOrd$ are finite nonempty linearly preordered sets $(I, \leq_{I})$. 
\item A morphism from $(I, \leq_{I} )$ to $(J, \leq_{J} )$ in the category $\LinPreOrd$
is a nondecreasing map $\gamma: I \rightarrow J$ which is essentially surjective (that is, for every
element $j \in J$, there exists $i \in I$ such that $\gamma(i) =_{J} j$).
\end{itemize}
By virtue of Remark \ref{remark.functor}, the construction
$(I, \leq_I) \mapsto \broken^{I}$ determines a functor from the category $\LinPreOrd^{\op}$ to the $2$-category of topological stacks.
\end{construction}

Our goal in this section is to show that the moduli stack $\broken$ can be identified with the (homotopy) colimit
$$ \varinjlim_{I \in \LinPreOrd^{\op} } \broken^{I},$$
in the $2$-category of topological stacks. For later use, it will be convenient to formulate and prove a stronger version of this result.

\begin{notation}\label{notation.kan-complexes}
Let $\SSet$ denote the $\infty$-category of Kan complexes and let $\Top$ denote the category of topological spaces.
We say that a functor $\sheafF: \Top^{\op} \rightarrow \SSet$ is a {\it sheaf} if, for every topological space
$S$, the construction
$$ (U \subseteq S) \mapsto \sheafF(U)$$
determines an $\SSet$-valued sheaf on the topological space $S$ (see Definition \ref{definition.sheaf}). 
We let $\Shv_{\SSet}( \Top )$ denote the full subcategory of $\Fun( \Top^{\op}, \SSet)$ spanned by the sheaves on $\Top$.
\end{notation}

\begin{example}
The $2$-category of topological stacks (Definition \ref{definition.topological-stack}) can be identified with the full subcategory of
$\Shv_{\SSet}( \Top )$ spanned by those sheaves $\sheafF: \Top^{\op} \rightarrow \SSet$ with the property that, for each
topological space $S$, the Kan complex $\sheafF(S)$ is $1$-truncated: that is, the homotopy groups $\pi_{n}( \sheafF(S), x)$ vanish
for $n \geq 2$ (and any choice of base point $x \in \sheafF(S)$). 

In particular, we can regard the moduli stack $\broken$ as an object of $\Shv_{\SSet}( \Top )$, and the construction
$(I, \leq_{I} ) \mapsto \broken^{I}$ as a functor from $\LinPreOrd^{\op}$ to the $\infty$-category $\Shv_{\SSet}(\Top)$.
\end{example}

\begin{theorem}\label{leftor}
The canonical map
$$ \varinjlim_{I \in \LinPreOrd^{\op} } \broken^{I} \rightarrow \broken$$
exhibits $\broken$ as a colimit of the diagram $\{ \broken^{I} \}_{I \in \LinPreOrd^{\op} }$ in the $\infty$-category $\Shv_{\SSet}( \Top )$.
\end{theorem}
 
We begin with some general remarks. For every topological space $S$, let us abuse notation by identifying $S$ with
the functor $\Hom_{ \Top}( \bullet, S )$, which we regard as an object of $\Shv_{\SSet}( \Top )$.

\begin{remark}
Let $X$ be a topological space. Note that $\Hom_{\Top}(X,S)$ is a discrete Kan complex---for example, $\pi_0\Hom_{\Top}(X,S)$ is naturally in bijection with the set of all continuous maps from $X$ to $S$. 
\end{remark}

\begin{lemma}\label{oloic}
Let $S$ be a topological space, let $A$ be a partially ordered set, and let
$\{ U_{\alpha} \}_{\alpha \in A}$ be a collection of open subsets of $S$ indexed by $A$.
Assume that:
\begin{itemize}
\item[$(i)$] The open sets $U_{\alpha}$ cover the topological space $S$.
\item[$(ii)$] For every pair of elements $\alpha, \beta \in A$, there exists a least
upper bound $\alpha \vee \beta$ in $A$, and we have
$$ U_{\alpha \vee \beta} = U_{\alpha} \cap U_{\beta}.$$
\end{itemize}
It follows from $(ii)$ that the construction $\alpha \mapsto U_{\alpha}$ determines an
order-reversing map from $A$ to the partially ordered set $\calU(S)$ of open subsets of $S$,
and therefore defines a functor
$$ A^{\op} \rightarrow \calU(S) \rightarrow \Top \hookrightarrow \Shv_{\SSet}( \Top ).$$
Then the canonical map $\varinjlim_{\alpha \in A^{\op}} U_{\alpha} \rightarrow S$
exhibits $S$ as a colimit of the diagram $\{ U_{\alpha} \}_{\alpha \in A}$ in
the $\infty$-category $\Shv_{\SSet}( \Top )$.
\end{lemma}

\begin{proof}
Let $\sheafF: \Top^{\op} \rightarrow \SSet$ be a sheaf. We wish to show
that the canonical map
$$ \theta: \bHom_{ \Shv_{\SSet}( \Top) }( S, \sheafF ) \rightarrow
\varprojlim_{\alpha \in A} \bHom_{ \Shv_{\SSet}(\Top) }( U_{\alpha}, \sheafF )$$
is a homotopy equivalence. Using Yoneda's lemma, we can identify
$\theta$ with the canonical map
$$ \sheafF(S) \rightarrow \varprojlim_{\alpha \in A} \sheafF( U_{\alpha} ),$$
which is a homotopy equivalence by virtue of our assumption that $\sheafF$ is a sheaf. 
\end{proof}

Combining Lemma \ref{oloic} with Theorem \ref{theorem.fiber-product-description-vague}, we obtain the following:

\begin{lemma}\label{lemma.answered}
Let $(I, \leq_I)$ and $(J, \leq_J)$ be nonempty finite linearly preordered sets. Then the canonical map
$$ \varinjlim_{K \in \Amal(I,J)^{\op} } \broken^{K} \rightarrow \broken^{I} \times_{ \broken } \broken^{J}$$
exhibits the fiber product $\broken^{I} \times_{ \broken } \broken^{J}$ as a colimit of the diagram
$$\{ \broken^{K} \}_{K \in \Amal(I,J)^{\op} }$$ in the $\infty$-category $\Shv_{\SSet}( \broken )$.
\end{lemma}

\begin{proof}[Proof of Theorem \ref{leftor}]
By virtue of Proposition \ref{proposition.section-existence}, the canonical map
$$\theta: \varinjlim_{I \in \LinPreOrd^{\op} } \broken^{I} \rightarrow \broken$$
is locally surjective: that is, for every $S$-family of broken lines $L_{S}$,
we can find an open covering $\{ U_{\alpha} \}$ of $S$ for which each $L_{ U_{\alpha} }$
belongs to the essential image of $( \varinjlim_{I \in \LinPreOrd^{\op} } \broken^{I})( U_{\alpha} )$.
It will therefore suffice to show that $\theta$ is a monomorphism in $\Shv_{\SSet}( \Top)$.

Form a pullback diagram
$$ \xymatrix{ \sheafF \ar[r] \ar[d] & \varinjlim_{ I \in \LinPreOrd^{\op} } \broken^{I} \ar[d]^{\theta} \\
\varinjlim_{ J \in \LinPreOrd^{\op} } \broken^{J} \ar[r]^{\theta} & \broken. }$$
so that $\sheafF$ can be written as a colimit $\varinjlim_{I,J \in \LinPreOrd^{\op} } \broken^{I,J}$
in the $\infty$-category $\Shv_{\SSet}( \Top )$. Using Lemma \ref{lemma.answered}, we can write
$\sheafF$ as an iterated colimit
$$ \varinjlim_{I, J \in \LinPreOrd^{\op} } \varinjlim_{K \in \Amal(I,J)^{\op} } \broken^{K}.$$
Equivalently, $\sheafF$ can be realized as the colimit
$$ \varinjlim_{ (I,J,K) \in \calJ^{\op} } \broken^{K},$$
where $\calJ$ denotes the category whose objects are triples $(I, J, K )$
where $I, J \in \LinPreOrd$ and $K \in \Amal(I,J)$ is an amalgam of
$I$ and $J$; a morphism from $(I, J, K )$ to $(I', J', K' )$ is a
pair of morphisms $I \rightarrow I'$ and $J \rightarrow J'$ in $\LinPreOrd$
for which the induced map of sets $$K = I \amalg J \rightarrow I' \amalg J' = K'$$ 
is nondecreasing (and therefore also a morphism in $\LinPreOrd$).

The construction $(I,J,K) \mapsto K$ determines a functor
$\calJ \rightarrow \LinPreOrd$. This functor admits a right adjoint
(which carries a linearly preordered set $K$ to the disjoint union $K \amalg K$)
and is therefore right cofinal. It follows that the canonical map
$$\sheafF \simeq \varinjlim_{ (I, J, K) \in \calJ^{\op} } \broken^{K}
\rightarrow \varinjlim_{K \in \LinPreOrd^{\op}} \broken^{K}$$
is an equivalence in $\Shv_{\SSet}( \Top )$. We observe that this map is a left
homotopy inverse to the relative diagonal
$$ \delta: \varinjlim_{K \in \LinPreOrd^{\op} } \Rep(K, \BR^{+} )
\rightarrow \varinjlim_{I \in \LinPreOrd^{\op} } \Rep(I, \BR^{+} ) \times_{ \broken} 
\varinjlim_{J \in \LinPreOrd^{\op} } \Rep(J, \BR^{+} ) = \sheafF.$$
It follows that $\delta$ is also an equivalence in $\Shv_{\SSet}( \broken)$, so that $\theta$ is a monomorphism as desired.
\end{proof}

\subsection{The Moduli Stack \texorpdfstring{$\broken_{I}$}{Broken-I}}

Let $I$ be a nonempty finite linearly ordered set and let $\Rep(I, \BR^{+})$ denote the set of all functors from $I$ to
the category $\BR^{+}$ (Notation \ref{preorderrep}). According to Theorem \ref{theorem-representab}, we can regard $\Rep(I, \BR^{+} )$
as a topological space which represents the functor $\broken^{I}$ of Notation \ref{slabic}. However, it is perhaps more natural to view
$\Rep(I, \BR^{+})$ as the collection of objects of the {\em category} $\Fun(I, \BR^{+} )$ of all functors from $I$ to $\BR^{+}$. We let $\Fun(I, \BR^{+} )^{\simeq}$ denote the underlying groupoid of this category: that is,
the groupoid whose objects are the elements of $\Rep(I, \BR^{+})$, where a morphism from $\alpha: I \rightarrow \BR^{+}$
to $\alpha': I \rightarrow \BR^{+}$ is given by a map $\gamma: I \rightarrow \RR$ satisfying
$$ \alpha(i,j) + \gamma(j) = \gamma(i) + \alpha'(i,j)$$
for $i \leq_{I} j$. We will regard $\Fun(I, \BR^{+} )^{\simeq}$ as an example of a Lie groupoid with corners (Definition \ref{definition.liegroupoid}),
where the set of objects $\Rep(I, \BR^{+} )$ is regarded as a manifold with corners as in Remark \ref{plusup}, and the set of morphisms
is regarded as a manifold with corners by identifying it with $\Rep(I, \BR^{+}) \times \RR^{I}$.
Our goal in this section is to give an explicit description of the topological stack represented by the Lie groupoid $\Fun(I, \BR^{+})^{\simeq}$.

We begin with some elementary observations. 

\begin{lemma}\label{lemma-braceward}
Let $S$ be a topological space and let $\pi: L_{S} \rightarrow S$ be an $S$-family of broken lines.
Then:
\begin{itemize}
\item[$(i)$] The action of $\RR$ on $L_{S}^{\circ} = L_{S} - L_{S}^{\RR}$ is topologically
free: that is, $L_{S}^{\circ}$ is an $\RR$-torsor over the quotient $L_{S}^{\circ} / \RR$.

\item[$(ii)$] The induced map $L_{S}^{\circ} / \RR \rightarrow S$ is a local homeomorphism.
\end{itemize}
\end{lemma}

\begin{warning}
In the situation of Lemma \ref{lemma-braceward}, the quotient $L_{S}^{\circ} / \RR$ is usually not Hausdorff
(even if the space $S$ is assumed to be Hausdorff).
\end{warning}

\begin{proof}[Proof of Lemma \ref{lemma-braceward}]
Using Corollary \ref{corollary.local-coordinates}, we can assume that $S = \Rep(I, \BR^{+} )$
and $L_{S} = \widetilde{\Rep}(I, \BR^{+} )$ for some nonempty finite linearly ordered set $I$ (see Construction \ref{makelines}).
For each $i \in I$, let $L_{S}(i)$ denote the subset of $L_S$ consisting of those
pairs $( \alpha, \{ \beta(j) \}_{j \in I} )$ where $\beta(i) \in \RR$. Then the sets $\{ L_S(i) \}_{i \in I}$
comprise an open covering of $L_{S}^{\circ}$ by $\RR$-invariant open subsets. It will therefore suffice to prove the following:
\begin{itemize}
\item[$(i')$] The action of $\RR$ on each $L_{S}(i)$ is topologically free.
\item[$(ii')$] Each of the maps $L_{S}(i) / \RR \rightarrow S$ is a local homeomorphism.
\end{itemize}
In fact, something stronger is true: there are $\RR$-equivariant isomorphisms
$L_{S}(i) \simeq S \times \RR$, given by the construction $(\alpha, \{ \beta(j) \}_{j \in I} ) \mapsto (\alpha, \beta(i))$.
\end{proof}

\begin{definition}\label{definition.weak-I-section}
Let $S$ be a topological space, let $\pi: L_S \rightarrow S$ be an $S$-family of broken lines, and let $I$ be a finite nonempty linearly preordered set.
A {\it weak $I$-section} of $L_{S}$ is a collection of continuous maps $\{ \overline{\sigma}_i: S \rightarrow L_{S}^{\circ} / \RR \}_{i \in I}$ with the following properties:
\begin{itemize}
\item Each $\overline{\sigma}_i$ is a section of the local homeomorphism $L_{S}^{\circ} / \RR \rightarrow S$: that is,
the composite maps $S \xrightarrow{ \overline{\sigma}_i } L_{S}^{\circ} / \RR \rightarrow S$ are the identity on $S$.

\item For each $s \in S$, the induced map
$$ I \rightarrow L_{s}^{\circ} / \RR \quad \quad i \mapsto \overline{\sigma}_{i}(s)$$
is surjective and nondecreasing (where $L^{\circ}_{s} / \RR$ is equipped with the linear ordering inherited from the ordering on the broken line
$L_{s}$).
\end{itemize}
\end{definition}

\begin{example}\label{example.weak-sections-versus-sections}
Let $S$ be a topological space, let $\pi: L_S \rightarrow S$ be an $S$-family of broken lines, and let $I$ be a finite nonempty linearly preordered set. 
Suppose we are given a collection of maps $\{ \sigma_i: S \rightarrow L_S^{\circ} \}_{i \in I}$, and for each $i \in I$ let
$\overline{\sigma}_i$ denote the composite map $S \xrightarrow{\sigma_i} L_S^{\circ} \rightarrow L_S^{\circ} / \RR$. Then
$\{ \sigma_i \}_{i \in I}$ is an $I$-section of $L_S$ (in the sense of Definition \ref{definition.I-section}) 
if and only if $\{ \overline{\sigma}_i \}_{i \in I}$ is a weak $I$-section of $L_S$ (in the sense of Definition \ref{definition.weak-I-section}). In particular, every $I$-section of $L_S$ gives rise to a weak $I$-section of
$L_S$. Moreover, every weak $I$-section of $L_{S}$ arises in this way, at least locally on $S$ (or even globally on $S$, if every $\RR$-torsor over $S$ is trivial: this condition is satisfied whenever $S$ is a paracompact Hausdorff space).
\end{example}

\begin{remark}\label{suffices}
Let $\pi: L_S \rightarrow S$ and $I$ be as in Definition \ref{definition.weak-I-section}, and let $\{ \overline{\sigma}_{i}: S \rightarrow L_{S}^{\circ} / \RR \}_{i \in I}$ be a weak $I$-section.
Then $\overline{\sigma}_{i} = \overline{\sigma}_{j}$ whenever $i =_{I} j$. It follows that we can identify weak $I$-sections of $L_{S}$ with weak
$(I / =_{I} )$-sections of $L_{S}$. Consequently, nothing is lost by replacing $I$ by the quotient $I / =_{I}$, and thereby restricting to the case where
$I$ is a linearly ordered set.
\end{remark}

\begin{notation}\label{notation.broken-I}
Let $I$ be a finite nonempty linearly ordered set. We define a category $\Pt(\broken_I)$ as follows:
\begin{itemize}
\item The objects of $\Pt(\broken_I)$ are triples $(\pi: L_S \rightarrow S, \mu, \{ \overline{\sigma}_i \}_{i \in I})$, where $(\pi, \mu)$ is an object
of the category $\Pt(\broken)$ (see Construction \ref{construction.category-of-lines}) and $\{ \overline{\sigma}_i: S \rightarrow L_S^{\circ} / \RR \}_{i \in I}$ is
a weak $I$-section of $L_S$ (see Definition \ref{definition.weak-I-section}).

\item A morphism from $(\pi': L_{S'} \rightarrow S', \mu', \{ \overline{\sigma}'_i \}_{i \in I} )$ to $(\pi: L_{S} \rightarrow S, \mu,  \{ \overline{\sigma}_i \}_{i \in I} )$ in
the category $\Pt(\broken_I)$ is a pair of continuous maps
$g: S' \rightarrow S$, $\widetilde{g}: L_{S'} \rightarrow L_{S}$ which define a morphism from $(\pi', \mu')$ to $(\pi, \mu)$ in the category $\Pt(\broken)$ and for which each of the diagrams
$$ \xymatrix{ S' \ar[d]^{ \overline{\sigma}'_i} \ar[r]^{ g } & S \ar[d]^{ \overline{\sigma}_i } \\
L_{S'}^{\circ} / \RR \ar[r] & L_S^{\circ} / \RR }$$
commutes. Here, the bottom horizontal arrow is the map induced by $\widetilde{g}$. 
\end{itemize}
\end{notation}

\begin{lemma}\label{lemma.compare-broken-I}
Let $I$ be a finite nonempty linearly ordered set. Then the forgetful functor $\Pt(\broken_I) \rightarrow \Pt(\broken)$ is a representable
local homeomorphism of topological stacks (see Remark \ref{remark.extension-by-zero-2}). In other words, for any
family of broken lines $L_S \rightarrow S$, the fiber product $S \times_{\broken} \broken_I$ is (representable by)
a topological space $S'$ for which the projection map $S' \rightarrow S$ is a local homeomorphism.
\end{lemma}

\begin{proof}
Unwinding the definitions, we see that $S \times_{\broken} \broken_I$ can be identified with
an open subset of the product $\prod_{i \in I} L_{S}^{\circ} / \RR$, formed in the category $\Top_{/S}$. The desired result now
follows from the observation that the map $L_{S}^{\circ} / \RR \rightarrow S$ is a local homeomorphism (Lemma \ref{lemma-braceward}).
\end{proof}

\begin{proposition}\label{lemma.structure-of-broken-I}
Let $I$ be a finite nonempty linearly ordered set. Then, for every topological space $S$, we have a
canonical equivalence
$$ \broken_{I}(S) \simeq \Hom_{\cont}(S, \Fun(I, \BR^{+})^{\simeq} ),$$
where the right hand side is obtained by applying Remark \ref{makestack} to the Lie groupoid with corners $\Fun(I, \BR^{+} )^{\simeq}$.
\end{proposition}

\begin{proof}
Combine Example \ref{example.weak-sections-versus-sections} with Theorem \ref{theorem-representab}.
\end{proof}

\begin{remark}\label{remark.structure}
Proposition \ref{lemma.structure-of-broken-I} can be restated as follows: the topological stack $\broken_{I}$
can be realized as the stack-theoretic quotient of the topological space $\Rep(I, \BR^{+}) \simeq \broken^{I}$
by the action of $\RR^{I}$, given concretely by the map
$$ \RR \times \Rep(I, \BR^{+}) \rightarrow \Rep(I, \BR^{+}) \quad \quad (\gamma, \alpha) \mapsto \{ \alpha(i,j) - \gamma(i) + \gamma(j) \}_{i \leq_{I} j}.$$
\end{remark}

We close this section by establishing an analogue of Theorem \ref{leftor}.

\begin{notation}\label{notation.linord}
We define a category $\LinOrd$ as follows:
\begin{itemize}
\item The objects of $\LinOrd$ are pairs $(I, \leq_{I} )$, where $I$ is a finite nonempty set and
$\leq_{I}$ is a linear ordering on $I$.
\item A morphism from $(I, \leq_{I} )$ to $(J, \leq_{J} )$ in the category $\LinOrd$ is a surjection $f: I \rightarrow J$
which is nondecreasing: that is, $i \leq_{I} i'$ implies $f(i) \leq_{J} f(i')$.
\end{itemize}
Note that we can regard $\LinOrd$ as a full subcategory of the category $\LinPreOrd$ introduced in Construction \ref{construction.dotwell}.
We will generally abuse notation by identifying an object $(I, \leq_{I} ) \in \LinOrd$ with its underlying set $I$.
\end{notation}

\begin{theorem}\label{leftor2}
The canonical map
$$ \phi: \varinjlim_{ I \in \LinOrd^{\op} } \broken_{I} \rightarrow \broken $$
exhibits $\broken$ as a colimit of the diagram $\{ \broken_{I} \}_{I \in \LinOrd^{\op} }$ in the $\infty$-category $\Shv_{\SSet}( \Top )$
\end{theorem}

\begin{proof}
Note that, for every finite nonempty linearly preordered set $\widetilde{I}$, we have a map of topological stacks
$\broken^{\widetilde{I}} \rightarrow \broken_{ \widetilde{I} / =_{ \widetilde{I} } }$, depending functorially on $\widetilde{I}$ (see Remark \ref{suffices}). Passing to the colimit over $\widetilde{I}$,
we obtain a map $\psi: \varinjlim_{\widetilde{I} \in \LinPreOrd^{\op} } \broken^{ \widetilde{I}} \rightarrow \varinjlim_{ I \in \LinOrd^{\op} } \broken_{I}$, whose composition
with $\phi$ is the equivalence $\varinjlim_{\widetilde{I} \in \LinPreOrd^{\op} } \broken^{ \widetilde{I}} \simeq \broken$ of Theorem \ref{leftor} (here all colimits are formed in the
$\infty$-category $\Shv_{\SSet}( \Top )$). It will therefore suffice to show that $\psi$ is an equivalence. In fact, we make a more precise claim: the diagram
$\{ \broken_{I} \}_{I \in \LinOrd^{\op} }$ is a left Kan extension of the diagram $\{ \broken^{\widetilde{I}} \}_{\widetilde{I} \in \LinPreOrd^{\op} }$ along the functor
$$ q: \LinPreOrd^{\op} \rightarrow \LinOrd^{\op} \quad \quad \widetilde{I} \mapsto \widetilde{I} / =_{ \widetilde{I} }.$$
Since $q$ is an op-fibration of categories, this assertion can be reformulated as follows:
\begin{itemize}
\item[$(\ast)$] Let $I$ be a nonempty finite linearly ordered set, and set $\cC = \LinPreOrd \times_{ \LinOrd } \{ I \}$. Then
the canonical map $\varinjlim_{ \widetilde{I} \in \cC^{\op}} \broken^{ \widetilde{I} } \rightarrow \broken_{I}$ is an equivalence in the $\infty$-category $\Shv_{\SSet}( \Top )$.
\end{itemize}
Let us regard $\broken^{I}$ as an $\RR^{I}$-torsor over $\broken_{I}$, classified by a map of topological stacks $\broken_{I} \rightarrow \BR^{I}$. For each
$\widetilde{I} \in \cC^{\op}$, we have a pullback diagram of topological stacks
$$ \xymatrix{ \broken^{ \widetilde{I} } \ar[r] \ar[d] & \broken_{I} \ar[d] \\
\RR^{ \widetilde{I} } / \RR^{I} \ar[r] & \BR^{I}. }$$
It will therefore suffice to show that the lower horizontal maps induce an equivalence 
$$ \rho: \varinjlim_{ \widetilde{I} \in \cC^{\op} } \RR^{ \widetilde{I} } / \RR^{I} \rightarrow \BR^{I}.$$

Let $\cC_0$ denote the category of nonempty finite sets. Then the construction $\widetilde{I} \mapsto \{ \widetilde{I} \times_{I} \{i\} \}_{i \in I}$ induces an
equivalence of categories $\cC \simeq \cC_0^{I}$. Unwinding the definitions, we see that $\rho$ can be identified with the $I$th power of a morphism
$$ \rho_0: \varinjlim_{ J \in \cC_0^{\op} } \RR^{J} / \RR \rightarrow \BR$$
in $\Shv_{\SSet}( \Top )$; here $\RR$ acts diagonally on each $\RR^{J}$. To show that $\rho_0$ is an equivalence, it will suffice to show that the final
object ${\bf 1} \in \Shv_{\SSet}( \Top )$ is a colimit of the diagram $\{ \RR^{J} \}_{J \in \cC^{\op} }$; here we identify each $\RR^{J}$ with
the associated representable functor 
$$\Top^{\op} \rightarrow \Set \subseteq \SSet \quad \quad S \mapsto \Hom_{\Top}( S, \RR^{J} ).$$
In fact, we claim that ${\bf 1}$ is already a colimit of the diagram $\{ \RR^{J} \}_{J \in \cC^{\op} }$ in the presheaf $\infty$-category $\Fun( \Top^{\op}, \SSet)$:
that is, for every topological space $S$, the (homotopy) colimit $\varinjlim_{J \in \cC_0^{\op} } \Hom_{\Top}( S, \RR^{J} )$ is contractible. This
homotopy colimit can be identified with the (Kan complex associated to the) nerve of the category whose objects are nonempty finite sets $J$ equipped with a map
$J \rightarrow \Hom_{\Top}(S, \RR )$, which is contractible because it is nonempty and admits pairwise coproducts.
\end{proof}

\clearpage
\section{Sheaves on the Moduli Stack of Broken Lines}\label{section.sheaves-on-broken}

In this section, we study sheaves on the moduli stack $\broken$ of broken lines. By definition,
a sheaf $\sheafF$ on $\broken$ with values in a category $\cC$ (or, more generally, an $\infty$-category $\cC$)
is a functor $\Pt( \broken)^{\op} \rightarrow \cC$ which satisfies a suitable descent condition (see Definition \ref{definition.sheaf-on-stack}).
Roughly speaking, we can think of a $\cC$-valued sheaf $\sheafF$ on $\broken$ as a rule which assigns to each
$S$-family of broken lines $\pi: L_{S} \to S$ a $\cC$-valued sheaf $\sheafF_{L_S}$ on the topological space $S$, depending functorially on $L_S$.
This functoriality guarantees that the stalk $\sheafF_{L_S,s} \in \cC$ depends only on the broken line $L_{s} \simeq \pi^{-1} \{s\}$, which is determined
(up to isomorphism) by the linearly ordered set $\pi_0 L_{s}^{\circ}$ of connected components of the non-fixed locus. The main result of this section
 asserts that $\sheafF$ can be recovered from its value on individual broken lines, together with some additional data of a purely combinatorial nature:

\begin{theorem}\label{theorem.broken-sheaves-vague}
Let $\cC$ be a compactly generated $\infty$-category, let $\Shv_{\cC}( \broken)$ denote the $\infty$-category of $\cC$-valued sheaves on
$\broken$ (Definition \ref{definition.sheaf-on-stack}), and let $\LinOrd$ denote the category of finite nonempty linearly ordered sets (Notation \ref{notation.linord}).
Then there is a canonical equivalence of categories
$$ \Shv_{\cC}( \broken) \xrightarrow{\sim} \Fun( \LinOrd, \cC ).$$
Moreover, if $L$ is a broken line and $I \simeq \pi_0(L^{\circ})$, then the diagram of $\infty$-categories
$$ \xymatrix{ \Shv_{\cC}( \broken) \ar[dr] \ar[rr]^{\sim} & & \Fun(\LinOrd, \cC ) \ar[dl] \\
& \cC & }$$
commutes up to canonical isomorphism, where the vertical maps are given by evaluation on $L$ and $I$, respectively. 
\end{theorem}

\subsection{Sheaves on Topological Spaces}\label{section.sheaves-on-spaces}

We begin with a review of the theory of sheaves with values in an $\infty$-category $\cC$. For a more detailed discussion, we refer the reader to \cite{htt}.

\begin{definition}\label{definition.sheaf}
For every topological space $S$, let $\calU(S)$ denote the partially ordered set of open subsets of $S$.
Let $\cC$ be an $\infty$-category. A {\it $\cC$-valued presheaf on $S$} is a functor
$\sheafF: \calU(S)^{\op} \rightarrow \cC$. We let $\PShv_{\cC}(S)$ denote the $\infty$-category
$\Fun( \calU(S)^{\op}, \cC)$ whose objects are $\cC$-valued presheaves on $S$.

Assume now that the $\infty$-category $\cC$ admits small limits. We will say that a $\cC$-valued presheaf $\sheafF$ on a topological space $S$ is a {\it $\cC$-valued sheaf on $S$} if,
for every collection of open sets $\{ U_{\alpha} \}$ in $S$ having union $U$, the canonical map
$$ \sheafF(U) \rightarrow \varprojlim_{V} \sheafF(V)$$ is an equivalence, where the limit is indexed by the category of all open sets $V \subseteq X$ which are contained in some $U_{\alpha}$.
We let $\Shv_{\cC}(S)$ denote the full subcategory of $\PShv_{\cC}(S)$ spanned by the $\cC$-valued sheaves on $S$.
\end{definition}

\begin{notation}
In the situation of Definition \ref{definition.sheaf}, suppose that the $\infty$-category $\cC$ admits filtered colimits. Let $\sheafF$ be a $\cC$-valued presheaf on
a topological space $S$. For each point $s \in S$, we let $\sheafF_{s}$ denote the direct limit $\varinjlim_{ s \in U } \sheafF(U)$; here the colimit is taken over the
filtered system of all open neighborhoods of $s$ in $S$. We will refer to $\sheafF_{s}$ as the {\it stalk of $\sheafF$ at the point $s$}.
\end{notation}

\begin{definition}\label{definition.sheafification}
Let $S$ be a topological space and let $\cC$ be an $\infty$-category which admits small limits. Suppose
we are given a morphism $\alpha: \sheafF \rightarrow \sheafG$ of $\cC$-valued presheaves on $S$.
We will say that {\it $\alpha$ exhibits $\sheafG$ as a sheafification of $\sheafF$} if the following conditions are satisfied:
\begin{itemize}
\item[$(i)$] The presheaf $\sheafG$ is a $\cC$-valued sheaf on $S$.
\item[$(ii)$] For every $\cC$-valued sheaf $\sheafH$ on $S$, composition with $\alpha$ induces a homotopy equivalence
$\bHom_{ \PShv_{\cC}(S) }( \sheafG, \sheafH ) \rightarrow \bHom_{ \PShv_{\cC}(S) }( \sheafF, \sheafH )$.
\end{itemize}
\end{definition}

It follows immediately from the definitions that if $\sheafF$ is a $\cC$-valued presheaf on a topological space $S$, then
a sheafification of $\sheafF$ is unique up to equivalence (in fact, up to a contractible space of choices) provided that it exists. We can also guarantee existence under the hypothesis that $\cC$ is compactly generated (see Section~5.5.7 of~\cite{htt}):

\begin{proposition}\label{proposition.existence-of-sheafification}
Let $\cC$ be a compactly generated $\infty$-category and let $\sheafF$ be a $\cC$-valued presheaf on a topological space $S$. Then:
\begin{itemize}
\item[$(a)$] There exists a morphism $\alpha: \sheafF \rightarrow \sheafG$ which exhibits $\sheafG$ as a sheafification of $\sheafF$.
\item[$(b)$] For every point $s \in S$, the map $\alpha$ induces an equivalence of stalks $\sheafF_{s} \rightarrow \sheafG_{s}$.
\end{itemize}
\end{proposition}

\begin{corollary}
Let $\cC$ be a compactly generated $\infty$-category and let $S$ be a topological space. Then the inclusion functor
$\Shv_{\cC}(S) \hookrightarrow \PShv_{\cC}(S)$ admits a left adjoint, which assigns to each $\cC$-valued presheaf $\sheafF$ its
sheafification.
\end{corollary}

\begin{remark}
For part $(a)$ of Proposition \ref{proposition.existence-of-sheafification}, one does not need the full strength of the assumption that $\cC$ is compactly generated: one can
sheafify $\cC$-valued pre-sheaves whenever $\cC$ is a presentable $\infty$-category. However, in general one cannot control the stalks of the resulting sheaves.
\end{remark}

\begin{warning}\label{warning.paracompact-hausdorff}
Let $\cC$ be a compactly generated $\infty$-category, let $S$ be a topological space, and let $\alpha: \sheafF \rightarrow \sheafG$ be a morphism
of $\cC$-valued presheaves on $S$. Consider the following assertions:
\begin{itemize}
\item[$(1)$] The morphism $\alpha$ exhibits $\sheafG$ as a sheafification of $\sheafF$.
\item[$(2)$] The presheaf $\sheafG$ is a $\cC$-valued sheaf on $S$, and the map $\alpha$ induces an equivalence of stalks $\sheafF_{s} \rightarrow \sheafG_{s}$ for each point $s \in S$.
\end{itemize}
It follows from Proposition \ref{proposition.existence-of-sheafification} that $(1) \Rightarrow (2)$. The reverse implication holds if $\cC$ is an ordinary category (or, more generally, an
$n$-category for $n < \infty$), or if $S$ is a paracompact Hausdorff space of finite covering dimension. However, condition $(2)$ does not imply $(1)$ in general: equivalences of
$\cC$-valued presheaves cannot be detected at the level of stalks. For our applications, this technicality will be irrelevant (since we are interested primarily in sheaves on
finite-dimensional manifolds). 
\end{warning}

\begin{remark}[Functoriality]
Let $f: S \rightarrow T$ be a map of topological spaces and let $\cC$ be an $\infty$-category. If $\sheafF$ is a $\cC$-valued presheaf on $S$, we let $f_{\ast} \sheafF$ denote the
$\cC$-valued presheaf on $T$ given by the formula $(f_{\ast} \sheafF)(U) = \sheafF( f^{-1}(U) )$. If $\cC$ admits small limits and $\sheafF$ is a $\cC$-valued sheaf on $S$,
then $f_{\ast} \sheafF$ is a $\cC$-valued sheaf on $T$. We can therefore regard the construction $\sheafF \mapsto f_{\ast} \sheafF$ as a functor
$f_{\ast}: \Shv_{\cC}( S) \rightarrow \Shv_{\cC}(T)$. We will refer to $f_{\ast}$ as the {\it direct image functor associated to $f$}.
\end{remark}

\begin{proposition}\label{proposition.existence-of-pullback}
Let $\cC$ be a compactly generated $\infty$-category and let $f: S \rightarrow T$ be a map of topological spaces. Then
the direct image functor $f_{\ast}: \Shv_{\cC}(S) \rightarrow \Shv_{\cC}(T)$ admits a left adjoint $f^{\ast}: \Shv_{\cC}(T) \rightarrow \Shv_{\cC}(S)$.
\end{proposition}

In the situation of Proposition \ref{proposition.existence-of-pullback}, we will refer to $f^{\ast}$ as the {\it pullback functor associated to $f$}.

\begin{notation}
In the situation of Proposition \ref{proposition.existence-of-pullback}, we will sometimes denote the image of
a sheaf $\sheafF \in \Shv_{\cC}(T)$ under the pullback functor $f^{\ast}$ by $\sheafF|_{S}$, and refer to it as 
{\it the restriction of $\sheafF$ to $S$}.
\end{notation}

\begin{example}\label{example.stalk-is-pullback}
Let $\cC$ be a compactly generated $\infty$-category, let $S$ be a topological space containing a point $s$, and let $i: \{ s \} \hookrightarrow S$
be the inclusion map. Then the pullback functor $i^{\ast}: \Shv_{\cC}(S) \rightarrow \Shv_{\cC}( \{s\} ) \simeq \cC$ assigns to each
sheaf $\sheafF \in \Shv_{\cC}(S)$ its stalk $\sheafF_{s}$ at the point $s$.
\end{example}

\begin{example}[Constant Sheaves]
Let $S$ be a topological space and let $\cC$ be a compactly generated $\infty$-category. Suppose we are given an object
$C \in \cC$, which we will identify with a $\cC$-valued sheaf on the one-point space $\ast$. We let
$\underline{C}$ denote the pullback $f^{\ast} C$, where $f: S \rightarrow \ast$ is the projection map.
We will refer to $\underline{C} \in \Shv_{\cC}(S)$ as the {\it constant sheaf with value $C$}.
\end{example}

\begin{example}
Let $f: S \rightarrow T$ be a continuous map of topological spaces and let $\cC$ be a compactly generated $\infty$-category.
For any $\cC$-valued sheaf $\sheafF$ on $T$, the stalks of the pullback sheaf $f^{\ast} \sheafF$ are given
by the formula $( f^{\ast} \sheafF)_{s} \simeq \sheafF_{ f(s) }$ (this follows from Example \ref{example.stalk-is-pullback} together with the
transitivity properties of the pullback construction).
\end{example}

\begin{example}
Let $S$ be a topological space, let $\cC$ be a compactly generated $\infty$-category, and let $U \subseteq S$
be an open subset. Then the restriction $\sheafF|_{U}$ is given by the formula $\sheafF|_{U}(V) = \sheafF(V)$
for $V \subseteq U$.
\end{example}

\begin{remark}[Extension by Zero]\label{remark.extension-by-zero-1}
Let $f: S' \rightarrow S$ be a local homeomorphism of topological spaces. Then, for any compactly generated $\infty$-category $\cC$,
the pullback functor $f^{\ast}: \Shv_{\cC}(S) \rightarrow \Shv_{\cC}(S')$ preserves small limit and colimits. It follows from the adjoint functor theorem
that $f^{\ast}$ admits a left adjoint $f_{!}: \Shv_{\cC}(S') \rightarrow \Shv_{\cC}(S)$. Moreover, this construction enjoys the following base change property:
for any pullback diagram of topological spaces
$$ \xymatrix{ T' \ar[r]^{g'} \ar[d]^{f'} & S' \ar[d]^{f} \\
T \ar[r]^{g} & S, }$$
the Beck-Chevalley transformation $f'_{!} \circ g'^{\ast} \rightarrow g^{\ast} f_{!}$ is an equivalence of functors
from $\Shv_{\cC}(S')$ to $\Shv_{\cC}(T)$. In particular, the functor $f_{!}$ is given at the level of stalks by the formula
$$ ( f_{!} \sheafF)_{s} = \coprod_{ f(s') = s } \sheafF_{ s' }.$$
\end{remark}

\begin{remark}[Functoriality in $\cC$]\label{remark.functoriality-in-C}
Suppose we are given a pair of adjoint functors $\Adjoint{F}{\cC}{\cD}{G}$, where $\cC$ and $\cD$ are compactly
generated $\infty$-categories. Then $G$ preserves small limits. It follows that for any topological space $S$, postcomposition with $G$
determines a functor $\Shv_{\cD}(S) \rightarrow \Shv_{\cC}(S)$. This functor admits a left adjoint
$\tilde{F}: \Shv_{\cC}(S) \rightarrow \Shv_{\cD}(S)$, which carries a sheaf $\sheafF \in \Shv_{\cC}(S)$ to the
sheafification of the presheaf
$$ (U \in \calU(S) ) \mapsto F( \sheafF(U) ).$$

Observe that if $f: S \rightarrow T$ is a continuous map of topological spaces, then we have a commutative diagram of $\infty$-categories
$$ \xymatrix{ \Shv_{\cD}(S) \ar[r]^{G} \ar[d]^{ f_{\ast} } & \Shv_{\cC}(S) \ar[d]^{ f_{\ast} } \\
\Shv_{\cD}(T) \ar[r]^{ G} & \Shv_{\cC}(T). }$$
It follows that the diagram of left adjoints
$$ \xymatrix{ \Shv_{\cD}(S) & \Shv_{\cC}(S) \ar[l]_{ \tilde{F} } \\
\Shv_{\cD}(T) \ar[u]_{ f^{\ast} } & \Shv_{\cC}(T) \ar[u]_{ f^{\ast} } \ar[l]_{ \tilde{F} } }$$
also commutes, up to canonical homotopy.
\end{remark}

\subsection{Sheaves on Topological Stacks}\label{section.sheaves-on-stacks}

We now adapt the ideas of \S \ref{section.sheaves-on-spaces} to the setting of topological stacks.

\begin{definition}\label{definition.presheaf-on-stack}
Let $X$ be a topological stack and let $\cC$ be an $\infty$-category. A {\it lax $\cC$-valued presheaf on $X$} is a functor
$\sheafF: \Pt(X)^{\op} \rightarrow \cC$. We let $\PShv_{\cC}^{\lax}( X ) = \Fun( \Pt(X)^{\op}, \cC)$ denote the $\infty$-category of lax $\cC$-valued presheaves on
$X$.
\end{definition}

\begin{warning}
Let $X$ be a topological space, viewed as a topological stack as in Example \ref{example.space-as-stack}. Then the $\infty$-category $\PShv_{\cC}(X)$ of
$\cC$-valued presheaves on $X$ (in the sense of Definition \ref{definition.sheaf}) is not equivalent to the $\infty$-category $\PShv_{\cC}^{\lax}(X)$ of
lax $\cC$-valued presheaves on $X$ (in the sense of Definition \ref{definition.presheaf-on-stack}), except in the trivial case where $X$ is empty.
The latter $\infty$-category is much larger (perhaps unreasonably so), since the $\infty$-category $\Pt(X) \simeq \Top_{/X}$ is not small.
\end{warning}

\begin{definition}\label{definition.lax-sheaf-on-stack}
Let $X$ be a topological stack, let $\cC$ be an $\infty$-category which admits small limits, and let $\sheafF$ be a lax $\cC$-valued presheaf on $X$.
For each object $\widetilde{S} \in \Pt(X)$ having image $S \in \Top$, we let $\sheafF_{ \widetilde{S} }$ denote the $\cC$-valued presheaf on $S$ given by the formula
$\sheafF_{\widetilde{S}}(U) = \sheafF( { \widetilde{S}}|_{U} )$. We will say that $\sheafF$ is a {\it lax $\cC$-valued sheaf on $X$} if, for each object $\widetilde{S} \in \Pt(X)$,
the presheaf $\sheafF_{ \widetilde{S}}$ is a $\cC$-valued sheaf on $S$. We let $\Shv_{\cC}^{\lax}( X )$ denote the full subcategory of $\PShv^{\lax}_{\cC}(X)$ spanned
by the lax $\cC$-valued sheaves on $S$.
\end{definition}

\begin{definition}\label{definition.sheafification.stacky}
Let $X$ be a topological stack and let $\cC$ be an $\infty$-category which admits small limits. Suppose
we are given a morphism $\alpha: \sheafF \rightarrow \sheafG$ of lax $\cC$-valued presheaves on $X$.
We will say that {\it $\alpha$ exhibits $\sheafG$ as a sheafification of $\sheafF$} if the following conditions are satisfied:
\begin{itemize}
\item[$(i)$] The lax presheaf $\sheafG$ is a lax $\cC$-valued sheaf on $X$.
\item[$(ii)$] For every lax $\cC$-valued sheaf $\sheafH$ on $X$, composition with $\alpha$ induces a homotopy equivalence
$\bHom_{ \PShv^{\lax}_{\cC}(X) }( \sheafG, \sheafH ) \rightarrow \bHom_{ \PShv^{\lax}_{\cC}(X) }( \sheafF, \sheafH )$.
\end{itemize}
\end{definition}

It follows immediately from the definitions that if $\sheafF$ is a lax $\cC$-valued presheaf on a topological stack $X$, then
a sheafification of $\sheafF$ is unique up to equivalence (in fact, up to a contractible space of choices) provided that it exists. For existence, we have the following analogue
of Proposition \ref{proposition.existence-of-sheafification}:

\begin{proposition}\label{proposition.existence-of-sheafification-stacky}
Let $\cC$ be a compactly generated $\infty$-category and let $\sheafF$ be a $\cC$-valued presheaf on a topological stack $X$. Then:
\begin{itemize}
\item[$(a)$] There exists a morphism $\alpha: \sheafF \rightarrow \sheafG$ which exhibits $\sheafG$ as a sheafification of $\sheafF$.
\item[$(b)$] Let $\alpha: \sheafF \rightarrow \sheafG$ be any morphism of lax $\cC$-valued presheaves on $X$. Then
$\alpha$ exhibits $\sheafG$ as a sheafification of $\sheafF$ (in the sense of Definition \ref{definition.sheafification.stacky}) if and only if, for each object $\widetilde{S} \in \Pt(X)$,
the induced map $\sheafF_{ \widetilde{S}} \rightarrow \sheafG_{ \widetilde{S}}$ exhibits $\sheafG_{ \widetilde{S}}$ as a sheafification of
$\sheafF_{ \widetilde{S} }$ (in the sense of Definition \ref{definition.sheafification}).
\end{itemize}
\end{proposition}

\begin{remark}
More informally, Proposition \ref{proposition.existence-of-sheafification-stacky} asserts that if $\sheafF$ is a lax $\cC$-valued presheaf on a topological stack
$X$, then we can sheafify each of the associated presheaves $\sheafF_{ \widetilde{S} }$ (in the sense of Definition \ref{definition.sheafification}) and assemble the resulting
objects into a lax $\cC$-valued sheaf on $X$, which also satisfies the requirements of Definition \ref{definition.sheafification.stacky}.
\end{remark}

\begin{corollary}
Let $\cC$ be a compactly generated $\infty$-category and let $X$ be a topological stack. Then the inclusion functor
$\Shv^{\lax}_{\cC}(X) \hookrightarrow \PShv^{\lax}_{\cC}(X)$ admits a left adjoint, which assigns to each $\cC$-valued presheaf $\sheafF$ its
sheafification.
\end{corollary}

\begin{notation}\label{notation.transition-maps}
Let $X$ be a topological stack, let $\sheafF$ be a $\cC$-valued presheaf on $X$, and let $\widetilde{f}: \widetilde{S}' \rightarrow \widetilde{S}$ be a morphism in
$\Pt(X)$, covering a continuous map of topological spaces $f: S' \rightarrow S$. For each open set $U \subseteq S$, we have a tautological
map $\widetilde{S}'|_{ f^{-1}(U) } \rightarrow \widetilde{S}|_{U}$ in the category $\Pt(X)$, which induces a map
$$ \sheafF_{ \widetilde{S}(U)} 
= \sheafF( \widetilde{S}|_{U} ) 
\rightarrow \sheafF( \widetilde{S}'|_{ f^{-1}(U) } )
= \sheafF_{ \widetilde{S}'}( f^{-1} U) = 
(f_{\ast} \sheafF_{ \widetilde{S}'})(U).$$
This construction depends functorially on $U$, and therefore determines a morphism
$\eta_{ \widetilde{f} }: \sheafF_{ \widetilde{S} } \rightarrow f_{\ast} \sheafF_{ \widetilde{S}'}$ in the $\infty$-category of presheaves $\PShv_{\cC}(S)$. 

If the $\infty$-category $\cC$ is compactly generated and $\sheafF$ is a lax $\cC$-valued sheaf on $X$, then we can identify
$\eta_{ \widetilde{f}}$ with a map $\eta_{ \widetilde{f} }^{\sharp}: f^{\ast} \sheafF_{ \widetilde{S}} \rightarrow \sheafF_{ \widetilde{S}'}$ in the
$\infty$-category of $\cC$-valued presheaves on $S'$. 
\end{notation}

\begin{definition}\label{definition.sheaf-on-stack}
Let $X$ be a topological stack, let $\cC$ be a compactly generated $\infty$-category, and let $\sheafF$ be a lax $\cC$-valued presheaf on $X$.
We will say that $\sheafF$ is a {\it $\cC$-valued sheaf on $X$} if it is a lax $\cC$-valued sheaf on $X$ (Definition \ref{definition.lax-sheaf-on-stack})
and, for every morphism $\widetilde{f}: \widetilde{S}' \rightarrow \widetilde{S}$ in $\Pt(X)$, the map
$\eta_{ \widetilde{f} }: \sheafF_{ \widetilde{S}} \rightarrow f_{\ast} \sheafF_{ \widetilde{S}'}$ is an equivalence in the $\infty$-category $\Shv_{\cC}(S)$. 
We let $\Shv_{\cC}(X)$ denote the full subcategory of $\PShv_{\cC}^{\lax}(X)$ spanned by the $\cC$-valued sheaves on $X$.
\end{definition}

\begin{warning}\label{warning.sheaf-two-notions}
Let $X$ be a topological space and let $\cC$ be a compactly generated $\infty$-category. We have now given two different definitions
for the $\infty$-category $\Shv_{\cC}(X)$ of $\cC$-valued sheaves on $X$:
\begin{itemize}
\item[$(a)$] Viewing $X$ as a topological space and applying Definition \ref{definition.sheaf}, obtain an $\infty$-category
$\Shv_{\cC}(X)$ which is a full subcategory of $\Fun( \calU(X)^{\op}, \cC)$.
\item[$(b)$] Viewing $X$ as a topological stack (via Example \ref{example.space-as-stack}) and applying Definition \ref{definition.sheaf-on-stack},
we obtain an $\infty$-category which is a full subcategory of $\Fun( \Pt(X)^{\op}, \cC )$; we denote this full subcategory by $\Shv_{\cC}'(X)$ in this Warning only.
\end{itemize}
The $\infty$-categories $\Shv_{\cC}(X)$ and $\Shv'_{\cC}(X)$ are not isomorphic. However, they are canonically equivalent: composition
with the inclusion functor $\calU(X) \hookrightarrow \Pt(X) = \Top_{/X}$ yields a forgetful functor $\theta: \Shv'_{\cC}(X) \rightarrow \Shv_{\cC}(X)$,
given by $\sheafF \mapsto \sheafF_{ \widetilde{X} }$ where $\widetilde{X} \in \Pt(X)$ is the object given by the identity map $\id: X \rightarrow X$.
The functor $\theta$ has a homotopy inverse, which assigns to each sheaf $\sheafG \in \Shv_{\cC}(X)$ the functor
$$ \overline{\sheafG}: \Pt(X)^{\op} \rightarrow \cC \quad \quad \overline{\sheafG}( f: S \rightarrow X) = (f^{\ast} \sheafG )(S).$$
\end{warning}

\begin{remark}[Stalks of Lax Sheaves]
Let $\cC$ be a compactly generated $\infty$-category, let $X$ be a topological stack, and let $\sheafF$ be a lax $\cC$-valued sheaf on $X$.
For every object $\widetilde{S} \in \Pt(X)$ and every point $s$ of the underlying topological space $S$, we can consider the stalk
$\sheafF_{ \widetilde{S}, s}$. If $\widetilde{f}: \widetilde{S} \rightarrow \widetilde{S}'$ is a morphism in $\Pt(X)$, then the construction of Notation
\ref{notation.transition-maps} determines a canonical map $\sheafF_{ \widetilde{S}', f(s) } \simeq
(f^{\ast} \sheafF_{ \widetilde{S}'})_{s} \rightarrow \sheafF_{ \widetilde{S}, s}$.
If $\sheafF$ is a $\cC$-valued sheaf on $X$, then this map is an equivalence: in other words, the stalk $\sheafF_{ \widetilde{S}, s}$ depends
only on the underlying map $\{s\} \rightarrow X$, and not on the ``ambient space'' $S$. Moreover, the converse is {\em almost} true:
if a lax $\cC$-valued sheaf $\sheafF$ has the property that each of the maps
$\sheafF_{ \widetilde{S}', f(s) } \rightarrow \sheafF_{ \widetilde{S}, s}$ is an equivalence, then the map
$\eta_{ \widetilde{f}}^{\sharp}: f^{\ast} \sheafF_{ \widetilde{S} } \rightarrow \sheafF_{ \widetilde{S}'}$ induces an equivalence on stalks, and is
therefore an equivalence if either $\cC$ is an ordinary category or $S$ is a paracompact Hausdorff space of finite covering dimension. (See Warning~\ref{warning.paracompact-hausdorff}.)
\end{remark}

\begin{remark}[Functoriality]\label{remark.pullback-of-lax-sheaves}
Let $f: X \rightarrow Y$ be a morphism of topological stacks (see Remark \ref{remark.topstack}). For any compactly generated $\infty$-category $\cC$,
composition with $f$ determines functors
$$ \PShv_{\cC}^{\lax}(Y) \rightarrow \PShv_{\cC}^{\lax}(X) \quad \Shv_{\cC}^{\lax}(Y) \rightarrow \Shv_{\cC}^{\lax}(X) \quad 
\Shv_{\cC}(Y) \rightarrow \Shv_{\cC}(X).$$
We will denote each of these functors by $f^{\ast}$. By means of this observation, we can regard the constructions
$X \mapsto \PShv_{\cC}^{\lax}(X), \Shv_{\cC}^{\lax}(X), \Shv_{\cC}(X)$ as functors from the $2$-category
$\TopStack$ of topological stacks (Remark \ref{remark.topstack}) to the $\infty$-category $\widehat{ \Cat}_{\infty}$ of 
(not necessarily small) $\infty$-categories.
\end{remark}

\begin{remark}[Extension by Zero]\label{remark.extension-by-zero-2}
Let $f: X \rightarrow Y$ be a morphism of topological stacks. Assume that $f$ is representable by local homeomorphisms: that is,
for every object $\widetilde{S} \in \Pt(Y)$ with underlying topological space $S$, the fiber product
$S_X = X \times_{Y} S$ is (representable by) a topological space and the projection map $\pi: S_{X} \rightarrow S$ is a local homeomorphism.
In this case, the pullback functor $f^{\ast}: \Shv_{\cC}^{\lax}(Y) \rightarrow \Shv_{\cC}^{\lax}(X)$ admits a left adjoint
$f_{!}: \Shv_{\cC}^{\lax}(X) \rightarrow \Shv_{\cC}^{\lax}(Y)$, which is given concretely by the formula
$$ (f_{!} \sheafF)_{ \widetilde{S}} = \pi_{!} (\sheafF_{ \widetilde{S}_X } ),$$ where $\widetilde{S}_{X}$ denotes the object of $\Pt(X)$ given by the projection map $S_{X} \rightarrow X$. 
Note that if $\sheafF$ belongs to $\Shv_{\cC}(X)$, then $f_{!} \sheafF$ belongs to $\Shv_{\cC}(Y)$ (this follows from compatibility of extension-by-zero with base change;
see Remark \ref{remark.extension-by-zero-1}).
\end{remark}

\begin{remark}
Let $f: X \rightarrow Y$ be a map of topological spaces and let $\cC$ be a compactly generated $\infty$-category.
Then the pullback functor $f^{\ast}: \Shv_{\cC}(Y) \rightarrow \Shv_{\cC}(X)$ of Proposition \ref{proposition.existence-of-pullback}
agrees (under the identification provided in Warning \ref{warning.sheaf-two-notions}) with the pullback functor defined in
Remark \ref{remark.pullback-of-lax-sheaves}, if we identify $X$ and $Y$ with the associated topological stacks.
\end{remark}

\begin{remark}\label{remark.stays-a-sheaf}
Suppose we are given a pair of adjoint functors $\Adjoint{F}{\cC}{\cD}{G}$, where $\cC$ and $\cD$ are compactly
generated. For any topological stack $X$, postcomposition with $G$ determines a forgetful functor 
$\Shv^{\lax}_{\cD}(X) \rightarrow \Shv^{\lax}_{\cC}(X)$. This functor admits a left adjoint
$\widetilde{F}: \Shv^{\lax}_{\cC}(X) \rightarrow \Shv^{\lax}_{\cC}(X)$. Concretely, the functor $\widetilde{F}$ carries
a lax sheaf $\sheafF \in \Shv^{\lax}_{\cC}(X)$ to the sheafification of the lax presheaf given by
$$ (\widetilde{S} \in \Pt(X) ) \mapsto F( \sheafF( \widetilde{S} ) ).$$
Note that if $\sheafF$ belongs to $\Shv_{\cC}(X)$, then $\widetilde{F}(\sheafF)$ belongs to
$\Shv_{\cD}(X)$ (see Remark \ref{remark.functoriality-in-C}).
\end{remark}

\subsection{Constructible Sheaves on \texorpdfstring{$\broken^{I}$}{Broken-I}}

Recall that the moduli stack $\broken$ of broken lines can be presented as colimit
$\varinjlim \broken^{I}$, where each $\broken^{I}$ is (representable by) a topological space
and the colimit is taken over all linearly preordered sets $I$ (Theorem \ref{leftor}). Consequently,
a $\cC$-valued sheaf $\sheafF$ on $\broken$ can be recovered from its restrictions $\sheafF|_{\broken^{I}}$
to the topological spaces $\broken^{I}$. We now articulate a special property enjoyed by sheaves
of the form $\sheafF|_{\broken^{I}}$: they are {\em constructible} with respect to a certain stratification
of the topological space $\broken^{I}$ (Example \ref{example.automatic-constructibility}).

\begin{construction}[The Stratification of $\broken^{I}$]\label{construction.stratify}
Let $I$ be a nonempty finite set equipped with a linear preordering $\leq_{I}$. 
We will say that an equivalence relation $E$ on $I$ is {\it convex} if the $E$-equivalence classes of $I$
are convex: that is, if for every triple of elements $i \leq_{I} j \leq_{I} k$ satisfying $i E k$, we also have $i E j$ and $j E k$.
We let $\Conv(I)$ denote the collection of all convex equivalence relations on $I$, partially ordered by refinement.

In what follows, we will identify $\broken^{I}$ with the topological space $\Rep(I, \BR^{+})$ of
functors $\alpha: I \rightarrow \BR^{+}$ (see Notation \ref{preorderrep}). For each convex equivalence relation $E$ on $I$, we define subsets
$K_{E} \subseteq U_{E} \subseteq \broken^{I}$ by the formulae
$$U_{E} = \{ \alpha \in \Rep(I, \BR^{+} ): (i E j) \Rightarrow (\alpha(i,j) < \infty) \}$$
$$K_{E} = \{ \alpha \in \Rep(I, \BR^{+} ): (i E j) \Leftrightarrow (\alpha(i,j)< \infty) \}.$$
The following properties are easily verified:
\begin{itemize}
\item[$(i)$] Each $U_{E}$ is an open subset of $\broken^{I}$. Moreover, the construction
$E \mapsto U_{E}$ is order-reversing: if $E \subseteq E'$, then $U_{E'} \subseteq U_{E}$.

\item[$(ii)$] Each $K_{E}$ is a closed subset of $U_{E}$. More precisely, $K_{E}$ is the closed
subset of $U_{E}$ complementary to the union $\bigcup_{E \subsetneq E'} U_{ E'}$.

\item[$(iii)$] Each $K_{E}$ is homeomorphic to Euclidean space of dimension $| I | - | I / E |$
(in particular, it is contractible).

\item[$(iv)$] Each point $\alpha \in \broken^{I}$ belongs to exactly one of the sets $\{ K_{E} \}_{E \in \Conv(I) }$.
\end{itemize}
In particular, the sets $\{ K_{E} \}_{E \in \Conv(I) }$ determine a stratification of the topological space
$\Rep(I, \BR^{+} )$, indexed by the partially ordered set $\Conv(I)$.
\end{construction}

\begin{definition}
Let $I$ be a nonempty finite linearly preordered set, let $\cC$ be a compactly generated $\infty$-category, and let
$\sheafF$ be a $\cC$-valued sheaf on the topological space $\broken^{I}$. We will say that $\sheafF$ is {\it constructible} if the
restriction $\sheafF|_{ K_{E} }$ is constant for each $E \in \Conv(I)$.
We let $\Shv^{c}_{\cC}( \broken^{I} )$ denote the full subcategory of $\Shv_{\cC}( \broken^{I} )$ spanned
by the constructible sheaves on $\broken^{I}$.
\end{definition}

\begin{example}\label{example.automatic-constructibility}
Let $\cC$ be a compactly generated $\infty$-category and let $\sheafF$ be a $\cC$-valued sheaf on the topological stack
$\broken$. Then, for every finite nonempty linearly preordered set $I$, the sheaf $\sheafF_{ \widetilde{\Rep}(I, \BR^{+} )} \in
\Shv_{\cC}( \broken^{I} )$ is constructible. In other words, for every convex equivalence
relation $E \in \Conv(I)$, the restriction $\sheafF_{ \widetilde{\Rep}(I, \BR^{+} )} |_{ K_{E} }$ is constant. This follows
from the observation that the family of broken lines
$$ K_{E} \times_{ \Rep(I, \BR^{+} ) } \widetilde{\Rep}(I, \BR^{+} ) \rightarrow K_{E}$$
is constant: that is, it is equivalent to a product $K_{E} \times L$ for some fixed broken line $L$ (namely, a concatenation
of copies of $[ - \infty, \infty ]$ indexed by the linearly ordered set $I / E$). Consequently, the map $K_E \to \broken$ classifying this family of broken lines factors through the projection $K_E \rightarrow \ast$.
\end{example}

\begin{proposition}\label{proposition.constructible-characterization}
Let $I$ be a finite nonempty linearly preordered set and let $\cC$ be a compactly generated $\infty$-category. Then the construction
$$ (\sheafF \in \Shv_{ \cC}( \broken^{I} ) ) \mapsto \{ \sheafF(U_E) \}_{E \in \Conv(I) }$$ restricts to
an equivalence of $\infty$-categories
$$ \Shv_{\cC}^{c}( \broken^{I} ) \rightarrow \Fun( \Conv(I), \cC ).$$
\end{proposition}

\begin{proof}
Let $\mathcal{J}$ denote the $\infty$-category of exit paths associated to the stratification of Construction \ref{construction.stratify}.
By virtue of Theorem A.9.3 of~\cite{HA}, it will suffice to show that the projection map $\mathcal{J} \rightarrow \Conv(I)$
is an equivalence of $\infty$-categories. Fix points $\alpha, \alpha' \in \broken^{I}$, having images $E, E' \in \Conv(I)$ with $E \leq E'$;
we wish to show that the mapping space $\bHom_{ \calJ}( \alpha, \alpha' )$ is contractible. Unwinding the definitions, we can identify
$\bHom_{\calJ}( \alpha, \alpha' )$ with the Kan complex whose $n$-simplices are given by continuous maps
$f: \Delta^{n} \times [0,1] \rightarrow \broken^{I}$ having the property that for each $x \in \Delta^{n}$,
we have 
$$ f(x,0) = \alpha \quad \quad f(x,1) = \alpha' \quad \quad f(x,t) \in K_{E'} \text{ for $t > 0$.}$$
It is easy to see that this Kan complex is contractible: writing $I = \{ i_0 \leq i_1 \leq \cdots \leq i_n \}$, 
the construction 
$$ (\beta \in \Rep(I, \BR^{+} )) \mapsto ( \log( \beta(i_0, i_1) ), \ldots, \log( \beta(i_{n-1}, i_n) )$$
restricts to homeomorphisms of both $K_{E'}$ and $\{ \alpha \} \cup K_{E'}$ with convex subsets of $\RR_{\leq 0}^{n}$.
\end{proof}

\begin{remark}\label{remark.constructible-evaluation}
Let $\cC$ be a compactly generated $\infty$-category and let $\sheafF$ be a $\cC$-valued sheaf on the topological stack
$\broken$. Then, for every finite nonempty linearly ordered set $I$, the restriction map
$$ \sheafF( \broken^{I} ) \rightarrow \sheafF( \{ \alpha \} )$$
is an equivalence, where $\{ \alpha \}$ denotes the point of $\broken^{I} \simeq \Rep(I, \BR^{+} )$
given by the formula
$$\alpha(i,j) = \begin{cases} 0 & \text{ if } i =_{I} j \\
\infty & \text{ otherwise. } \end{cases}$$
This follows by applying Proposition \ref{proposition.constructible-characterization} to the sheaf $\sheafF_{ \widetilde{\Rep}(I, \BR^{+} )}$, which
is constructible by virtue of Example \ref{example.automatic-constructibility}.
\end{remark}

\begin{remark}[Functoriality]
Let $f: I \rightarrow J$ be an essentially surjective morphism of finite nonempty linearly preordered sets. 
If $E$ is a convex equivalence relation on $J$, we let
$\overline{E}$ denote the convex equivalence relation on $I$ defined by $(i \overline{E} i') \Leftrightarrow ( f(i) E f(i') )$.
Note that composition in $f$ determines a map $\broken^{J} \rightarrow \broken^{I}$ which
carries each stratum $K_{E} \subseteq \broken^{J}$ into the corresponding stratum $K_{ \overline{E} } \subseteq \broken^{I}$.
Consequently, for any compactly generated $\infty$-category $\cC$, the pullback functor
$\Shv_{ \cC}( \broken^{I} ) \rightarrow \Shv_{\cC}( \broken^{J} )$ carries constructible sheaves on
$\broken^{I}$ to constructible sheaves on $\broken^{J}$. Moreover, the pullback functor on constructible sheaves fits into a commutative diagram of $\infty$-categories
$$ \xymatrix{ \Shv^{c}_{ \cC}( \broken^{I} ) \ar[r] \ar[d] & \Shv^{c}_{\cC}( \broken^{J} ) \ar[d] \\
\Fun( \Conv(I), \cC) \ar[r] & \Fun( \Conv(J), \cC ), }$$
where the vertical maps are the equivalences of Proposition \ref{proposition.constructible-characterization} and the bottom horizontal map
is given by precomposition with the map $$\Conv(J) \rightarrow \Conv(I) \quad \quad E \mapsto \overline{E}.$$
\end{remark}

\subsection{Classification of Sheaves}

We now give a more precise formulation of Theorems \ref{theorem.mainB} and~\ref{theorem.broken-sheaves-vague}. Let $\LinOrd$ denote the category of finite nonempty linearly ordered sets, with
morphisms given by surjections (Notation \ref{notation.linord}).
For every object $I \in \LinOrd$, let $\widetilde{\Rep}(I, \BR^{+} ) \rightarrow \Rep(I, \BR^{+} ) \simeq \broken^{I}$
be the family of broken lines given in Construction \ref{makelines}. By virtue of Remark \ref{remark.functor}, the construction $I \mapsto \widetilde{\Rep}(I, \BR^{+} )$ determines
a functor $\LinOrd \rightarrow \Pt( \broken)^{\op}$. 

\begin{theorem}\label{theorem.broken-sheaves}
Let $\cC$ be a compactly generated $\infty$-category. Then composition with the functor $I \mapsto \widetilde{\Rep}( I, \BR^{+} )$ induces an equivalence of
$\infty$-categories $\Shv_{\cC}( \broken) \rightarrow \Fun( \LinOrd, \cC )$.
\end{theorem}

\begin{proof}[Proof of Theorem \ref{theorem.broken-sheaves-vague} from Theorem \ref{theorem.broken-sheaves}]
Let $L$ be a broken line and set $I = \pi_0( L^{\circ} )$, which we regard as a linearly ordered set. The only nontrivial point
is to verify that the diagram of $\infty$-categories
$$ \xymatrix{ \Shv_{\cC}( \broken) \ar[dr] \ar[rr]^{\sim} & & \Fun(\LinOrd, \cC ) \ar[dl] \\
& \cC & }$$
commutes up to isomorphism, where the horizontal map is the equivalence of Theorem \ref{theorem.broken-sheaves} and
the vertical maps are given by evaluation on $L$ and $I$, respectively. In other words, we must show that if
$\sheafF$ is a $\cC$-valued sheaf on $\broken$, then the value of $\sheafF$ on the family of broken lines
$\widetilde{\Rep}(I, \BR^{+} ) \rightarrow \Rep(I, \BR^{+} )$ is equivalent to the value of $\sheafF$ on $L$ (which we can
regard as a family of broken lines parametrized by a point). This is a restatement of Remark \ref{remark.constructible-evaluation}.
\end{proof}

We would like to deduce Theorem \ref{theorem.broken-sheaves} from Proposition \ref{proposition.constructible-characterization}, which concerns sheaves on
topological spaces rather than topological stacks. For this, we appeal to the following formal observation:

\begin{proposition}\label{proposition.formal-statement}
Let $X$ be a topological stack and suppose we are given a diagram $\{ X_{\alpha} \}$ in the $2$-category $\TopStack_{/X}$ which satisfies the following condition:
\begin{itemize}
\item[$(\ast)$] The induced map $\varinjlim X_{\alpha} \rightarrow X$ is an equivalence in the $\infty$-category of $\SSet$-valued sheaves
on the category $\Top$.
\end{itemize}
Then, for any compactly generated $\infty$-category $\cC$, the canonical maps
$$ \Shv_{\cC}^{\lax}(X) \rightarrow \varprojlim \Shv_{\cC}^{\lax}( X_{\alpha} ) \quad \quad
\Shv_{\cC}(X) \rightarrow \varprojlim \Shv_{\cC}( X_{\alpha} )$$
are equivalences of $\infty$-categories.
\end{proposition}

\begin{corollary}\label{corollary.not-yet-constructible}
Let $\cC$ be a compactly generated $\infty$-category. Then the construction
$\sheafF \mapsto \{ \sheafF_{ \widetilde{\Rep}(I, \BR^{+} ) } \}_{I \in \LinPreOrd}$ induces an equivalence of $\infty$-categories
$$ \Shv_{\cC}( \broken) \rightarrow \varprojlim_{I \in \LinPreOrd} \Shv_{\cC}( \broken^{I} ).$$
\end{corollary}

\begin{proof}
Combine Proposition \ref{proposition.formal-statement} with Theorem \ref{leftor}.
\end{proof}

Example~\ref{example.automatic-constructibility} shows that for any $\sheafF \in \Shv_{\cC}( \broken )$, the sheaves
$\sheafF_{ \widetilde{\Rep}(I, \BR^{+} )}$ are automatically constructible. We therefore obtain the following variant of Corollary~\ref{corollary.not-yet-constructible}:

\begin{corollary}\label{corollary.now-constructible}
Let $\cC$ be a compactly generated $\infty$-category. Then the construction
$\sheafF \mapsto \{ \sheafF_{ \widetilde{\Rep}(I, \BR^{+} ) } \}_{I \in \LinPreOrd}$ induces an equivalence of $\infty$-categories
$$ \Shv_{\cC}( \broken) \rightarrow \varprojlim_{I \in \LinPreOrd} \Shv^{c}_{\cC}( \broken^{I} ).$$
\end{corollary}

We now wish to combine Corollary \ref{corollary.now-constructible} with the characterization of constructible sheaves supplied by
Proposition \ref{proposition.constructible-characterization}. For this, we need an auxiliary construction.

\begin{notation}
We define a category $\overline{\LinPreOrd}$ as follows:
\begin{itemize}
\item The objects of $\overline{\LinPreOrd}$ are triples $(I, \leq_{I}, E)$, where $I$ is a nonempty finite set, $\leq_{I}$ is a linear preordering on $I$,
and $E$ is a convex equivalence relation on $I$ (see Construction \ref{construction.stratify}).

\item A morphism from $(I, \leq_{I}, E)$ to $(J, \leq_{J}, E' )$ in the category $\overline{\LinPreOrd}$ consists of a nondecreasing map
$f: I \rightarrow J$ with the property that $(i E i') \Rightarrow ( f(i) E' f(i') )$ (so that $f$ induces a map $I/E \rightarrow J/E'$).
\end{itemize}
The construction $(I, \leq_I, E) \mapsto (I, \leq_I)$ determines a forgetful functor 
	\begin{equation*}
	\theta: \overline{\LinPreOrd} \rightarrow \LinPreOrd.
	\end{equation*}
This functor is a Grothendieck fibration, whose fiber over an object $I \in \LinPreOrd$ can be identified with the partially ordered set
$\Conv(I)$ of Construction \ref{construction.stratify}.
\end{notation}

\begin{remark}\label{remark.open-sets-of-broken-lines}
For each object $(I, \leq_I, E)$ of $\overline{\LinPreOrd}$, we let $\widetilde{U}_{E}$ denote the fiber product
$U_{E} \times_{ \Rep(I, \BR^{+} ) } \widetilde{\Rep}(I, \BR^{+} )$, where $U_{E}$ is the open subset of
$\Rep(I, \BR^{+} )$ defined in Construction \ref{construction.stratify}. We note that the construction
$(I, \leq_{I}, E) \mapsto \widetilde{U}_{E}$ determines a functor $\overline{\LinPreOrd} \rightarrow \Pt(\broken)^{\op}$.
\end{remark}

\begin{remark}
The Cartesian fibration $\theta: \overline{\LinPreOrd} \rightarrow \LinPreOrd$ is classified by a functor
$\LinPreOrd \rightarrow \Cat_{\infty}^{\op}$, given concretely by $I \mapsto \Conv(I)$. It follows
that for any $\infty$-category $\cC$, the construction $I \mapsto \Fun( \Conv(I), \cC )$ determines
a functor $\LinPreOrd \rightarrow \Cat_{\infty}$. Moreover, Corollary 3.3.3.2 of \cite{htt}
supplies a fully faithful embedding 
\begin{eqnarray}\label{eqn.map}
\varprojlim_{I \in \LinPreOrd} \Fun( \Conv(I), \cC) \rightarrow \Fun( \overline{\LinPreOrd}, \cC), \end{eqnarray}
whose essential image is spanned by those functors which carry each $\theta$-Cartesian
morphism in $\overline{\LinPreOrd}$ to an equivalence in $\cC$. 
\end{remark}

If $\cC$ is a compactly generated $\infty$-category, then we can use Proposition \ref{proposition.constructible-characterization}
to identify the domain of the map (\ref{eqn.map}) with the limit $\varprojlim_{I \in \LinPreOrd} \Shv^{c}_{\cC}( \Rep(I, \BR^{+} ) )$.
Combining this observation with Corollary \ref{corollary.now-constructible}, we obtain the following:

\begin{corollary}\label{corollary.closer-closer}
Let $\cC$ be a compactly generated $\infty$-category. Then composition with the functor 
$$ \overline{\LinPreOrd} \rightarrow \Pt(\broken)^{\op} \quad \quad  (I, \leq_{I}, E) \mapsto \widetilde{U}_{E}$$
of Remark \ref{remark.open-sets-of-broken-lines} induces a fully faithful embedding
$$ T: \Shv_{\cC}( \broken) \rightarrow \Fun( \overline{\LinPreOrd}, \cC ),$$
whose essential image is spanned by those functors $F: \overline{\LinPreOrd} \rightarrow \cC$
which carry $\theta$-Cartesian morphisms of $\overline{\LinPreOrd}$ to equivalences in $\cC$.
\end{corollary}

\begin{proof}[Proof of Theorem \ref{theorem.broken-sheaves}]
Let $u: \LinOrd \rightarrow \overline{\LinPreOrd}$ be the functor given by $u(I, \leq_{I} ) = (I, \leq_{I}, =_{I} )$. We wish to show that for any compactly generated $\infty$-category
$\cC$, the composite functor
$$ \Shv_{\cC}( \broken) \xrightarrow{T} \Fun( \overline{\LinPreOrd}, \cC ) \xrightarrow{ \circ u} \Fun( \LinOrd, \cC )$$
is an equivalence of $\infty$-categories, where $T$ is the functor appearing in Corollary \ref{corollary.closer-closer}.

Note that the functor $u$ is a fully faithful embedding which admits a left adjoint $v: \overline{\LinPreOrd} \rightarrow \LinOrd$, given by
the formula $v(I, \leq_{I}, E) = ( I/E, \leq_{I/E} )$. Let $\Fun'( \overline{\LinPreOrd}, \cC )$ denote the full subcategory of $\Fun( \overline{\LinPreOrd}, \cC )$ spanned by those functors $F: \overline{\LinPreOrd} \rightarrow \cC$ which satisfy the following condition:
\begin{itemize}
\item[$(\ast)$] For every morphism $\alpha$ in $\overline{\LinPreOrd}$, if $v(\alpha)$ is an isomorphism in $\LinOrd$, then $F(\alpha)$ is an equivalence in $\LinPreOrd$.
\end{itemize}
Then condition $(\ast)$ is equivalent to the requirement that $F$ is a right Kan extension of its restriction $F|_{ \LinOrd}$, so the restriction functor
$\Fun'( \overline{\LinPreOrd}, \cC ) \rightarrow \Fun( \LinOrd, \cC)$ is an equivalence of $\infty$-categories (with homotopy inverse given by composition with $v$).
We now complete the proof by observing that $T$ restricts to an equivalence of $\infty$-categories
$\Shv_{\cC}(\broken) \simeq \Fun'( \overline{\LinPreOrd}, \cC)$, by virtue of Corollary \ref{corollary.closer-closer} (note that a morphism
$\alpha$ in $\overline{\LinPreOrd}$ has the property that $v(\alpha)$ is an isomorphism if and only if $\alpha$ is $\theta$-Cartesian, where
$\theta: \overline{\LinPreOrd} \rightarrow \LinPreOrd$ is the forgetful functor).
\end{proof}

\subsection{Lax Sheaves on \texorpdfstring{$\broken$}{Broken}}

Let $\cC$ be a compactly generated $\infty$-category and let $\sheafF, \sheafG \in \Shv_{\cC}(\broken)$ be $\cC$-valued sheaves on the moduli stack of broken lines.
It follows from Theorem \ref{theorem.broken-sheaves} that the space of maps $\bHom_{ \Shv_{\cC}}( \sheafF, \sheafG )$ can be recovered from the values of $\sheafF$ and
$\sheafG$ on the families of broken lines $\widetilde{\Rep}(I, \BR^{+} ) \rightarrow \Rep(I, \BR^{+} )$ given in Construction \ref{makelines}. In \S \ref{section.factorizable}, we will need a slight
extension of this result, which applies to certain classes of lax $\cC$-valued sheaves (see Definition \ref{definition.lax-sheaf-on-stack}). To formulate this extension, we need to introduce some terminology.

\begin{definition}
Let $X$ be a topological stack, let $\cC$ be a compactly generated $\infty$-category, and let $\sheafF \in \PShv_{\cC}^{\lax}( X )$
be a lax $\cC$-valued presheaf on $X$. We will say that $\sheafF$ is {\it homotopy invariant} if the following condition is satisfied:
\begin{itemize}
\item[$(\ast)$] Let $\widetilde{S}' \rightarrow \widetilde{S}$ be a morphism in the category $\Pt(X)$ whose underlying map of topological spaces
is the projection $\pi: S \times \RR \rightarrow S$, for some topological space $S$. Then the induced map
$\sheafF( \widetilde{S} ) \rightarrow \sheafF( \widetilde{S}' )$ is an equivalence in the $\infty$-category $\cC$.
\end{itemize}
\end{definition}

\begin{proposition}\label{proposition.automatic-homotopy-invariance}
Let $X$ be a topological stack, let $\cC$ be a compactly generated $\infty$-category, and let $\sheafF \in \Shv_{\cC}(X)$ be a $\cC$-valued sheaf on $X$.
Then $\sheafF$ is homotopy invariant.
\end{proposition}

\begin{proof}
Let $\widetilde{S}$ be an object of the category $\Pt(X)$ having underlying topological space $S$, and let $\sheafF_{ \widetilde{S} } \in \Shv_{\cC}(S)$ be as in Definition
\ref{definition.lax-sheaf-on-stack}. Then the projection map $\pi: S \times \RR \rightarrow S$ admits an essentially unique lift to a map
$\widetilde{S}' \rightarrow \widetilde{S}$ in the category $\Pt(X)$. Our assumption that $\sheafF$ belongs to $\Shv_{\cC}(X)$ guarantees that
we can identify $\sheafF_{ \widetilde{S}' }$ with the pullback $\pi^{\ast} \sheafF_{ \widetilde{S} }$. It will therefore suffice to show that
the canonical map $\sheafF_{ \widetilde{S} }(S) \rightarrow ( \pi^{\ast} \sheafF_{ \widetilde{S}} )( S \times \RR )$ is an equivalence. 
Equivalently, we wish to show that for every object $C \in \cC$, the induced map
$$ \theta: \bHom_{\cC}( C, \sheafF_{ \widetilde{S}}(S)) \rightarrow \bHom_{ \cC}( C, \pi^{\ast} \sheafF_{ \widetilde{S} }(S) )$$
is a homotopy equivalence. Since the $\infty$-category $\cC$ is compactly generated, we may assume without loss of
generality that the object $C \in \cC$ is compact. In this case, the construction
$$ (U \in \calU(S) ) \mapsto \bHom_{\cC}(C, \sheafF_{ \widetilde{S} }(U) )$$
determines a $\SSet$-valued sheaf $\sheafG$ on $S$, and $\theta$ can identified with the canonical map $\sheafG(S) \rightarrow ( \pi^{\ast} \sheafG)( S \times \RR )$,
which is an equivalence by virtue of Lemma A.2.9 of \cite{HA}.
\end{proof}

We can now formulate the main result of this section; it will be used in the proof of Theorem~\ref{theorem.even-more-precise}.

\begin{proposition}\label{proposition.homotopy-invariant-is-enough}
Let $\cC$ be a compactly generated $\infty$-category. Suppose we are given objects $\sheafF \in \Shv_{\cC}( \broken)$ and $\sheafG \in \Shv_{\cC}^{\lax}(\broken)$.
Let $$\widetilde{\Rep}: \LinOrd^{\op} \rightarrow \Pt(\broken) \quad \quad I \mapsto \widetilde{\Rep}( I, \BR^{+} )$$
denote the functor of Construction \ref{makelines}. If $\sheafG$ is homotopy invariant, then the restriction map
$$ \bHom_{ \Shv_{\cC}^{\lax}(\broken) }( \sheafF, \sheafG ) \rightarrow \bHom_{ \Fun( \LinOrd, \cC) }( \sheafF \circ \widetilde{\Rep}, \sheafG \circ \widetilde{\Rep} )$$
is a homotopy equivalence.
\end{proposition}

\begin{proof}
Let $\cC$ be a compactly generated $\infty$-category and let $$T: \Fun( \LinOrd, \cC) \rightarrow \Shv_{\cC}(\broken)$$ denote a homotopy inverse to the equivalence of
$\infty$-categories $\Shv_{\cC}(\broken) \rightarrow \Fun( \LinOrd, \cC)$ of Theorem \ref{theorem.broken-sheaves}. Suppose that $\sheafG \in \Shv_{\cC}^{\lax}( \broken)$ is homotopy-invariant.
Let us say that a functor $F: \LinOrd \rightarrow \cC$ is {\it good} if the canonical map
\begin{eqnarray*}
\theta_{F}: \bHom_{ \Shv_{\cC}^{\lax}(\broken) }( T(F), \sheafG ) & 
\rightarrow & \bHom_{ \Fun( \LinOrd, \cC) }( T(F) \circ \widetilde{\Rep}, \sheafG \circ \widetilde{\Rep} ) \\
& \simeq & \bHom_{ \Fun( \LinOrd, \cC) }( F, \sheafG \circ \widetilde{\Rep} ). \end{eqnarray*}
is a homotopy equivalence. We wish to prove that every object $F \in \Fun( \LinOrd, \cC)$ is good.

For each object $I \in \LinOrd$ and each object $C \in \cC$, let
$F_{I,C}: \LinOrd \rightarrow \cC$ denote the functor given by the formula
$$ F_{I, C}(J) = \coprod_{ I \twoheadrightarrow J} C,$$
where the coproduct is taken over all isomorphism classes of morphisms $I \twoheadrightarrow J$ in the category $\LinOrd$ (here $I$ is fixed). In other words, $F_{I,C}$ is the left Kan extension of the constant
functor $\{ I \} \rightarrow \{C\} \subseteq \cC$ along the inclusion $\{ I \} \hookrightarrow \LinOrd$. Note that the the $\infty$-category
$\Fun( \LinOrd, \cC)$ is generated under small colimits by objects of the form $F_{I,C}$. 
The functor $T$ preserves small colimits when regarded as a functor from $\Fun(\LinOrd, \cC)$ to $\Shv^{\lax}_{\cC}(\broken)$ (since
the full subcategory $\Shv_{\cC}(\broken) \subseteq \Shv_{\cC}^{\lax}(\broken)$ is closed under small colimits). It follows that the collection of
good functors $F \in \Fun( \LinOrd, \cC)$ is closed under small colimits. Consequently, to complete the proof, it will suffice to show that each of the functors
$F_{I,C} \in \Fun( \LinOrd, \cC)$ is good.

Let us henceforth regard the objects $I \in \LinOrd$ and $C \in \cC$ as fixed. Let $\broken_{I}$ denote the topological stack of Notation \ref{notation.broken-I}, let $u: \broken_{I} \rightarrow \broken$ be the tautological map, and let $\underline{C}$ denote the constant sheaf on $\broken_{I}$ taking the value $C$. We will prove the following:
\begin{itemize}
\item[$(\ast)$] Let $\sheafH \in \Shv^{\lax}_{\cC}( \broken)$ be homotopy invariant. Then the canonical map
$$\bHom_{ \Shv_{\cC}^{\lax}( \broken_I ) }( \underline{C}, u^{\ast} \sheafH ) \rightarrow \bHom_{\cC}( C, \sheafH( \widetilde{\Rep}(I, \BR^{+} ) ) )$$
is a homotopy equivalence.
\end{itemize}
Let us assume $(\ast)$ for the moment and use it to complete the proof of Proposition \ref{proposition.homotopy-invariant-is-enough}.
Lemma \ref{lemma.compare-broken-I} asserts that the map $u: \broken_{I} \rightarrow \broken$ is a local homeomorphism of topological stacks.
Let $u_{!}: \Shv_{\cC}^{\lax}( \broken_{I} ) \rightarrow \Shv_{\cC}^{\lax}( \broken )$ denote the ``extension-by-zero'' functor of Remark \ref{remark.extension-by-zero-2}.
Then $u_{!} \underline{C}$ belongs to $\Shv_{\cC}(\broken)$ (Remark \ref{remark.extension-by-zero-2}). Applying $(\ast)$ in the special case where $\sheafH \in \Shv_{\cC}(\broken)$
(noting that any such sheaf is homotopy invariant by Proposition \ref{proposition.automatic-homotopy-invariance}), we obtain a homotopy equivalence
\begin{align}
\bHom_{ \Shv_{\cC}( \broken) }( u_{!} \underline{C}, \sheafH) 
&\rightarrow \bHom_{\cC}( C, \sheafH( \widetilde{\Rep}(I, \BR^{+} ) ) ) \nonumber \\
&\simeq \bHom_{ \Fun( \LinOrd, \cC)}( F_{I,C}, \sheafH \circ \widetilde{\Rep} ). \nonumber
\end{align}
It follows that we can identify $u_{!} \underline{C}$ with the sheaf $T( F_{I,C} )$. Applying $(\ast)$ again in the case $\sheafH = \sheafG$, we conclude that
the functor $F_{I,C}$ is good, as desired.

It remains to prove $(\ast)$. Let $X_{\bullet}$ denote the simplicial topological space given by the nerve of the map $\Rep(I, \BR^{+} ) \rightarrow \broken_{I}$.
Then the composite map $X_{\bullet} \rightarrow \broken_{I} \xrightarrow{u} \broken$ allows us to lift $X_{\bullet}$ to a simplicial object
$\widetilde{X}_{\bullet}$ of the category $\Pt(\broken)$. It follows from Proposition \ref{lemma.structure-of-broken-I} that the mapping space
\begin{equation*}
\bHom_{ \Shv_{\cC}^{\lax}( \broken_I ) }( \underline{C}, u^{\ast} \sheafH )
\end{equation*}
can be realized as the totalization of the cosimplicial space
$E^{\bullet} =  \bHom_{\cC}( C, \sheafH( \widetilde{X}_{\bullet} ) )$. Consequently, $(\ast)$ is equivalent to the assertion that the map of spaces
$\Tot( E^{\bullet} ) \rightarrow E^{0}$ is a homotopy equivalence. To prove this, it will suffice to show that the cosimplicial space $E^{\bullet}$ is constant:
that is, that each of the iterated coface maps
$$ E^{0} = \bHom_{ \cC}( C, \sheafH( \widetilde{X}_0) ) \rightarrow \bHom_{\cC}( C, \sheafH( \widetilde{X}_k )) \simeq E^{k}$$
is a homotopy equivalence. This follows from our assumption that $\sheafH$ is homotopy invariant (note that $X_k$ can be identified
with the product of $X_{0} = \Rep(I, \BR^{+} )$ with the $k$th power of $\RR^{I}$, by virtue of Proposition \ref{lemma.structure-of-broken-I}).
\end{proof}

\clearpage
\section{Factorizable sheaves and nonunital associative algebras}\label{section.factorizable}

Let $\cC$ be a compactly generated $\infty$-category. Suppose that $\cC$ is equipped with a monoidal structure, whose underlying tensor product
$$ \otimes: \cC \times \cC \rightarrow \cC$$
preserves colimits separately in each variable. In this section, we study {\em factorizable} $\cC$-valued sheaves on $\broken$: that is,
sheaves $\sheafF \in \Shv_{\cC}( \broken )$ equipped with an isomorphism $m^{\ast} \sheafF \simeq \sheafF \boxtimes \sheafF$,
satisfying appropriate coherence conditions (here $m: \broken \times \broken \rightarrow \broken$ is the map given by concatenation of broken lines). 
Our goal in this section is to prove the main result of this paper, which asserts that the datum of a factorizable $\cC$-valued sheaf on $\broken$ is
equivalent to the datum of a nonunital associative algebra object of $\cC$ (Theorem \ref{theorem.mainC}; see \S \ref{maintheo} for a precise formulation). 

\begin{notation}\label{notation.otimes}
We assume throughout this section that the reader is familiar with the theory of monoidal $\infty$-categories, as developed in \cite{HA}.
We will be primarily interested in {\em nonunital} monoidal $\infty$-categories: that is, $\infty$-categories $\cC$ equipped with
a tensor product operation $\otimes: \cC \times \cC \rightarrow \cC$ which is coherently associative (but not necessarily unital).
Given nonunital monoidal $\infty$-categories $\cC$ and $\calD$, we let $\Fun^{\lax}( \cC, \calD )$ denote the
$\infty$-category of nonunital lax monoidal functors from $\cC$ to $\calD$, and we let $\Fun^{\otimes}( \cC, \calD )$ denote
the full subcategory of $\Fun^{\lax}( \cC, \calD )$ spanned by the nonunital monoidal functors from $\cC$ to $\calD$.
More informally: the objects of $\Fun^{\lax}( \cC, \calD)$ are functors $F: \cC \rightarrow \calD$ which are equipped
with natural transformations $u_{C,C'}: F(C) \otimes F(C') \rightarrow F(C \otimes C' )$ (compatible with the associativity constraints on $\cC$ and $\calD$, in an appropriate homotopy-coherent sense);
such an object belongs to $\Fun^{\otimes}(\cC, \calD)$ if each of the maps $u_{C,C'}$ is an equivalence in $\calD$.
\end{notation}

\begin{warning}
Our use of the superscript $\lax$ in Notation \ref{notation.otimes} (in connection with lax monoidal functors) is logically unrelated to our use of the same superscript in 
Definition \ref{definition.lax-sheaf-on-stack} (in connection with lax sheaves on topological stacks). These notions will make a brief simultaneous appearance in our proof of
Theorem \ref{theorem.even-more-precise}, but it should be clear in context which of the two meanings is relevant. 
\end{warning}

\subsection{Concatenation of Broken Lines}\label{section.monoid}

Recall that if $L$ and $L'$ are broken lines, the {\it concatenation} $L \star L'$ is obtained from the disjoint union $L \amalg L'$ by identifying
the final point of $L$ with the initial point of $L'$. We now extend this construction to families. Given an $S$-family of broken lines $L_S$
and a $T$-family of broken lines $L_{T}$, we will define an $(S \times T)$-family of broken lines $L_{S} \star L_{T}$ (see Construction \ref{construction.relative-concatenation} below)
whose fiber at a point $(s,t) \in S \times T$ is given by the concatenation $L_{s} \star L_{t}$. First, we need the following elementary observation:

\begin{proposition}\label{proposition.initial-section}
Let $\pi: L_S \rightarrow S$ be a family of broken lines over a topological space $S$. Let
$v_{\init}: S \rightarrow L_S$ be the function which assigns to each point $s \in S$
the initial point of the broken line $L_s$, and define $v_{\term}: S \rightarrow L_S$ similarly.
Then $v_{\init}$ and $v_{\term}$ are continuous.
\end{proposition}

\begin{proof}
The assertion is local on $S$, so we can assume that there exists a directed homeomorphism
$h: L_S \simeq S \times [0,1]$. In this case, the maps $v_{\init}$ and $v_{\term}$ are given by the formulae
$$ v_{\init}( x ) = h^{-1}( \pi(x), 0) \quad \quad v_{\term}(x) = h^{-1} ( \pi(x), 1).$$ \end{proof}

\begin{construction}[Concatenation]\label{construction.relative-concatenation}
Let $S$ and $T$ be topological spaces, and suppose we are given families of broken lines $\pi_S: L_S \rightarrow S$ and
$\pi_{T}: L_T \rightarrow T$. Let $v_{\term}: S \rightarrow L_S$ and $v_{\init}: T \rightarrow L_T$
be as in Proposition \ref{proposition.initial-section}. We define a new map:
$$ L_S \star L_T \rightarrow S \times T$$
as follows:
\begin{itemize}
\item As a topological space, $L_S \star L_T$ fits into a pushout diagram
$$ 
	\xymatrix{ 
	S \times T \ar[rr]^-{ v_{\term} \times \id_T } \ar[d]^-{ \id_{S} \times v_{\init} } 
		&& L_S \times T \ar[d] \\
	S \times L_T \ar[rr] 
		&& L_S \star L_T 
	}
$$

\item The projection map $L_S \star L_T \rightarrow S \times T$ is obtained by amalgamating the maps
$$ (\pi_S \times \id_T): L_S \times T \rightarrow S \times T$$
$$ (\id_{S} \times \pi_{T}): S \times L_T \rightarrow S \times T.$$
\end{itemize}

Note that there is a canonical action of the group $\RR$ on $L_S \star L_T$, which is uniquely determined by the requirement that the inclusion maps
$$ L_S \times T \hookrightarrow L_S \star L_T \hookleftarrow S \times L_T$$
are $\RR$-equivariant.
\end{construction}

\begin{proposition}\label{ocso}
Let $\pi_{S}: L_S \rightarrow S$ and $\pi_{T}: L_T \rightarrow T$ be families of broken lines over topological spaces $S$ and $T$, respectively.
the induced map $L_S \star L_T \rightarrow S \times T$ is a family of broken lines over the product $S \times T$ (where we equip
$L_S \star L_T$ with the $\RR$-action described in Construction \ref{construction.relative-concatenation}).
\end{proposition}

\begin{proof}
The assertion is local on $S$ and $T$. We may therefore assume without loss of generality that there exist directed homeomorphisms
$h: L_S \simeq S \times [0,1]$ and $h': L_T \simeq T \times [1,2]$. We combine $h$ and $h'$ to obtain
a directed homeomorphism $L_S \star L_T \simeq S \times T \times [0,2]$. Unwinding the definitions, we see that the fixed point locus
$(L_S \star L_T)^{\RR}$ is given by the pushout
$$ (L_S^{\RR} \times T) \amalg_{ S \times T} ( S \times L_{T}^{\RR} ),$$
and is therefore unramified over $S \times T$ (by virtue of the fact that both $L_S^{\RR} \times T$ and
$S \times L_T^{\RR}$ are unramified over $S \times T$).
\end{proof}

\begin{remark}\label{easystar}
Suppose we are given two families of broken lines $L_{S} \rightarrow S \leftarrow L'_{S}$, parametrized by the same topological space $S$.
In this case, we let $L_{S} \star_{S} L'_{S}$ denote the fiber product $(L_S \star L'_{S}) \times_{ S \times S} S$. This is an $S$-family of broken lines,
whose fiber over a point $s \in S$ is given by the concatenation $L_{s} \star L'_{s}$.
\end{remark}

Recall that $\Pt(\broken)$ is the category whose objects are pairs $(\pi: L_S \rightarrow S, \mu)$,
where $S$ is a topological space and $\mu: \RR \times L_S \rightarrow S$ exhibits $L_S$ as a family of broken lines over $S$
(see Construction \ref{construction.category-of-lines}). Construction \ref{construction.relative-concatenation} determines a functor
\begin{equation*}
 \star: \Pt(\broken) \times \Pt(\broken) \rightarrow \Pt(\broken)
\end{equation*}
$$ ( (L_S \rightarrow S), ( L_T \rightarrow T) ) \mapsto ( L_S \star L_T \rightarrow S \times T ).$$
Moreover, for every triple of families of broken lines $( L_S \rightarrow S ), ( L_T \rightarrow T), ( L_U \rightarrow U)$, we have
a canonical isomorphism 
\eqn\label{eqn.star-associativity}
	( L_S \star L_T) \star L_U \simeq L_S \star (L_T \star L_U)
\eqnd
in the category $\Pt(\broken)$:
both sides can be identified with the colimit of the diagram of topological spaces
$$ \xymatrix{ & S \times T \times U \ar[dl] \ar[dr] & & S \times T \times U \ar[dl] \ar[dr] & \\
L_S \times T \times U & & S \times L_T \times U & & S \times T \times L_U. }$$
It is easy to check that these isomorphisms satisfy the pentagon identity, leading to the following conclusion:

\begin{prop}
The natural isomorphism~\eqref{eqn.star-associativity} endows $( \Pt(\broken), \star)$ with the structure of
a {\em nonunital} monoidal category.
\end{prop}

\begin{remark}
Let $\rho: \Pt(\broken) \rightarrow \Top$ be the forgetful functor of Proposition \ref{proposition.cartesian-fibration}, given by
$\rho(\pi: L_S \rightarrow S) = S$. Then $\rho$ has the structure of a (nonunital) monoidal functor, where we equip $\Top$ with the monoidal structure
given by the Cartesian product. 

Proposition \ref{proposition.cartesian-fibration} asserts that $\rho$ is a fibration in groupoids, which is classified by a functor
from the category of topological spaces to the $2$-category of groupoids (given by $S \mapsto \broken(S)$). The nonunital monoidal structures
on $\Pt(\broken)$ and on $\rho$ encode the fact that the construction $S \mapsto \broken(S)$ is a {\it nonunital lax monoidal} functor:
in particular, for every pair of topological spaces $S$ and $T$, we have a canonical map $\broken(S) \times \broken(T) \rightarrow \broken(S \times T)$, given by
$( L_S, L_T) \mapsto (L_S \star L_T)$.
\end{remark}

We now use the monoidal structure on the category $\Pt(\broken)$ to introduce the class of sheaves we are interested in.

\begin{definition}\label{definition.factorizable-sheaf}
Let $\cC$ be a monoidal $\infty$-category. Assume that $\cC$ is compactly generated and that the tensor product functor
$$ \otimes: \cC \times \cC \rightarrow \cC$$
preserves small colimits separately in each variable. A {\it weakly factorizable $\cC$-valued sheaf on $\broken$} is a nonunital lax monoidal functor
$\sheafF: \Pt(\broken)^{\op} \rightarrow \cC$ which satisfies the following condition:
\begin{itemize}
\item[$(a)$] The underlying functor $\sheafF: \Pt(\broken)^{\op} \rightarrow \cC$ is a $\cC$-valued sheaf on $\broken$, in the sense of Definition \ref{definition.sheaf-on-stack}.
\end{itemize}
We call $\sheafF$ a {\em factorizable $\cC$-valued sheaf} if it satisfies the following additional condition:
\begin{itemize}
\item[$(b)$]\label{item.factorizing-condition} For every pair of broken lines $L$ and $L'$, the map $\sheafF(L) \otimes \sheafF(L') \rightarrow \sheafF(L \star L')$ (supplied by the lax monoidal structure on
the functor $\sheafF$) is an equivalence in the $\infty$-category $\cC$.
\end{itemize}
We let $\Shv^{\wfact}_{\cC}(\broken)$ denote the $\infty$-category of weakly factorizable $\cC$-valued sheaves on $\broken$ (which we regard as a full subcategory of
the $\infty$-category of all nonunital lax monoidal functors from $\Pt(\broken)^{\op}$ to $\cC$), and we let $$\Shv^{\fct}_{\cC}( \broken ) \subseteq \Shv^{\wfact}_{\cC}(\broken)$$ denote
the full subcategory spanned by the factorizable $\cC$-valued sheaves on $\broken$.
\end{definition}

\begin{remark}
The definition of a factorizable $\cC$-valued sheaf on $\broken$ does not mention the unit for the monoidal structure on $\cC$; consequently, Definition
\ref{definition.factorizable-sheaf} also makes sense when $\cC$ is a (compactly generated) {\em nonunital} monoidal $\infty$-category. 
\end{remark}

\begin{warning}
The terminology of Definition \ref{definition.factorizable-sheaf} is potentially misleading: Factorizability is a {\em structure} on a $\cC$-valued sheaf, rather than a property (so the term {\it factorized sheaf} might be more appropriate).
\end{warning}

\begin{warning}
Condition $(b)$ of factorizability only refers to the value of $\cF$ on broken lines---not on families thereof. In general, if $\cF$ is a factorizable sheaf and $L_S, L_T$ are families of broken lines, the map $\cF(L_S) \tensor \cF(L_T) \to \cF(L_S \star L_T)$ is not an equivalence (and in particular, $\cF$ is still a {\em lax} monoidal functor). 
\end{warning}

\subsection{Digression: The Twisted Arrow Category}

We now introduce a categorical construction which will be useful in our analysis of factorizable sheaves on the moduli stack of broken lines.

\begin{construction}[Twisted Arrows]\label{construction.twisted}
Let $\cC$ be a category. We define a new category $\Tw(\cC)$ as follows:
\begin{itemize}
\item The objects of $\Tw(\cC)$ are morphisms $f: C \rightarrow \overline{C}$ in the category $\cC$.
\item A morphism from $f: C \rightarrow \overline{C}$ to $g: D \rightarrow \overline{D}$ in $\Tw(\cC)$
is given by a pair of morphisms $u: C \rightarrow D$ and $v: \overline{D} \rightarrow \overline{C}$ satisfying $f = v \circ g \circ u$; that is,
by a commutative diagram
$$ \xymatrix{ C \ar[d]^{f} \ar[r] & D \ar[d]^{g} \\
\overline{C} & \overline{D} \ar[l] }$$
in the category $\cC$.
\end{itemize}
We refer to $\Tw(\cC)$ as the {\it twisted arrow category} of the category $\cC$. Note that we have forgetful functors
$\cC \leftarrow \Tw(\cC) \rightarrow \cC^{\op}$, given on objects by $C \mapsfrom (f: C \to \overline{C} ) \mapsto \overline{C}$.
\end{construction}

We will only be interested in the following special case of Construction \ref{construction.twisted}:

\begin{example}\label{example.twist}
Let $\LinOrd$ denote the category whose objects are nonempty finite linearly ordered sets and whose morphisms are monotone surjections (Notation \ref{notation.linord}).
Given an object $I \in \LinOrd$, the datum of a monotone surjection $I \rightarrow \overline{I}$ is equivalent to the datum of an equivalence relation 
$\simeq_{I}$ on the set $I$ which is {\em convex}, in the sense that $i \leq_{I} j \leq_{I} k$ and $i \simeq_{I} k$ implies $i \simeq_{I} j \simeq_{I} k$. 
This equivalence gives a more concrete description of the twisted arrow category $\Tw(\LinOrd)$:
\begin{itemize}
\item The objects of $\Tw(\LinOrd)$ are pairs $(I, \simeq_{I} )$, where $I \in \LinOrd$ and $\simeq_{I}$ is a convex equivalence relation on $I$.

\item A morphism from $(I, \simeq_{I})$ to $(J, \simeq_{J} )$ in $\Tw(\LinOrd)$ is a monotone surjection $f: I \rightarrow J$
having the following additional property: if $i, i' \in I$ satisfy $f(i) \simeq_{J} f(i')$, then $i \simeq_{I} i'$. Note that this condition
ensures the existence of a commutative diagram $$ \xymatrix{ I \ar[r] \ar[d] & J \ar[d] \\
I / \simeq_I & J / \simeq_{J} \ar[l] }$$
in the category $\LinOrd$.

\end{itemize}
Under this description, the forgetful functors $\LinOrd \leftarrow \Tw(\LinOrd) \rightarrow \LinOrd^{\op}$ are given
by $I \mapsfrom (I, \simeq_{I} ) \mapsto I / \simeq_{I}$.
\end{example}

\begin{notation}\label{remark.left-adjoint-to-forget}
Given an object $I \in \LinOrd$, we let $I^{\sharp}$ denote the object $(I, \simeq_{I} ) \in \Tw( \LinOrd)$, where
$\simeq_{I}$ is the {\em indiscrete} equivalence relation on $I$: that is, we have $i \simeq_{I} i'$ for every pair of elements $i, i' \in I$.
The construction $I \mapsto I^{\sharp}$ determines a fully faithful embedding $\LinOrd \hookrightarrow \Tw(\LinOrd)$, which is left
adjoint to the forgetful functor $\Tw( \LinOrd) \rightarrow \LinOrd$.
\end{notation}

We now introduce (nonunital) monoidal structures on the categories $\LinOrd$ and $\Tw( \LinOrd )$, given by a combinatorial
counterpart of the geometric concatenation introduced in \S \ref{section.monoid}.

\begin{construction}[Concatenation of Linearly Ordered Sets]\label{construction.concatenate-linear-orderings}
Let $I$ and $J$ be sets equipped with linear orderings $\leq_{I}$ and $\leq_{J}$, respectively.
We let $I \star J$ denote the disjoint union $I \amalg J$, which we equip with the linear ordering
$\leq_{I \star J}$ given by 
$$ (x \leq_{ I \star J} y) \Leftrightarrow ( x,y \in I \text{ and } x \leq_I y) \text{ or }
(x,y \in J \text{ and } x \leq_{J} y ) \text{ or } (x \in I \text{ and } y \in J).$$
We will refer to $I \star J$ as the {\it concatenation of $I$ and $J$}. 

Note that if $I$ and $J$ are finite and nonempty, then $I \star J$ has the same property. Consequently,
concatenation determines a functor
$$ \star: \LinOrd \times \LinOrd \rightarrow \LinOrd.$$
Moreover, we have evident isomorphisms $(I \star J) \star K \simeq I \star (J \star K)$, which endow
$\LinOrd$ with the structure of a nonunital monoidal category.
\end{construction}

\begin{remark}\label{remark.universal-property-linord}
The nonunital monoidal category $\LinOrd$ enjoys the following universal property: it is the {\em free} nonunital
monoidal $\infty$-category on a single generator (that is, the monoidal envelope of the nonunital associative $\infty$-operad,
in the sense of \cite{HA}). In particular, for every nonunital monoidal $\infty$-category $\cC$,
we have an equivalence of $\infty$-categories
$$ \Fun^{\otimes}( \LinOrd, \cC ) \simeq \Alg^{\nounit}(\cC),$$
where $\Fun^{\otimes}( \LinOrd, \cC)$ denotes the $\infty$-category of nonunital monoidal functors from $\LinOrd$ to
$\cC$, and $\Alg^{\nounit}(\cC)$ denotes the $\infty$-category of nonunital associative algebra objects of $\cC$.
Alternatively, one can take this to be the definition of the $\infty$-category $\Alg^{\nounit}(\cC)$.
\end{remark}

\begin{variant}[Concatenation in $\Tw(\LinOrd)$]\label{construction.concatenate-decorated}
By functoriality, the concatenation functor on the category $\LinOrd$ induces a concatenation functor on the twisted arrow category $\Tw(\LinOrd)$, which
we will also denote by $\star$. In terms of the description of the category $\Tw(\LinOrd)$ given in Example \ref{example.twist}, this functor is given concretely
by the formula
$$ (I, \simeq_{I} ) \star (J, \simeq_{J} ) = (I \star J, \simeq_{I \star J} ),$$
where the equivalence relation $\simeq_{I \star J}$ is given by
$$ (x \simeq_{I \star J} y) \Leftrightarrow ( x,y \in I \text{ and } x \simeq_{I} y ) \text{ or } (x,y \in J \text{ and } x \simeq_{J} y ).$$
Note that the operation $I, J \mapsto I \star J$ (together with the evident isomorphisms $(I \star J) \star K \simeq I \star (J \star K)$)
endow $\Tw(\LinOrd)$ with the structure of a nonunital monoidal category, compatible with the forgetful functors
$\LinOrd \leftarrow \Tw(\LinOrd) \rightarrow \LinOrd^{\op}$.
\end{variant}

\begin{warning}\label{warning.UL}
Let $\sharp: \LinOrd \rightarrow \Tw(\LinOrd)$ denote the left adjoint to the forgetful functor
(Notation \ref{remark.left-adjoint-to-forget}). Since the forgetful functor is a nonunital monoidal functor, $\sharp$ inherits the structure of a {\em colax nonunital monoidal functor}: in particular,
for every pair of objects $I, J \in \LinOrd$, we have a canonical map $( I \star J)^{\sharp} \rightarrow I^{\sharp} \star J^{\sharp}$.
Beware that this map is bijective on the underlying sets (which can be identified with the disjoint union $I \amalg J$ on both sides), but
is not an isomorphism in the category $\Tw(\LinOrd)$: the equivalence relation on $I^{\sharp} \star J^{\sharp}$ has two equivalence classes,
rather than one.
\end{warning}

\begin{notation}\label{notation.nought}
Let $\cC$ be an $\infty$-category and let $\Fun( \Tw(\LinOrd), \cC)$ denote the $\infty$-category of functors from $\Tw(\LinOrd)$ to $\cC$. We let $\Fun_0( \Tw(\LinOrd), \cC)$ denote the full subcategory of $\Fun( \Tw(\LinOrd), \cC )$ spanned by those functors
$F$ with the following property: for every object $(I, \simeq_{I} ) \in \Tw(\LinOrd)$, the  map
$F(I^{\sharp}) \rightarrow F(I, \simeq_{I})$ 
is an equivalence in $\cC$. Note that this condition is equivalent to the requirement that
$F$ is a left Kan extension of its restriction to the essential image of the fully faithful embedding $\sharp: \LinOrd \hookrightarrow \Tw(\LinOrd)$.
It follows that the restriction map $\Fun_0( \Tw(\LinOrd), \cC ) \rightarrow \Fun( \LinOrd, \cC )$ is an equivalence of $\infty$-categories.
\end{notation}

\begin{variant}\label{notation.Fun-Lin-pullbacks}
Let $\cC$ be a monoidal $\infty$-category. We let $\Fun^{\tensor}_0( \Tw(\LinOrd),\cC)$ and $\Fun^{\lax}_0( \Tw(\LinOrd),\cC)$ be the full subcategories
of the $\infty$-categories $\Fun^{\tensor}( \Tw(\LinOrd), \cC )$ and $\Fun^{\lax}( \Tw(\LinOrd), \cC)$ spanned by those objects for
which the underlying functor $F: \Tw(\LinOrd) \rightarrow \cC$ belongs to the full subcategory
$\Fun_0( \Tw(\LinOrd), \cC) \subseteq \Fun( \Tw(\LinOrd), \cC)$ of Notation \ref{notation.nought}.
By definition, we have a diagram of pullback squares
	\begin{equation*}
	\xymatrix{
	\Fun^{\tensor}_0( \Tw(\LinOrd) ,\cC) \ar[r] \ar[d]
		& \fun^{\tensor} (\Tw(\LinOrd),\cC) \ar[d]\\
	\Fun^{\lax}_0(\Tw(\LinOrd),\cC) \ar[r] \ar[d]
		& \Fun^{\lax} (\Tw(\LinOrd),\cC) \ar[d] \\
	\Fun_0(\Tw(\LinOrd),\cC) \ar[r]
		& \Fun(\Tw(\LinOrd),\cC).
	}
	\end{equation*}
\end{variant}

Since the functor $\sharp: \LinOrd \hookrightarrow \Tw(\LinOrd)$ is fully faithful, the forgetful functor $$\Tw(\LinOrd) \rightarrow \LinOrd$$
exhibits $\LinOrd$ as the $\infty$-category obtained from $\Tw(\LinOrd)$ by formally inverting all morphisms of the form
$I^{\sharp} \rightarrow (I, \simeq_I )$. From this description, we immediately deduce the following result:

\begin{proposition}\label{proposition.character-functor}
Let $\cC$ be a nonunital monoidal $\infty$-category. Then composition with the forgetful functor $\Tw(\LinOrd) \rightarrow \LinOrd$
induces a fully faithful embedding
$$ \Alg^{\nounit}( \cC ) \simeq \Fun^{\otimes}( \LinOrd, \cC ) \rightarrow \Fun^{\otimes}( \Tw(\LinOrd), \cC )$$
whose essential image is the full subcategory $\Fun^{\otimes}_{0}( \Tw(\LinOrd), \cC )$ of Variant \ref{notation.Fun-Lin-pullbacks}.
\end{proposition}

\begin{remark}
Proposition \ref{proposition.character-functor} admits a refinement: the category $\Tw(\LinOrd)$ is actually
{\em universal} among nonunital monoidal $\infty$-categories which admit a colax nonunital monoidal functor
from $\LinOrd$. In other words, for any nonunital monoidal $\infty$-category $\cC$, composition with the functor
$\sharp: \LinOrd \hookrightarrow \Tw(\LinOrd)$ induces an equivalence from $\Fun^{\otimes}( \Tw(\LinOrd), \cC)$ to the $\infty$-category of nonunital colax monoidal functors from
$\LinOrd$ into $\cC$. Since we will not need this fact, we leave details to the reader.
\end{remark}

\subsection{The Main Theorem}\label{maintheo}

In this section, we give a more precise formulation of Theorem \ref{theorem.mainC}. First, we need to understand the relationship between
the concatenation operations in the setting of broken lines (Construction \ref{construction.relative-concatenation}) and linearly ordered sets (Construction \ref{construction.concatenate-linear-orderings}).

\begin{construction}[Concatenation of Sections]\label{construction.compare}
Let $\pi: L_{S} \to S$ and $\pi': L_T \to T$ be families of broken lines, and let $I$ and $J$ be finite nonempty linearly ordered sets.
Suppose that $L_{S}$ is equipped with an $I$-section $\{ \sigma_i: S \to L_S \}_{i \in I}$, and that $L_{T}$ is equipped with
a $J$-section $\{ \tau_j: T \to L_{T} \}_{j \in J}$. In this case, we can equip the concatenation 
$L_{S} \star L_{T}$ with an $(I \star J)$-section $\{ \rho_{k}: S \times T \to L_{S} \star L_{T} \}_{k \in I \star J}$, given by the formula
$$ \rho_k(s,t) = \begin{cases} (\sigma_k(s), t) & \text{ if } k \in I \\
(s, \tau_k(t) ) & \text{ if } k \in J.\end{cases}$$
Here we abuse notation by identifying the products $L_{S} \times T$ and $S \times L_{T}$ with their images in $L_S \star L_T$.

Applying this observation in the universal example (where $\pi$ is the projection map $\widetilde{\Rep}( I, \BR^{+} ) \rightarrow \Rep(I, \BR^{+} )$
and $\pi'$ is the projection map $\widetilde{\Rep}( J, \BR^{+} ) \rightarrow \Rep(J, \BR^{+} )$), we obtain a morphism
$$ \widetilde{\Rep}(I, \BR^{+} ) \star \widetilde{\Rep}( J, \BR^{+} ) \rightarrow \widetilde{\Rep}(I \star J, \BR^{+} )$$
in the category $\Pt(\broken)$. These maps equip the construction $I \mapsto \widetilde{\Rep}(I, \BR^{+} )$ with
the structure of a nonunital lax monoidal functor $\LinOrd^{\op} \rightarrow \Pt(\broken)$; here
$\LinOrd$ is equipped with the nonunital monoidal structure given by concatenation of linearly ordered sets, and
$\Pt(\broken)$ is equipped with the nonunital monoidal structure given by concatenation of broken lines.
\end{construction}

Beware that the (nonunital) lax monoidal functor of Construction \ref{construction.compare} is not monoidal: given a pair of objects $I,J \in \LinOrd$,
we have a pullback square
$$ \xymatrix{ \widetilde{\Rep}(I, \BR^{+} ) \star \widetilde{\Rep}( J, \BR^{+} ) \ar[r] \ar[d] & \widetilde{\Rep}( I \star J, \BR^{+} ) \ar[d] \\
\Rep(I, \BR^{+} ) \times \Rep(J, \BR^{+} ) \ar[r] & \Rep(I \star J, \BR^{+} ) }$$
where the bottom vertical map is given by
$$ ( \alpha: I \to \BR^{+}, \beta: J \to \BR^{+} ) \mapsto (\gamma: I \star J, \BR^{+})$$
$$\gamma(x,y) = \begin{cases} \alpha(x,y) & \text{ if } x,y \in I \\
\beta(x,y) & \text{ if } x,y \in J \\
\infty & \text{ if } x \in I, y \in J. \end{cases}$$
This map is a closed embedding (but not a homeomorphism), whose image consists of those functors $\gamma: I \star J \rightarrow \BR^{+}$ satisfying
$\gamma(i,j) = \infty$ for $i \in I$ and $j \in J$. To remedy the situation, we introduce a variant of Construction \ref{construction.compare}:

\begin{construction}[The Functor $\Phi$]\label{construction.phi}
Let $(I, \simeq_I )$ be an object of the twisted arrow category $\Tw( \LinOrd )$ (see Example \ref{example.twist}),
and let $\Rep(I, \BR^{+} )$ be as in Notation \ref{preorderrep}. We let $\Phi( I, \simeq_I )$ denote the closed subset of $\Rep(I, \BR^{+} )$
consisting of those maps $\alpha: I \to \BR^{+}$ having the property that $\alpha(i,j) < \infty$ implies $i \simeq_{I} j$.
We let $\widetilde{\Phi}(I, \simeq_{I} )$ denote the object of $\Pt(\broken)$ given by the family of broken lines
$$ \Phi(I, \simeq_I) \times_{ \Rep(I, \BR^{+} ) } \widetilde{\Rep}( I, \BR^{+} ) \rightarrow \Phi(I, \simeq_I ).$$

Note that if $f: (I, \simeq_{I} ) \rightarrow (J, \simeq_J )$ is a morphism in the category $\Tw(\LinOrd)$, then the map
$\Rep(J, \BR^{+} ) \rightarrow \Rep(I, \BR^{+})$ carries the closed subset $\Phi(J, \simeq_J) \subseteq \Rep(J, \BR^{+} )$
to the closed subset $\Phi(I, \simeq_I) \subseteq \Rep(I, \BR^{+} )$. It follows that the construction $(I, \simeq_I) \mapsto \widetilde{\Phi}(I, \simeq_I )$
determines a functor
$$ \widetilde{\Phi}: \Tw( \LinOrd)^{\op} \rightarrow \Pt( \broken ).$$
\end{construction}

\begin{example}
Let $I$ be a nonempty finite linearly ordered set and let $I^{\sharp} \in \Tw(\LinOrd)$ be as in Notation \ref{remark.left-adjoint-to-forget}: that is,
the object obtained by equipping $I$ with the indiscrete equivalence relation. Then $\Phi( I^{\sharp} ) = \Rep(I, \BR^{+} ) \simeq \broken^{I}$.
\end{example}

\begin{example}
Let $I = [n] = \{ 0 < 1 < \cdots < n \}$ and let $\simeq_{I}$ be the discrete equivalence relation on $I$
(so that $i \simeq_{I} j$ if and only if $i = j$). Then $\Phi(I, \simeq_I)$ consists of a single point, and $\widetilde{\Phi}$ is
the ``$n$-times broken line'': that is, an $(n+1)$-fold concatenation of the standard (un)broken line $[ - \infty, \infty ]$ of Example \ref{example.standard-broken-line}.
\end{example}

Suppose we are given a pair of objects $(I, \simeq_I), (J, \simeq_J) \in \Tw( \LinOrd )$. Then the closed embedding
$$ \Rep(I, \BR^{+} ) \times \Rep(J, \BR^{+} ) \hookrightarrow \Rep(I \star J, \BR^{+} )$$
restricts to a homeomorphism
$$ \Phi(I, \simeq_I ) \times \Phi(J, \simeq_J) \xrightarrow{\sim} \Phi(I \star J, \simeq_{ I \star J} ).$$
It follows that we can identify $\widetilde{\Phi}( I \star J, \simeq_{ I \star J} )$ with the concatenation of
$\widetilde{\Phi}(I, \simeq_I)$ with $\widetilde{\Phi}(J, \simeq_J)$ in the category $\Pt(\broken)$, which proves the following:

\begin{proposition}\label{ols}
The functor $\widetilde{\Phi}: \Tw(\LinOrd)^{\op} \rightarrow \Pt(\broken)$ admits a nonunital monoidal structure, where
we equip $\Tw(\LinOrd)$ with the nonunital monoidal structure given by concatenation (Variant \ref{construction.concatenate-decorated}) and
$\Pt(\broken)$ with the nonunital monoidal structure described in \S \ref{section.monoid}.
\end{proposition}

We can now formulate the main result of this paper:

\begin{theorem}\label{theorem.precise-version}
Let $\cC$ be a compactly generated monoidal $\infty$-category, and assume that the tensor product functor $\otimes: \cC \times \cC \rightarrow \cC$ preserves small colimits
separately in each variable. Then:
\begin{itemize}
\item[$(1)$] For every factorizable sheaf $\sheafF \in \Shv_{ \cC}^{\fct}(\broken)$, the composition
$$ \Tw(\LinOrd) \xrightarrow{ \widetilde{\Phi} } \Pt(\broken)^{\op} \xrightarrow{\sheafF} \cC$$
is a nonunital monoidal functor.

\item[$(2)$] The functor
$$ \Shv_{\cC}^{\fct}( \broken ) \xrightarrow{ \circ \widetilde{\Phi} } \Fun^{\otimes}( \Tw(\LinOrd), \cC)$$
is fully faithful, and its essential image is the full subcategory 
$$\Fun^{\otimes}_{0}( \Tw(\LinOrd), \cC), \subseteq \Fun^{\otimes}( \Tw(\LinOrd), \cC )$$
described in Variant \ref{notation.Fun-Lin-pullbacks}.
\end{itemize}
\end{theorem}

From Theorem \ref{theorem.precise-version}, we easily deduce the result promised in the introduction to this paper:

\begin{proof}[Proof of Theorem \ref{theorem.mainC}]
Theorem \ref{theorem.precise-version} and Proposition \ref{proposition.character-functor} supply equivalences
$$\Alg^{\nounit}(\cC) \xrightarrow{\sim} \Fun^{\otimes}_{0}( \Tw(\LinOrd), \cC) \xleftarrow{\sim} \Shv_{\cC}^{\fct}( \broken ).$$
\end{proof}

We will deduce Theorem \ref{theorem.precise-version} from the following more general statement:

\begin{theorem}\label{theorem.more-precise-version}
Let $\cC$ be a compactly generated monoidal $\infty$-category, and assume that the tensor product functor $\otimes: \cC \times \cC \rightarrow \cC$ preserves small colimits
separately in each variable. 
Then composition with the nonunital lax monoidal functor $\widetilde{\Phi}: \Tw(\LinOrd) \rightarrow \Pt(\broken)^{\op}$
determines a fully faithful embedding $\Shv^{\wfact}_{\cC}( \broken ) \hookrightarrow \Fun^{\lax}( \Tw(\LinOrd), \cC )$, whose essential
image is the full subcategory $$\Fun^{\lax}_{0}( \Tw(\LinOrd), \cC) \subseteq \Fun^{\lax}( \Tw(\LinOrd), \cC)$$ described in
Variant \ref{notation.Fun-Lin-pullbacks}.
\end{theorem}

Theorem \ref{theorem.more-precise-version} is a consequence of a more general statement that we will formulate in \S \ref{section.daycon}
(Theorem \ref{theorem.even-more-precise}) and prove in \S \ref{section.bigproof}). 

\begin{proof}[Proof of Theorem \ref{theorem.precise-version} from Theorem \ref{theorem.more-precise-version}]
Let $\sheafF \in \Shv_{\cC}^{\wfact}( \broken)$ be a weakly factorizable $\cC$-valued sheaf on $\broken$, and let $F = \sheafF \circ \widetilde{\Phi}$ denote the associated lax monoidal functor
$\Tw(\LinOrd) \rightarrow \cC$. By virtue of Theorem \ref{theorem.more-precise-version}, it will suffice to show that the following
conditions are equivalent:
\begin{itemize}
\item[$(i)$] The functor $\sheafF$ is a factorizable $\cC$-valued sheaf on $\broken$. In other words, for every
pair of broken lines $L$ and $L'$, the canonical map $\sheafF(L) \otimes \sheafF(L') \rightarrow \sheafF(L \star L')$ is an equivalence in $\cC$.

\item[$(ii)$] The functor $F$ is nonunital monoidal: that is, for every pair of objects $(I, \simeq_{I} )$ and $(J, \simeq_J)$
in $\Tw(\LinOrd)$, the canonical map $F(I, \simeq_I) \otimes F(J, \simeq_J) \rightarrow F( I \star J , \simeq_{I \star J})$ is an equivalence in $\cC$.
\end{itemize}
Note that assertion $(i)$ is equivalent to the special case of assertion $(ii)$ where we assume that the equivalence relations $\simeq_I$ and $\simeq_J$ are discrete.
Conversely, suppose we are given an arbitrary pair of objects $(I, \simeq_{I} )$ and $(J, \simeq_{J} )$ of the category
$\Tw(\LinOrd)$. Set $I^{\flat} = (I, =_I)$ and $J^{\flat} = (J, =_{J} )$. (These are the discrete equivalence relations on $I$ and $J$; see Notation~\ref{linpreorder}.) Then we have a commutative diagram
$$\xymatrix{ F(I , \simeq_I) \otimes F(J, \simeq_J) \ar[r] \ar[d] & F(I \star J, \simeq_{I \star J}) \ar[d] \\
F(I^{\flat} ) \otimes F(J^{\flat} ) \ar[r] & F(I^{\flat} \star J^{\flat} )}$$
where the vertical maps are equivalences. If condition $(i)$ is satisfied, then the lower horizontal map is an equivalence, so that the upper horizontal map is an equivalence as well.
\end{proof}

\subsection{Digression: Day Convolution}\label{section.daycon}

In this section, we recall the theory of {\it Day convolution} in the setting of nonunital monoidal $\infty$-categories. For more details, we refer the reader to Section 2.2.6 of \cite{HA}.

\begin{construction}[The Day Convolution Product]\label{construction.daycon}
Let $\cC$ and $\calD$ be nonunital monoidal $\infty$-categories with tensor product functors $\otimes_{\cC}: \cC \times \cC \rightarrow \cC$ and $\otimes_{\calD}: \calD \times \calD \rightarrow \calD$.
Assume that $\calD$ is small and that $\cC$ admits small colimits. Given a pair of functors $F_0, F_1: \calD \rightarrow \cC$, we define a new functor $(F_0 \circledast F_1): \calD \rightarrow \cC$
by the formula
$$ (F_0 \circledast F_1)(D) = \varinjlim_{ D_0 \otimes_{\calD} D_1 \rightarrow D} F_0(C_0) \otimes_{\cC} F_1(C_1).$$
Here the colimit is taken over the $\infty$-category $(\calD \times \calD) \times_{\calD} \calD_{/D}$ parametrizing pairs of objects $D_0, D_1 \in \calD$ equipped
with a morphism $D_0 \otimes_{\calD} D_1 \rightarrow D$. We refer to $F_0 \circledast F_1$ as the {\it Day convolution product} of the functors $F_0$ and $F_1$.
\end{construction}

Under mild hypotheses, one can show that the Day convolution product underlies a nonunital monoidal structure on the $\infty$-category $\Fun( \calD, \cC )$. Moreover,
this monoidal structure can be characterized by a universal property:

\begin{theorem}\label{theorem.daycon}
Let $\cC$ and $\calD$ be nonunital monoidal $\infty$-categories. Assume that $\calD$ is small, that $\cC$ admits small colimits, and that the tensor product on $\cC$ preserves small colimits separately in each variable. Then there is a nonunital monoidal structure on the $\infty$-category $\Fun(\calD, \cC)$ with the following properties:
\begin{itemize}
\item[$(a)$] The underlying tensor product $\circledast: \Fun(\calD, \cC) \times \Fun(\calD, \cC) \rightarrow \Fun( \calD, \cC )$ is given by the Day convolution product of
Construction \ref{construction.daycon}.
\item[$(b)$] For every nonunital monoidal $\infty$-category $\mathcal{E}$, there is a canonical equivalence of $\infty$-categories
$$ \Fun^{\lax}( \mathcal{E}, \Fun(\calD, \cC) ) \simeq \Fun^{\lax}( \mathcal{E} \times \calD, \cC).$$
In particular, there is an equivalence of $\infty$-categories $\Alg^{\nounit}( \Fun(\calD, \cC) ) \simeq \Fun^{\lax}( \calD, \cC)$.
\end{itemize}
\end{theorem}

\begin{example}\label{example.daycon-combinatorial}
Let $\cC$ be a compactly generated monoidal $\infty$-category for which the tensor product $\otimes: \cC \times \cC \rightarrow \cC$ preserves small colimits
separately in each variable. Applying Theorem \ref{theorem.daycon} in the case $\calD = \Tw(\LinOrd)$, we obtain a nonunital monoidal structure on
the $\infty$-category $\Fun( \Tw(\LinOrd), \cC )$, whose tensor product $\circledast$ is given concretely by the formula
$$ (F_0 \circledast F_1)(I, \simeq_I) = \coprod F_0( I_0, \simeq_{I_0}) \otimes F_1( I_1, \simeq_{I_1} ).$$
Here the coproduct is taken over all decompositions $I = I_0 \amalg I_1$ into nonempty subsets, where $I_0$ is closed downward
under the linear order $\leq_{I}$ and invariant under the equivalence relation $\simeq_I$ (so that $I_1$ is closed upwards under $\leq_{I}$
and invariant under $\simeq_I$), and $\simeq_{I_0}$, $\simeq_{I_1}$ denote the equivalence relations on $I_0$ and $I_1$ obtained by
restricting the equivalence relation $\simeq_{I}$. Moreover, there is a canonical equivalence of $\infty$-categories
$$ \Alg^{\nounit}( \Fun( \Tw(\LinOrd), \cC) ) \simeq \Fun^{\lax}( \Tw(\LinOrd), \cC ).$$
Under this equivalence, the subcategory $\Fun^{\lax}_0( \Tw(\LinOrd), \cC ) \subseteq \Fun^{\lax}( \Tw(\LinOrd), \cC )$
of Variant \ref{notation.Fun-Lin-pullbacks} corresponds to the full subcategory of $\Alg^{\nounit}( \Fun( \Tw(\LinOrd), \cC) )$
spanned by those algebras for which the underlying functor $F: \Tw(\LinOrd) \rightarrow \cC$ belongs
to the full subcategory $\Fun_0( \Tw(\LinOrd), \cC) \rightarrow \Fun( \Tw(\LinOrd), \cC)$ of Notation \ref{notation.nought}.
\end{example}

We would like to consider an analogue of Example \ref{example.daycon-combinatorial}, where we replace the twisted arrow category
$\Tw(\LinOrd)$ with the category $\Pt(\broken)^{\op}$ whose objects are families of broken lines. Here we cannot apply
Theorem \ref{theorem.daycon} as stated, because the category $\Pt(\broken)^{\op}$ is not small. However, this is a minor technical nuisance:

\begin{lemma}\label{lemma.smallness}
Let $\cC$ be an $\infty$-category which admits small coproducts, and suppose we are given a pair of functors
$F_0, F_1: \Pt(\broken)^{\op} \rightarrow \cC$. Then the Day convolution product $F_0 \circledast F_1$ is well-defined.
That is, for any object of $\Pt(\broken)^{\op}$ given by a family of broken lines $L_{S} \to S$, the colimit
$$  \varinjlim_{ L_S \rightarrow L_{S_0} \star L_{S_1} } F_0( L_{S_0} ) \otimes_{\cC} F_1( L_{S_1} ).$$
exists in the $\infty$-category $\cC$.
\end{lemma}

\begin{proof}
The relevant colimit is indexed by the opposite of the category
$$ \mathcal{E} = (\Pt(\broken) \times \Pt(\broken) ) \times_{ \Pt(\broken) } \Pt(\broken)_{L_S / }$$
parametrizing pairs of families of broken lines $L_{S_0} \to S_0$, $L_{S_1} \to S_1$
together with a morphism
$$ \xymatrix{ L_S \ar[r] \ar[d] & L_{S_0} \star L_{S_1} \ar[d] \\ 
S \ar[r] & S_0 \times S_1. }$$
in the category $\Pt(\broken)$. Let $\mathcal{E}_0$ be the full subcategory of $\mathcal{E}$ spanned by those objects
for which the maps $S_0 \leftarrow S \rightarrow S_1$ are homeomorphisms. The inclusion $\mathcal{E}_0 \hookrightarrow \mathcal{E}$ has
a right adjoint, and is therefore right cofinal. It will therefore suffice to show that the colimit
$$  \varinjlim_{ (L_{S_0}, L_{S_1}) \in \mathcal{E}_0 } F_0( L_{S_0} ) \otimes_{\cC} F_1( L_{S_1} ).$$
exists in the category $\cC$. This follows from our assumption that $\cC$ admits small coproducts, since the category
$\mathcal{E}_0$ is equivalent to a (small) set, regarded as a category with only identity morphisms.
\end{proof}

Using Lemma \ref{lemma.smallness}, we obtain the following variant of Theorem \ref{theorem.daycon}:

\begin{theorem}\label{theorem.daycon-topological}
Let $\cC$ be a compactly generated monoidal $\infty$-category for which the tensor product $\otimes: \cC \times \cC \rightarrow \cC$ preserves small colimits
separately in each variable. Then there is a nonunital monoidal structure on the $\infty$-category $\Fun(\Pt(\broken)^{\op}, \cC)$ with the following properties:
\begin{itemize}
\item[$(a)$] The underlying tensor product $$\circledast: \Fun(\Pt(\broken)^{\op}, \cC) \times \Fun(\Pt(\broken)^{\op}, \cC) \rightarrow \Fun( \Pt(\broken^{\op}, \cC )$$ is given by the Day convolution product 
described in Lemma \ref{lemma.smallness}.
\item[$(b)$] There is a canonical equivalence of $\infty$-categories $$\Alg^{\nounit}( \Fun(\Pt(\broken)^{\op}, \cC) ) \simeq \Fun^{\lax}( \Pt(\broken)^{\op}, \cC).$$
\end{itemize}
\end{theorem}

\begin{proof}
Let $\cC_0$ be the full subcategory of $\cC$ spanned by the compact objects. Our assumption that $\cC$ is compactly generated guarantees that we can identify
$\cC$ with the $\infty$-category $\Ind( \cC_0) = \Fun^{\lex}( \cC_0^{\op}, \SSet )$ of left-exact $\SSet$-valued functors on $\cC_0^{\op}$ (see Notation~\ref{notation.kan-complexes}). Let
$\widehat{\SSet}$ denote the $\infty$-category of Kan complexes that are not necessarily small, and set $\widehat{\cC} = \Fun^{\lex}( \cC_0^{\op}, \widehat{\SSet} )$.
Then we can regard $\widehat{\cC}$ as a compactly generated $\infty$-category in a larger universe where the category $\Pt(\broken)^{\op}$ is small.
Applying Theorem \ref{theorem.daycon} in that context, we obtain a nonunital monoidal structure on the $\infty$-category $\Fun( \Pt(\broken)^{\op}, \widehat{\cC} )$ satisfying
the analogues of conditions $(a)$ and $(b)$. It follows from Lemma \ref{lemma.smallness} that the full subcategory $$\Fun( \Pt(\broken)^{\op}, \cC ) \subseteq \Fun( \Pt(\broken)^{\op}, \widehat{\cC} )$$
is closed under Day convolution, and therefore inherits a nonunital monoidal structure with the desired properties.
\end{proof}

\begin{remark}
Using a more refined analysis, one can prove Theorem \ref{theorem.daycon-topological} directly from the analysis of Lemma \ref{lemma.smallness}, without appealing to the notion of universes. We can also
weaken the assumption that $\cC$ is compactly generated: it is only necessary to assume that $\cC$ admits small coproducts. We leave the details to the reader.
\end{remark}

\begin{remark}
Let $\cC$ be a compactly generated $\infty$-category. The proof of Lemma \ref{lemma.smallness} shows that the Day convolution on
$\Fun( \Pt(\broken)^{\op}, \cC)$ can be described concretely by the formula
$$ (\sheafF' \circledast \sheafF'')(L_S) = \coprod_{ L_S \simeq L'_{S} \star_{S} L''_{S}} \sheafF'( L'_S ) \otimes \sheafF''( L''_S),$$
where the coproduct is taken over all (isomorphism classes of) decompositions $L_{S} \simeq L'_{S} \star_{S} L''_{S}$ in the
category $\broken(S)$ of broken lines over $S$; here $\star_{S}$ is defined as in Remark \ref{easystar}.
\end{remark}

\begin{remark}\label{remark.weak-factor}
Under the equivalence $$\Alg^{\nounit}( \Fun(\Pt(\broken)^{\op}, \cC) ) \simeq \Fun^{\lax}( \Pt(\broken)^{\op}, \cC)$$ of
Theorem \ref{theorem.daycon-topological}, the full subcategory $\Shv^{\wfact}_{\cC}( \broken ) \subseteq \Fun^{\lax}( \Pt(\broken)^{\op}, \cC )$
corresponds to the $\infty$-category of nonunital algebra objects of $\Fun(\Pt(\broken)^{\op}, \cC)$ which are
$\cC$-valued sheaves on $\broken$, in the sense of Definition \ref{definition.sheaf-on-stack}.
\end{remark}

In the situation of Theorem \ref{theorem.daycon-topological}, precomposition with the nonunital monoidal functor $\widetilde{\Phi}: \Tw(\LinOrd) \rightarrow \Pt(\broken)^{\op}$ of Proposition \ref{ols}
induces a nonunital lax monoidal functor
$$ \Fun( \Pt(\broken)^{\op}, \cC) \rightarrow \Fun( \Tw(\LinOrd), \cC ),$$
which carries the $\infty$-category $\Shv_{\cC}( \broken )$ to $\Fun_0( \Tw(\LinOrd), \cC )$. Using Remark \ref{remark.weak-factor} and Example \ref{example.daycon-combinatorial},
we can identify $\Shv^{\wfact}_{\cC}( \broken)$ and $\Fun^{\lax}_{0}( \Tw(\LinOrd), \cC )$ with the inverse images of $\Shv_{\cC}( \broken)$ and
$\Fun_0( \Tw(\LinOrd), \cC )$ in the $\infty$-categories of nonunital algebra objects of $\Fun( \Pt(\broken)^{\op}, \cC)$ and $\Fun( \Tw(\LinOrd), \cC )$, respectively.
Consequently, Theorem \ref{theorem.more-precise-version} is a consequence of the following assertion, which we will prove in our next and final section, \S \ref{section.bigproof}:

\begin{theorem}\label{theorem.even-more-precise}
Let $\cC$ be a compactly generated monoidal $\infty$-category, and assume that the tensor product $\otimes: \cC \times \cC \rightarrow \cC$ preserves small colimits
separately in each variable. Then composition with the functor $\widetilde{\Phi}: \Tw(\LinOrd) \rightarrow \Pt(\broken)^{\op}$ of Proposition \ref{ols}
induces an equivalence of $\infty$-categories $$\Shv_{\cC}( \broken ) \rightarrow \Fun_0( \Tw(\LinOrd), \cC ).$$ Moreover,
this functor $\theta$ is an equivalence of (nonunital) planar $\infty$-operads. In other words, for any collection of objects
$\sheafF_1, \sheafF_2, \ldots, \sheafF_n, \sheafG \in \Shv_{\cC}( \broken )$, the canonical map
$$ \xymatrix{ \bHom_{ \Fun( \Pt(\broken)^{\op}, \cC) }( \sheafF_1 \circledast \cdots \circledast \sheafF_n, \sheafG ) \ar[d]^{\psi} \\
\bHom_{ \Fun( \Tw(\LinOrd)^{\op}, \cC) }( (\sheafF_1 \circ \widetilde{\Phi}) \circledast \cdots \circledast (\sheafF_n \circ \widetilde{\Phi}), \sheafG \circ \widetilde{\Phi} ) }$$
is a homotopy equivalence.
\end{theorem}

\subsection{Proof of Theorem \ref{theorem.even-more-precise}}\label{section.bigproof}

We will deduce the first assertion of Theorem \ref{theorem.even-more-precise} by combining Theorem \ref{theorem.broken-sheaves} with the following:

\begin{proposition}\label{proposition.eval}
Let $\cC$ be a compactly generated $\infty$-category and let $\sheafF \in \Shv_{\cC}( \broken )$ be a $\cC$-valued sheaf on $\broken$.
Then, for every object $(I, \simeq_I ) \in \Tw( \LinOrd )$, the inclusion $\widetilde{ \Phi }( I, \simeq_I) \hookrightarrow \widetilde{\Rep}(I, \BR^{+} ) = \widetilde{\Phi}( I^{\sharp} )$
induces an equivalence $\sheafF( \widetilde{\Phi}(I^{\sharp} )) \rightarrow \sheafF( \widetilde{\Phi}(I, \simeq_I) )$.
\end{proposition}

\begin{proof}
Let $\sheafF_{ \widetilde{\Rep}(I, \BR^{+} )}$ denote the $\cC$-valued sheaf on the space $\Rep(I, \BR^{+} ) \simeq \broken^{I}$
determined by the family of broken lines $\widetilde{\Rep}(I, \BR^{+} ) \rightarrow \Rep(I, \BR^{+} )$. 
Let us regard the topological space $\Rep(I, \BR^{+} ) \simeq \broken^{I}$ as equipped with the stratification of Construction \ref{construction.stratify}, and let
$\calJ$ be the $\infty$-category of exit paths associated to that stratification (see the proof of Proposition \ref{proposition.constructible-characterization}). Note that the closed subset $\Phi(I, \simeq_I) \subseteq \Rep(I, \BR^{+} )$
is a union of strata (namely, those strata which correspond to convex equivalence relations $E$ which are contained in $\simeq_{I}$). Consequently,
$\Phi(I, \simeq_I)$ inherits a stratification whose exit path $\infty$-category can be identified with a full subcategory $\calJ_0 \subseteq \calJ$.
The sheaf $\sheafF$ is constructible, and can therefore be identified with a functor $F: \calJ \rightarrow \cC$ (see Theorem A.9.3 of~\cite{HA}). To prove Proposition \ref{proposition.eval},
it will suffice to show that the restriction map
$$ \varprojlim_{\alpha \in \calJ} F(\alpha) \rightarrow \varprojlim_{\alpha \in \calJ_0} F(\alpha)$$
is an equivalence in the $\infty$-category $\cC$. This is clear, since both sides are given by evaluating $F$ at the point $\alpha_0 \in \Rep(I, \BR^{+} )$ given
by the formula $$\alpha_0(i,j) = \begin{cases} 0 & \text{ if } i = j \\
\infty & \text{ if } i <_I j; \end{cases}$$ note that $\alpha_0$ is an initial object of the $\infty$-category $\calJ$ (this follows from the proof of 
Proposition \ref{proposition.constructible-characterization}), and is therefore also initial as an object of the subcategory $\calJ_0$.
\end{proof}

\begin{corollary}\label{corollary.plain-version}
Let $\cC$ be a compactly generated $\infty$-category. Then composition with the functor $\widetilde{\Phi}: \Tw(\LinOrd)^{\op} \rightarrow \Pt(\broken)$ of Construction \ref{construction.phi}
induces an equivalence of $\infty$-categories $\Shv_{\cC}(\broken) \rightarrow \Fun_0( \Tw(\LinOrd), \cC)$.
\end{corollary}

\begin{proof}
It follows from Proposition \ref{proposition.eval} that composition with the functor $\widetilde{\Phi}$ carries $\Shv_{\cC}( \broken)$ into the subcategory
$\Fun_0( \Tw(\LinOrd), \cC ) \subseteq \Fun( \Tw(\LinOrd), \cC )$. We conclude by observing that this composition functor fits into a commutative diagram
$$ \xymatrix{ & \Fun_0( \Tw(\LinOrd), \cC) \ar[dr] & \\
\Shv_{\cC}( \broken) \ar[ur] \ar[rr] & & \Fun( \LinOrd, \cC), }$$
where the right diagonal map is the equivalence of Notation \ref{notation.nought} and the horizontal map is
the equivalence of Theorem \ref{theorem.broken-sheaves}.
\end{proof}

\begin{proof}[Proof of Theorem \ref{theorem.even-more-precise}]
Let $\cC$ be a compactly generated monoidal $\infty$-category, and suppose that the tensor product $\otimes: \cC \times \cC \rightarrow \cC$ preserves small colimits separately in each variable.
By virtue of Corollary \ref{corollary.plain-version}, it will suffice to show that for every collection of $\cC$-valued sheaves $\sheafF_1, \ldots, \sheafF_n, \sheafG \in \Shv_{\cC}(\broken)$, 
the canonical map
$$ \xymatrix{ \bHom_{ \Fun( \Pt(\broken)^{\op}, \cC) }( \sheafF_1 \circledast \cdots \circledast \sheafF_n, \sheafG ) \ar[d]^{\psi} \\
\bHom_{ \Fun( \Tw(\LinOrd)^{\op}, \cC) }( (\sheafF_1 \circ \widetilde{\Phi}) \circledast \cdots \circledast (\sheafF_n \circ \widetilde{\Phi}), \sheafG \circ \widetilde{\Phi} ) }$$
is a homotopy equivalence.

Define functors $\overline{\sheafF}, \overline{\sheafG}: (\Pt(\broken)^{\op})^{n} \rightarrow \cC$ by the formulae
$$ \overline{\sheafF}(L_{S_1}, \ldots, L_{S_n} ) = \sheafF_1( L_{S_1}) \otimes \cdots \otimes \sheafF_n( L_{S_n} ) \quad
\overline{\sheafG}( L_{S_1}, \ldots, L_{S_n} ) = \sheafG(L_{S_1} \star \cdots \star L_{S_n} ).$$
Unwinding the definitions, we can identify $\psi$ with the restriction map
$$ \theta: \bHom_{ \Fun( (\Pt(\broken)^{\op})^{n}, \cC)}( \overline{\sheafF}, \overline{\sheafG} ) \rightarrow \bHom_{ \Fun( \Tw(\LinOrd)^{n}, \cC)}( \overline{\sheafF} \circ \widetilde{\Phi}^{n}, \overline{\sheafG} \circ \widetilde{\Phi}^{n} ).$$

For $0 \leq m \leq n$, let $\cJ^{(m)}$ denote the product $\Tw(\LinOrd)^{m} \times ( \Pt(\broken)^{\op} )^{n-m}$, and
let $\overline{\sheafF}^{(m)}, \overline{\sheafG}^{(m)}: \cJ^{(m)} \rightarrow \cC$ denote the functors given by the formulae
$$ \overline{\sheafF}^{(m)}( I_1, \ldots, I_m, L_{S_{m+1}}, \ldots, L_{S_n} ) = 
\overline{\sheafF}( \widetilde{\Phi}(I_1), \ldots, \widetilde{\Phi}(I_m), L_{S_{m+1}}, \ldots, L_{S_n})$$
$$ \overline{\sheafG}^{(m)}( I_1, \ldots, I_m, L_{S_{m+1}}, \ldots, L_{S_n} ) = 
\overline{\sheafG}( \widetilde{\Phi}(I_1), \ldots, \widetilde{\Phi}(I_m), L_{S_{m+1}}, \ldots, L_{S_n} ).$$
We will complete the proof by showing that each of the restriction maps
$$ \theta_m: \bHom_{ \Fun(\cJ^{(m-1)}, \cC)}( \overline{\sheafF}^{(m-1)}, \overline{\sheafG}^{(m-1)} ) \rightarrow \bHom_{ \Fun(\cJ^{(m)}, \cC)}( \overline{\sheafF}^{(m)}, \overline{\sheafG}^{(m)} ) $$
is a homotopy equivalence.

Let us henceforth regard the integer $m > 0$ as fixed, and let $\cK$ denote the product $\Tw(\LinOrd)^{m-1} \times (\Pt( \broken)^{\op})^{n-m}$. Let us identify
$\sheafF^{(m-1)}$ and $\sheafG^{(m-1)}$ with functors $F, G: \cK \rightarrow \Fun( \Pt(\broken)^{\op}, \cC)$, so that $\theta_m$ is given by the restriction map
$$ \bHom_{ \Fun( \cK, \Fun( \Pt(\broken)^{\op}, \cC )}( F, G) \rightarrow \bHom_{ \Fun( \cK, \Fun( \Tw(\LinOrd), \cC )}( F \circ \widetilde{\Phi}, G \circ \widetilde{\Phi} ).$$
To prove that this map is a homotopy equivalence, it will suffice to show that for every pair of objects $K, K' \in \cK$, the induced map
$$ \xymatrix{ \bHom_{ \Fun( \Pt(\broken)^{\op}, \cC)}( F(K), G(K') ) \ar[d]^{ \theta_{K,K'} } \\
 \bHom_{ \Fun( \cK, \Fun( \Tw(\LinOrd), \cC )}( F(K) \circ \widetilde{\Phi}, G(K') \circ \widetilde{\Phi}) }$$
is a homotopy equivalence.

To simplify the notation, we now assume that $1 < m < n$ (the extremal cases $m=1$ and $m=n$ are similar, but somewhat simpler).
Unwinding the definitions, we see that the functors $F(K), G(K'): \Pt(\broken)^{\op} \rightarrow \cC$ are given by the formulae
$$F(K)( L_S ) = C \otimes \sheafF_{m}( L_S ) \otimes C' \quad \quad G(K)( L_S ) = \sheafG( \widetilde{B} \star L_S \star \widetilde{B}' )$$
for some objects $C, C' \in \cC$, $\widetilde{B}, \widetilde{B}' \in \Pt(\broken)$. Note that $F(K) \circ \widetilde{\Phi}$ is a 
left Kan extension of its restriction along the functor $\sharp: \LinOrd \hookrightarrow \Tw(\LinOrd)$ of Notation
\ref{remark.left-adjoint-to-forget}, so that the restriction map
$$ \bHom_{ \Fun( \Tw(\LinOrd), \cC)}( F(K) \circ \widetilde{\Phi}, G(K') \circ \widetilde{\Phi} ) \rightarrow \bHom_{ \Fun( \LinOrd, \cC)}( F(K) \circ \widetilde{ \Fun}, G(K') \circ \widetilde{\Rep} )$$
is a homotopy equivalence; here $\widetilde{\Rep}$ denotes the functor $I \mapsto \widetilde{\Phi}( I^{\sharp} )$. We are therefore reduced to proving
that the composite map 
$$ \bHom_{ \Fun( \Pt(\broken)^{\op}, \cC)}( F(K), G(K') ) \rightarrow \bHom_{ \Fun( \LinOrd, \cC)}( F(K) \circ \widetilde{\Rep}, G(K') \circ \widetilde{\Rep} )$$
is a homotopy equivalence.

Let $\sheafH \in \Shv_{\cC}^{\lax}( \broken )$ denote the sheafification of the functor 
\begin{equation*}
F(K): \Pt(\broken)^{\op} \rightarrow \cC.
\end{equation*}
Since $\sheafF_m$ belongs to $\Shv_{\cC}(\broken)$, it follows from Remark \ref{remark.stays-a-sheaf} that $\sheafH$ also
belong to $\Shv_{\cC}( \broken )$. For every object $I \in \LinOrd$, let $L_{I}$ denote the broken line obtained by taking the fiber of the map $\widetilde{\Rep}( I, \BR^{+} ) \rightarrow \Rep(I, \BR^{+} )$
over the point $\alpha \in \Rep(I, \BR^{+} )$ given by $\alpha(i,j) = \begin{cases} 0 & \text{ if } i = j \\
\infty & \text{ if } i \neq j. \end{cases}$. 
Then we have a commutative diagram
$$ \xymatrix{ F(K)( \widetilde{\Rep}(I, \BR^{+} ) ) \ar[r] \ar[d] & F(K)( L_{I} ) \ar[d] \\
\sheafH( \widetilde{ \Rep}(I, \BR^{+} ) )\ar[r] & \sheafH(L_I). }$$
Here the right vertical map is an equivalence by construction (since the process of sheafification does not change the value of a presheaf on a single point),
and the horizontal maps are equivalences by virtue of Remark \ref{remark.constructible-evaluation}. It follows that the left vertical map is also an equivalence.
Allowing $I$ to vary, we conclude that the canonical map $F(K) \circ \widetilde{\Rep} \rightarrow \sheafH \circ \widetilde{\Rep}$ is an equivalence.
We can therefore identify $\theta_{K,K'}$ with the restriction map
$$ \bHom_{ \Fun( \Pt(\broken)^{\op}, \cC)}( \sheafH, G(K') ) \rightarrow \bHom_{ \Fun( \LinOrd, \cC)}( \sheafH \circ \widetilde{\Rep}, G(K') \circ \widetilde{\Rep} ).$$

Note that the formula $G(K')( L_S ) = \sheafG( \widetilde{B} \star L_S \star \widetilde{B}' )$ immediately shows that $G(K')$ is
a lax $\cC$-valued sheaf on $\broken$. To complete the proof that $\theta_{K,K'}$ is a homotopy equivalence, it will suffice to show that
$G(K')$ is homotopy invariant (see Proposition \ref{proposition.homotopy-invariant-is-enough}). Let $L_S \rightarrow S$ be a family of broken lines
parametrized by a topological space $S$ and set $L_{S \times \RR} = L_S \times \RR$, which we regard as a family of broken lines over $S \times \RR$.
We wish to show that the canonical map
$$ G(K')( L_S ) = \sheafG( \widetilde{B} \star L_S \star \widetilde{B}' ) \rightarrow
\sheafG( \widetilde{B} \star L_{S \times \RR} \star \widetilde{B}') = G(K')( L_{S \times \RR})$$
is an equivalence in $\cC$. This follows from the homotopy invariance of the presheaf $\sheafG \in \Shv_{\cC}(\broken)$ (see Proposition \ref{proposition.automatic-homotopy-invariance}).
\end{proof}


\begin{thebibliography}{99}

\bibitem{boardman-vogt} J. M. Boardman and R. M. Vogt, {\it Homotopy invariant algebraic structures on topological spaces}, Lecture Notes in Mathematics, Vol. 347, Springer-Verlag, Berlin, 1973.
\bibitem{joyal}  Joyal, Andr\'e, {\em Notes on quasicategories}, available at \url{http://www.math.uchicago.edu/~may/IMA/Joyal.pdf}, June 2008.
\bibitem{HA} Lurie, Jacob. {\it Higher Algebra}, available at \url{http://www.math.harvard.edu/~lurie/}, updated September 2017.
\bibitem{htt} Lurie, Jacob. {\it Higher Topos Theory}, available at \url{http://www.math.harvard.edu/~lurie/}, updated April 2017. 
\end{thebibliography}
\end{document}